\pdfoutput=1
\documentclass[reqno]{amsart}
\usepackage[utf8]{inputenc}

\usepackage[margin=1in]{geometry}
\usepackage[expansion=false]{microtype}

\usepackage{amsmath, amsfonts, amssymb, amsthm, amsopn, mathtools}
\usepackage{eucal}

\usepackage{tikz-cd}
\usetikzlibrary{patterns}
\usepackage{bm}
\usepackage{rotating}
\usepackage{xparse}
\usepackage{enumitem}

\usepackage[pdfusetitle,unicode,colorlinks]{hyperref}

\numberwithin{equation}{subsection}

\theoremstyle{plain}
\newtheorem{theorem}[equation]{Theorem}
\newtheorem{proposition}[equation]{Proposition}

\newtheorem{lemma}[equation]{Lemma}
\newtheorem{corollary}[equation]{Corollary}

\makeatletter
\newtheorem*{rep@theorem}{\rep@title}
\newcommand{\newreptheorem}[2]{%
\newenvironment{rep#1}[1]{%
 \def\rep@title{#2 \ref{##1}}%
 \begin{rep@theorem}}%
 {\end{rep@theorem}}}
\makeatother
\newreptheorem{theorem}{Theorem}

\theoremstyle{definition}
\newtheorem{definition}[equation]{Definition}

\newtheorem{remark}[equation]{Remark}

\let\scr=\mathcal

\let\phi=\varphi

\let\into=\hookrightarrow
\let\onto=\twoheadrightarrow

\def\AA{\scr A}
\def\BB{\scr B}
\def\CC{\scr C}

\def\LL{\scr L}

\def\RR{\scr R}
\def\SS{\scr S}

\def\AAA{\widehat{\AA}}
\def\BBB{\widehat{\BB}}

\def\SSS{\widehat{\SS}}

\def\SSSS{\hathat{\SS}}
\def\BBBB{\hathat{\BB}}

\def\bU{\mathbf{U}}
\def\bV{\mathbf{V}}
\def\bW{\mathbf{W}}

\DeclareMathOperator{\id}{id}
\DeclareMathOperator{\ev}{ev}

\DeclareMathOperator{\PSh}{PSh}
\DeclareMathOperator{\IPSh}{\mathsf{PSh}}
\DeclareMathOperator{\Cat}{Cat}
\DeclareMathOperator{\ICat}{\mathsf{Cat}}
\DeclareMathOperator{\Shv}{Sh}
\DeclareMathOperator{\Cart}{Cart}
\DeclareMathOperator{\Cocart}{Cocart}
\DeclareMathOperator{\ICocart}{\mathsf{Cocart}}

\DeclareMathOperator{\ICart}{\mathsf{Cart}}
\DeclareMathOperator{\RFib}{RFib}
\DeclareMathOperator{\LFib}{LFib}
\DeclareMathOperator{\ILFib}{\mathsf{LFib}}
\DeclareMathOperator{\IRFib}{\mathsf{RFib}}

\DeclareMathOperator{\LTop}{LTop}

\DeclareMathOperator{\Set}{Set}
\DeclareMathOperator{\Sub}{Sub}
\DeclareMathOperator{\Tw}{Tw}
\DeclareMathOperator{\Fun}{Fun}
\DeclareMathOperator{\Map}{map}

\DeclareMathOperator{\Grpd}{Grpd}

\DeclareMathOperator{\const}{const}
\DeclareMathOperator{\diag}{diag}

\DeclareMathOperator{\pr}{pr}
\DeclareMathOperator{\res}{res}
\DeclareMathOperator{\lift}{lift}
\DeclareMathOperator{\Prop}{Prp}

\DeclareMathOperator{\IProp}{\mathsf{Prp}}
\DeclareMathOperator{\mSmooth}{\mathsf{Sm}}
\DeclareMathOperator{\St}{St}
\DeclareMathOperator{\Un}{Un}
\DeclareMathOperator{\Nec}{Nec}
\DeclareMathOperator{\Sk}{Sk}

\newcommand{\op}{\mathrm{op}}
\newcommand{\core}{\simeq}
\newcommand{\gp}{\mathrm{gpd}}
\newcommand{\ord}[1]{\langle#1\rangle}
\newcommand{\map}[1]{\Map_{#1}}

\newcommand{\Over}[2]{#1_{\hspace{-1pt}/#2}}
\newcommand{\Under}[2]{#1_{\hspace{-1pt}#2/}}
\newcommand{\sslash}{\mathbin{/\mkern-6mu/}}
\newcommand{\I}[1]{\mathsf{#1}}
\renewcommand{\smallint}{\textstyle\int}
\newcommand{\iFun}[2]{{[#1,#2]}}
\newcommand{\Comma}[3]{{#1}\downarrow_{#2}{#3}}

\newcommand{\Simp}[1]{#1_{\Delta}}
\newcommand{\mSimp}[1]{#1_{\Delta}^+}
\newcommand{\LaxUnder}[2]{#1_{\hspace{-1pt}#2\sslash}}
\newcommand{\CatS}{\Cat_{\infty}}
\newcommand{\CatSS}{\widehat{\Cat}_\infty}
\newcommand{\cc}{\text{\normalfont{cc}}}
\newcommand{\cocont}[1]{{#1\text{\normalfont{-cc}}}}

\NewDocumentCommand{\Gen}{m o}{%
	\IfNoValueTF{#2}{%
		\langle #1\rangle%
	}{%
		\langle #1\rangle_{#2}%
	}%
}

\NewDocumentCommand{\Univ}{o}{%
	\IfNoValueTF{#1}{%
		\I{\Omega}%
	}{%
		\I{\Omega}_{#1}%
	}%
}
\NewDocumentCommand{\UnivHat}{o}{%
	\IfNoValueTF{#1}{%
		\widehat{\I{\Omega}}%
	}{%
		\widehat{\I{\Omega}}_{#1}%
	}%
}

\let\lim=\relax
\DeclareMathOperator*{\lim}{lim}
\DeclareMathOperator*{\colim}{colim}

\makeatletter
\g@addto@macro\bfseries{\boldmath}
\makeatother

\makeatletter
\newcommand{\hathatInternal}[2]{%
	\begingroup%
	\let\macc@kerna\z@%
	\let\macc@kernb\z@%
	\let\macc@nucleus\@empty%
	\widehat{\raisebox{#2}{\vphantom{\ensuremath{#1}}}\smash{\widehat{#1}}}%
	\endgroup%
}
\makeatother
\newcommand{\hathat}[1]{\mathchoice
	{\hathatInternal{#1}{.2ex}}
	{\hathatInternal{#1}{.2ex}}
	{\hathatInternal{#1}{-1.5pt}}
	{\hathatInternal{#1}{1pt}}
}

\title{Cocartesian fibrations and straightening internal to an $\infty$-topos}
\author{Louis Martini}
\address{Norwegian University of Science and Technology (NTNU)\\
Alfred Getz' vei 1\\
7034 Trondheim\\
Norway}
\email{\href{mailto:louis.o.martini@ntnu.no}{louis.o.martini@ntnu.no}}
\date{\today}
\begin{document}
\begin{abstract}
We define and study cartesian and cocartesian fibrations between categories internal to an $\infty$-topos and prove a straightening equivalence in this context.
\end{abstract}
\maketitle
\setcounter{tocdepth}{1}
\tableofcontents

\section{Introduction}

\subsection*{Motivation}
One of the fundamental results in higher category theory is Lurie's \emph{straightening theorem}~\cite{htt}, which provides an equivalence $\Fun(\CC,\CatS)\simeq\Cocart(\CC)$ between the $\infty$-category of $\CatS$-valued functors on an $\infty$-category $\CC$ and that of \emph{cocartesian fibrations} over $\CC$. This theorem generalises Grothendieck's classical result on the equivalence between pseudo-functors from a $1$-category into the $2$-category of $1$-categories and Grothendieck opfibrations over that $1$-category. As it is notoriously challenging to directly construct a functor $\CC\to\CatS$ due to the infinite tower of coherence conditions, the straightening theorem provides an indispensable tool for the study of such functors. It is therefore not surprising that this result is (sometimes implicitly) present throughout higher category theory. For example, it has been used by Lurie to define and study adjunctions between $\infty$-categories, limits and colimits of $\infty$-categories as well as to develop the theory of $\infty$-operads and monoidal $\infty$-categories~\cite{htt, Lurie2017}.

This paper is the third in a series in which we aim to develop the theory of higher categories \emph{internal} to an arbitrary $\infty$-topos $\BB$. An internal higher category can be defined as a certain simplicial object in $\BB$ that satisfies the Segal conditions and univalence. Equivalently, one can view these objects as sheaves of $\infty$-categories on $\BB$.  As such, they arise in various different contexts, for example in the form of (higher) categorical invariants in algebraic geometry or topology such as the (un)stable motivic homotopy $\infty$-category of a scheme. From a different point of view, many relative constructions in higher category theory can be realised as internal categories. For example, presentable $\BB$-modules give rise to presentable categories internal to $\BB$, and $\infty$-topoi over $\BB$ can be realised as internal topoi in $\BB$. One can therefore employ techniques from internal higher category theory to study these objects.

 In~\cite{Martini2021}, we have already set up the basic language of higher categories internal to an $\infty$-topos $\BB$, referred to hereafter as \emph{$\BB$-categories}. Notably, we defined left and right fibrations between $\BB$-categories and proved that these are in correspondence with internal functors into the \emph{universe} $\Univ$, which is the internal analogue of the $\infty$-category of spaces. We used this result to prove Yoneda's lemma for $\BB$-categories. In joint work with Sebastian Wolf, we continued our study of internal higher categories in~\cite{Martini2021a}, where we developed a few basic tools such as the theory of adjunctions, limits and colimits as well as Kan extensions for $\BB$-categories. We also constructed the (large) $\BB$-category $\ICat_{\BB}$ of small $\BB$-categories. 
 
 The main goal of this text is to establish a straightening theorem in the context of higher categories \emph{internal} to an arbitrary $\infty$-topos $\BB$. Our central result can be formulated as follows:
\begin{reptheorem}{thm:StraighteningEquivalence}
	For every $\BB$-category $\I{C}$, there is an adjoint equivalence
	\begin{equation*}
	(\Un_{\I{C}}\dashv\St_{\I{C}})\colon \ICocart_{\I{C}}\simeq\iFun{\I{C}}{\ICat_{\BB}}
	\end{equation*}
	between the $\BB$-category of cocartesian fibrations over $\I{C}$ and the $\BB$-category of $\ICat_{\BB}$-valued functors on $\I{C}$.
\end{reptheorem}
Our strategy for the proof of Theorem~\ref{thm:StraighteningEquivalence} is to build upon the equivalence $\ILFib_{\I{C}}\simeq\iFun{\I{C}}{\Univ}$ between left fibrations over $\I{C}$ and $\Univ$-valued functors that we established in~\cite{Martini2021}, a strategy that has previously been outlined by Vladimir Hinich~\cite{Hinich2017} as a proof of the straightening equivalence for $\infty$-categories. By definition of the $\BB$-category $\ICat_{\BB}$, the functor $\BB$-category $\iFun{\I{C}}{\ICat_{\BB}}$ is a full subcategory of the $\BB$-category $\iFun{\Delta^{\op}}{\iFun{\I{C}}{\Univ}}$ of simplicial objects in $\iFun{\I{C}}{\Univ}$. Therefore, one may idenfity functors $\I{C}\to\ICat_{\BB}$ with certain simplicial objects in the $\BB$-category $\ILFib_{\I{C}}$ of left fibrations over $\I{C}$. We will show that cocartesian fibrations over $\I{C}$ are \emph{powered} over $\Delta$, i.e.\ that there is a functor $(-)^{\Delta^{\bullet}}\colon \Delta^{\op}\times\ICocart_{\I{C}}\to\ICocart_{\I{C}}$. Moreover, there is an inclusion $\ILFib_{\I{C}}\into\ICocart_{\I{C}}$ that admits a right adjoint $(-)_\sharp$ which carries a cocartesian fibration $\I{P}\to\I{C}$ to the underlying left fibration that is spanned by the cocartesian morphisms in $\I{P}$. This allows us to explicitly define the straightening functor $\St_{\I{C}}=(-)^{\Delta^{\bullet}}_\sharp\colon \ICocart_{\I{C}}\to\iFun{\Delta^{\op}}{\ILFib_{\I{C}}}$. Conversely, the $\BB$-category of cocartesian fibrations over $\I{C}$ is also \emph{tensored} over $\Delta$ in the form of a functor $\Delta^{\bullet}\otimes -\colon \Delta\times\ICocart_{\I{C}}\to\ICocart_{\I{C}}$. By making use of the universal property of presheaf $\BB$-categories that was established in~\cite{Martini2021a}, we may therefore define the \emph{unstraightening} functor $\Un_{\I{C}}\colon \iFun{\Delta^{\op}}{\iFun{\I{C}}{\Univ}}\simeq\IPSh_{\Univ}(\Delta\times\I{C}^{\op})\to\ICocart_{\I{C}}$ as the left Kan extension of the functor $\Delta^{\bullet}\otimes\Under{\I{C}}{-}\colon\Delta^{\op}\times\I{C}^{\op}\to\ICocart_{\I{C}}$ along the Yoneda embedding $h_{\Delta\times\I{C}^{\op}}$. By construction, the unstraightening functor is left adjoint to the straightening functor. We will complete our argument by showing that this adjunction is natural in $\I{C}$ in the appropriate sense, so that we can reduce to the case $\I{C}=1$, in which case the desired result follows trivially.

In the special case where $\BB$ is the $\infty$-topos of spaces, Theorem~\ref{thm:StraighteningEquivalence} recovers Lurie's straightening equivalence. One can therefore regard our proof of Theorem~\ref{thm:StraighteningEquivalence} as another approach to the straightening equivalence, complementing existing proofs such as Lurie's original account in~\cite{htt} and the more recent approaches by Boavida de Brito~\cite{BoavidadeBrito2018}, Nuiten~\cite{Nuiten2021}, Hebestreit-Heuts-Ruit~\cite{Hebestreit2021} and Rasekh~\cite{Rasekh2021b}.

The straightening equivalence for $\BB$-categories provides a versatile tool for the study of $\ICat_{\BB}$-valued functors in internal higher category theory. For example, the fibrational point of view can be used to characterise adjunctions between $\BB$-categories and to develop formulas for the (co)limit of a diagram in $\ICat_{\BB}$. We will discuss both applications in the last chapter of this paper. Further down the line, the straightening equivalence will become indispensable for the study of internal topos theory: in order to define the notion of \emph{descent} for $\BB$-topoi, it will be crucial to be able to define the internal presheaf of $\BB$-categories that classifies the codomain fibration over a given $\BB$-category. The study of internal topos theory, however, is beyond the scope of this paper and will be the content of upcoming work.

\subsection*{Structure of the paper}
The first part of this paper is devoted to the definition and study of cocartesian fibrations between $\BB$-categories. In \S~\ref{sec:CocartesianFibrations}, we define such cocartesian fibrations via an internal analogue of what is sometimes known as the \emph{Chevalley criterion} in (higher) category theory (see for example~\cite[Theorem~5.2.8]{Riehl2022}). Moreover, we study the concept of cocartesian morphisms in this context and show that a cocartesian fibration can be characterised by the existence of a sufficient amount of such cocartesian morphisms in the domain.

In \S~\ref{sec:markedModel}, we establish an internal analogue of Lurie's \emph{marked model stucture} for cocartesian fibrations. We define the marked simplex category $\Delta_+$ and study the $\infty$-topos $\mSimp\BB$ of marked simplicial objects in $\BB$, i.e.\ of $\BB$-valued presheaves on $\Delta_+$ . The benefit of passing to marked simplicial objects is that cocartesian fibrations are determined by a factorisation system in $\mSimp\BB$, which enables us to make use of the many desirable properties of factorisation systems in an $\infty$-topos to deepen our study of cocartesian fibrations. This already comes in handy when we define and study the $\BB$-category $\ICocart_{\I{C}}$ of cocartesian fibrations over a $\BB$-category $\I{C}$ in \S~\ref{sec:CategoryOfCocartesianFibrations}.

In \S~\ref{sec:SU}, we set up and discuss the straightening and unstraightening functors, which culminates in the proof of our main theorem. We complement this result with a study of the \emph{universal} cocartesian fibration, which helps us understand how the internal straightening of a cocartesian fibration is related to the straightening of the underlying cocartesian fibration of $\infty$-categories that is obtained by passing to global sections. Lastly, we investigate the special case of cocartesian fibrations over the interval $\Delta^1$ and how these can be used to characterise adjunctions between $\BB$-categories.

We conclude this paper by briefly mentioning two applications of the straightening equivalence in~\S~\ref{sec:app}. The first application gives a formula for the limit and colimit of $\ICat_{\BB}$-valued diagrams in terms of the associated cocartesian fibrations. In the second application, we use our knowledge of cocartesian fibrations over the interval to establish that passing from a right adjoint functor to its left adjoint (and vice versa) constitutes an equivalence between the $\BB$-category of $\BB$-categories with right adjoint functors and that of $\BB$-categories with left adjoint functors.

\subsection*{Related work}
Our strategy for the proof of the straightening equivalence is an adaptation of a proof that was sketched by Hinich in his lecture notes on $\infty$-categories~\cite{Hinich2017}. Notably, we adopted his construction of the straightening and unstraightening functors. However, Hinich did not provide details for his claim that this construction commutes with base change. A substantial portion of the present paper is devoted to a proof of this claim.

We have already mentioned above that by now there exist several proofs for the $\infty$-categorical straightening equivalence~\cite{htt, BoavidadeBrito2018, Nuiten2021, Hebestreit2021, Rasekh2021b}. In~\cite{Shah2018}, Jay Shah builds upon this result to derive a straightening equivalence for \emph{parametrised} higher categories. In our language, this corresponds to straightening internal to presheaf $\infty$-topoi.

Buchholtz and Weinberger~\cite{Buchholtz2021} developed the theory of cocartesian fibrations in the framework of synthetic $\infty$-category theory, which is an extension of homotopy type theory that makes it possible to study higher categories from a type theoretic point of view. As homotopy type theory admits semantics in arbitrary $\infty$-topoi, their results recover most of what is covered in \S~\ref{sec:CocartesianFibrations}.

\subsection*{Acknowledgments}
I would like express my gratitude to Rune Haugseng for his advice and support and for countless helpful discussions throughout the process of writing this paper. I also thank Sebastian Wolf for sharing his valuable insight during many discussions on the subject. Finally, I thank Bastiaan Cnossen for his comments on an earlier version of this manuscript.

\section{Preliminaries}

\subsection{General conventions and notation}
We generally follow the conventions and notation from~\cite{Martini2021} and~\cite{Martini2021a}. For the convenience of the reader, we will briefly recall the main setup. 

Throughout this paper we freely make use of the language of higher category theory. We will generally follow a model-independent approach to higher categories. This means that as a general rule, all statements and constructions that are considered herein will be invariant under equivalences in the ambient $\infty$-category, and we will always be working within such an ambient $\infty$-category.

We denote by $\Delta$ the simplex category, i.e.\ the category of non-empty totally ordered finite sets with order-preserving maps. Every natural number $n\in\mathbb N$ can be considered as an object in $\Delta$ by identifying $n$ with the totally ordered set $\ord{n}=\{0,\dots n\}$. For $i=0,\dots,n$ we denote by $\delta^i\colon \ord{n-1}\to \ord{n}$ the unique injective map in $\Delta$ whose image does not contain $i$. Dually, for $i=0,\dots n$ we denote by $\sigma^i\colon \ord{n+1}\to \ord{n}$ the unique surjective map in $\Delta$ such that the preimage of $i$ contains two elements. Furthermore, if $S\subset n$ is an arbitrary subset of $k$ elements, we denote by $\delta^S\colon \ord{k}\to \ord{n}$ the unique injective map in $\Delta$ whose image is precisely $S$. In the case that $S$ is an interval, we will denote by $\sigma^S\colon \ord{n}\to \ord{n-k}$ the unique surjective map that sends $S$ to a single object. If $\CC$ is an $\infty$-category, we refer to a functor $C\colon\Delta^{\op}\to\CC$ as a simplicial object in $\CC$. We write $C_n$ for the image of $n\in\Delta$ under this functor, and we write $d_i$, $s_i$, $d_S$ and $s_S$ for the image of the maps $\delta^i$, $\sigma^i$, $\delta^S$ and $\sigma^S$ under this functor. Dually, a functor $C^{\bullet}\colon \Delta\to\CC$ is referred to as a cosimplicial object in $\CC$. In this case we denote the image of $\delta^i$, $\sigma^i$, $\delta^S$ and $\sigma^S$ by $d^i$, $s^i$, $d^S$ and $\sigma^S$.

The $1$-category $\Delta$ embeds fully faithfully into the $\infty$-category of $\infty$-categories by means of identifying posets with $0$-categories and order-preserving maps between posets with functors between such $0$-categories. We denote by $\Delta^n$ the image of $n\in\Delta$ under this embedding.

\subsection{Set-theoretical foundations}
Once and for all we will fix three Grothendieck universes $\bU\in\bV\in\bW$ that contain the first infinite ordinal $\omega$. A set is \emph{small} if it is contained in $\bU$, \emph{large} if it is contained in $\bV$ and \emph{very large} if it is contained in $\bW$. An analogous naming convention will be adopted for $\infty$-categories and $\infty$-groupoids. The large $\infty$-category of small $\infty$-groupoids is denoted by $\SS$, and the very large $\infty$-category of large $\infty$-groupoids by $\SSS$. The (even larger) $\infty$-category of very large $\infty$-groupoids will be denoted by $\SSSS$. Similarly, we denote the large $\infty$-category of small $\infty$-categories by $\CatS$ and the very large $\infty$-category of large $\infty$-categories by $\CatSS$. We shall not need the $\infty$-category of very large $\infty$-categories in this article.

\subsection{$\infty$-topoi}
For $\infty$-topoi $\AA$ and $\BB$, a \emph{geometric morphism} is a functor $f_\ast\colon \BB\to \AA$ that admits a left exact left adjoint, and an \emph{algebraic morphism} is a left exact functor $f^\ast\colon \AA\to \BB$ that admits a right adjoint. The \emph{global sections} functor is the unique geometric morphism $\Gamma_{\BB}\colon \BB\to \SS$ into the $\infty$-topos of $\infty$-groupoids $\SS$. Dually, the unique algebraic morphism originating from $\SS$ is denoted by $\const_{\BB}\colon \SS\to \BB$ and referred to as the \emph{constant sheaf} functor. We will often omit the subscripts if they can be inferred from the context.
For an object $A \in \BB$, we denote the induced étale geometric morphism by $(\pi_A)_\ast \colon \BB_{/A} \rightarrow \BB$.

\subsection{Universe enlargement} 
If $\BB$ is an $\infty$-topos, we define its \emph{universe enlargement} $\BBB=\Shv_{\SSS}(\BB)$ as the $\infty$-category of $\SSS$-valued sheaves on $\BB$, i.e.\ of those functors $\BB^{\op}\to \SSS$ that preserve small limits; this is an $\infty$-topos relative to the larger universe $\bV$. Moreover, the Yoneda embedding gives rise to an inclusion $\BB\into\BBB$ that commutes with small limits and colimits and with taking internal mapping objects. The operation of enlarging universes is transitive: when defining the $\infty$-topos $\BBBB$ relative to $\bW$ as the universe enlargement of $\BBB$ with respect to the inclusion $\bV\in\bW$, the $\infty$-category $\BBBB$ is equivalent to the universe enlargement of $\BB$ with respect to $\bU\in\bW$.

\subsection{Factorisation systems} 
If $\CC$ is a presentable $\infty$-category and if $S$ is a small set of maps in $\CC$, there is a unique factorisation system $(\LL,\RR)$ in which a map is contained in $\RR$ if and only if it is \emph{right orthogonal} to the maps in $S$, and where $\LL$ is dually defined as the set of maps that are left orthogonal to the maps in $\RR$. We refer to $\LL$ as the \emph{saturation} of $S$; this is the smallest set of maps containing $S$ that is stable under pushouts, contains all equivalences and is stable under small colimits in $\Fun(\Delta^1,\CC)$. An object $c\in\CC$ is said to be \emph{$S$-local} if the unique morphism $c\to 1$ is contained in $\RR$. 
    
If $\CC$ is cartesian closed, one can analogously construct a factorisation system $(\LL^\prime,\RR^\prime)$ in which $\RR^\prime$ is the set of maps in $\BB$ that are \emph{internally} right orthogonal to the maps in $S$. Explicitly, a map is contained in $\RR^\prime$ if and only if it is right orthogonal to maps of the form $s\times \id_c$ for any $s\in S$ and any $c\in \CC$. The left complement $\LL^\prime$ consists of those maps in $\CC$ that are left orthogonal to the maps in $\RR^\prime$ and is referred to as the \emph{internal} saturation of $S$. Equivalently, $\LL^\prime$ is the saturation of the set of maps $s\times\id_c$ for $s\in S$ and $c\in\CC$. An object $c\in\CC$ is said to be \emph{internally $S$-local} if the unique morphism $c\to 1$ is contained in $\RR^\prime$. 
    
Given any factorisation system $(\LL,\RR)$ in $\CC$ in which $\LL$ is the saturation of a small set of maps in $\CC$, the inclusion $\RR\into\Fun(\Delta^1,\CC)$ admits a left adjoint that carries a map $f\in\Fun(\Delta^1,\CC)$ to the map $r\in\RR$ that arises from the unique factorisation $f\simeq rl$ into maps $l\in \LL$ and $r\in \RR$. By taking fibres over an object $c\in\CC$, one furthermore obtains a reflective subcategory $\Over{\RR}{c}\leftrightarrows\Over{\CC}{c}$ such that if $f\colon d\to c$ is an object in $\Over{\CC}{c}$ and if $f\simeq rl$ is its unique factorisation into maps $l\in \LL$ and $r\in \RR$, the adjunction unit is given by $l$.

\subsection{Recollection on $\BB$-categories}
\label{sec:recollection}

In this section we recall the basic framework of higher category theory internal to an $\infty$-topos from~\cite{Martini2021} and~\cite{Martini2021a}.
\begin{description}
    \item[Simplicial objects] If $\BB$ is an $\infty$-topos, we denote by $\Simp\BB=\Fun(\Delta^{\op},\BB)$ the $\infty$-topos of \emph{simplicial objects} in $\BB$. By precomposition with the global sections and the constant sheaf functor, one obtains an adjunction $\const_{\BB}\dashv \Gamma_{\BB}\colon \Simp\BB\leftrightarrows\Simp\SS$. We will often implicitly identify a simplicial $\infty$-groupoid $K$ with its image in $\Simp\BB$ along $\const_{\BB}$.
    
    \item[$\BB$-categories] A \emph{$\BB$-category} is a simplicial object $\I{C}\in\Simp\BB$ that is internally local with respect to $I^2\into\Delta^2$ (Segal conditions) and $E^1\to 1$ (univalence). Here $I^2=\Delta^1\sqcup_{\Delta^0}\Delta^1\into\Delta^2$ is the inclusion of the $2$-spine, and $E^1=\Delta^3\sqcup_{\Delta^1\sqcup\Delta^1}(\Delta^0\sqcup\Delta^0)$ is the walking equivalence. One obtains a reflective subcategory $\Cat(\BB)\into\Simp\BB$ in which the left adjoint commutes with finite products. In particular, $\Cat(\BB)$ is cartesian closed, and we denote by $\iFun{-}{-}$ the internal mapping bifunctor. We refer to the maps in $\Cat(\BB)$ as \emph{functors} between $\BB$-categories.
    \item[$\BB$-groupoids] A \emph{$\BB$-groupoid} is a simplicial object $\I{G}\in\Simp\BB$ that is internally local with respect to $s^0\colon\Delta^1\to\Delta^0$. A simplicial object in $\BB$ is a $\BB$-groupoid if and only if it is contained in the essential image of the diagonal embedding $\BB\into \Simp\BB$. Every $\BB$-groupoid is a $\BB$-category. Therefore, one obtains a full subcategory $\BB\simeq \Grpd(\BB)\into\Cat(\BB)$ that admits both a left adjoint $(-)^{\gp}$ and a right adjoint $(-)^\core$. We refer to the left adjoint as the \emph{groupoidification functor} and to the right adjoint as the \emph{core $\BB$-groupoid functor}. Explicitly, if $\I{C}$ is a $\BB$-category, one has $\I{C}^\gp \simeq \colim_{\Delta^{\op}}\I{C}$ and $\I{C}^{\simeq}\simeq \I{C}_0$.
    
    \item[Base change] If $f_\ast\colon \BB\to\AA$ is a geometric morphism and if $f^\ast$ is the associated algebraic morphism, postcomposition induces an adjunction $f_\ast\dashv f^\ast\colon \Cat(\BB)\leftrightarrows\Cat(\AA)$. If $f_\ast$ is furthermore \emph{\'etale}, the further left adjoint $f_!$ also induces a functor $f_!\colon \Cat(\BB)\to\Cat(\AA)$ that identifies $\Cat(\BB)$ with $\Over{\Cat(\AA)}{f_! 1}$. In particular, one obtains an adjunction $\const_{\BB}\dashv \Gamma_{\BB}\colon \Cat(\BB)\leftrightarrows \CatS$. We will often implicitly identify an $\infty$-category $\CC$ with the associated \emph{constant $\BB$-category} $\const_{\BB}(\CC)\in\Cat(\BB)$.
    
    \item[Tensoring and powering] One defines bifunctors
    \begin{align}
        \Fun_{\BB}(-,-)=\Gamma_{\BB}\circ\iFun{-}{-} \tag{Functor $\infty$-category} \\
        (-)^{(-)}=\iFun{\const_{\BB}(-)}{-} \tag{Powering}\\
        -\otimes - = \const_{\BB}(-)\times - \tag{Tensoring}
    \end{align}
    which fit into equivalences
    \begin{equation*}
        \map{\Cat(\BB)}(-\otimes -, -)\simeq \map{\CatS}(-,\Fun_{\BB}(-,-))\simeq\map{\Cat(\BB)}(-, (-)^{(-)}).
    \end{equation*}
    There is moreover an equivalence of functors $\id_{\Cat(\BB)}\simeq ((-)^{\Delta^\bullet})^\simeq$. In other words, for any $\BB$-category $\I{C}$ and any integer $n\geq 0$ one may canonically identify $\I{C}_n\simeq (\I{C}^{\Delta^n})_0$.
    
    \item[Sheaves of $\infty$-categories] $\BB$-categories are equivalently given by $\CatS$-valued sheaves on $\BB$: There is a canonical equivalence $\Cat(\BB)\simeq \Shv_{\CatS}(\BB)$ that is natural in $\BB$. Explicitly, this equivalence sends $\I{C}\in\Cat(\BB)$ to the sheaf $\Fun_{\BB}(-,\I{C})$ on $\BB$. We will often implicitly identify $\BB$-categories with their associated $\CatS$-valued sheaves on $\BB$. For example, if $\I{C}$ is a $\BB$-category, we will write $\I{C}(A)=\Fun_{\BB}(A,\I{C})$ for the $\infty$-category of local sections over $A\in\BB$. If $s \colon A \rightarrow B$ is a morphism in $\BB$, we write $s^\ast\colon \I{C}(B)\to\I{C}(A)$ for the associated map in $\CatS$.
    
    \item[Large $\BB$-categories] Postcomposition with the universe enlargement $\BB\into\BBB$ determines an inclusion $\Cat(\BB)\into\Cat(\BBB)$ that corresponds to the inclusion $\Shv_{\CatS}(\BB)\into\Shv_{\CatSS}(\BB)$ on the level of sheaves on $\BB$. Either inclusion is furthermore natural in $\BB$. We refer to the objects in $\Cat(\BBB)$ as \emph{large} $\BB$-categories (or as $\BBB$-categories) and to the objects in $\Cat(\BB)$ as \emph{small} $\BB$-categories. If not specified otherwise, every $\BB$-category is small. Note, however, that by replacing the universe $\bU$ with the larger universe $\bV$ (i.e.\ by working internally to $\BBB$), every statement about $\BB$-categories carries over to one about large $\BB$-categories as well. Also, we will sometimes omit specifying the relative size of a $\BB$-category if it is evident from the context.
    
    \item[Objects and morphisms] An object of a $\BB$-category $\I{C}$ is a local section $c\colon A\to\I{C}$ where $A\in\BB$ is called the \emph{context} of $c$. A morphism in $\I{C}$ is given by a local section $f\colon A\to \I{C}^{\Delta^1}$. Given objects $c,d\colon A\rightrightarrows \I{C}$, one defines the $\BB$-groupoid $\map{\I{C}}(c,d)$ of morphisms between $c$ and $d$ as the pullback
    \begin{equation*}
        \begin{tikzcd}
            \map{\I{C}}(c,d)\arrow[r]\arrow[d] & \I{C}_1\arrow[d, "{(d^1,d^0)}"]\\
            A\arrow[r, "{(c,d)}"]& \I{C}_0\times\I{C}_0.
        \end{tikzcd}
    \end{equation*}
    We denote a section $f\colon A\to \map{\I{C}}(c,d)$ by $f\colon c\to d$. A map $f\colon A\to \I{C}^{\Delta^1}$ is an equivalence if it factors through $s_0\colon \I{C}\into\I{C}^{\Delta^1}$. For any object $c\colon A\to \I{C}$ there is a canonical equivalence $\id_c\colon c\to c$ that is determined by the lift $s_0 c\colon A\to \I{C}_0\to\I{C}_1$ of $(c,c)\colon A\to \I{C}_0\times\I{C}_0$.
    
    Viewed as an $\SS$-valued sheaf on $\Over{\BB}{A}$, the object $\map{\I{C}}(c,d)$ is given by the assignment
    \begin{equation*}
        \Over{\BB}{A}\ni(s\colon B\to A)\mapsto \map{\I{C}(B)}(s^\ast c,s^\ast d)
    \end{equation*}
    where $s^\ast c= c s$ and likewise for $d$.
    
    \item[Fully faithful functors] A functor $f\colon \I{C}\to\I{D}$ between $\BB$-categories is said to be fully faithful if it is internally right orthogonal to the map $\Delta^0\sqcup\Delta^0\to \Delta^1$. Fully faithful functors are monomorphisms, hence the full subcategory $\Sub^{\mathrm{full}}(\I{D})\into\Over{\Cat(\BB)}{\I{D}}$ that is spanned by the fully faithful functors into $\I{D}$ is a poset whose objects we call \emph{full subcategories} of $\I{D}$. Taking core $\BB$-groupoids yields an equivalence $\Sub^{\mathrm{full}}(\I{D})\simeq \Sub(\I{D}_0)$ between the poset of full subcategories of $\I{D}$ and the poset of subobjects of $\I{D}_0\in\BB$.
    
    Dually, a functor $f\colon \I{C}\to\I{D}$ between $\BB$-categories is essentially surjective if contained in the internal saturation of the map $\Delta^0\sqcup\Delta^0\to \Delta^1$. This turns out to be the case if and only if $f_0\colon \I{C}_0\to\I{D}_0$ is a cover in $\BB$.
    
 \item[The universe] The $\CatSS$-valued sheaf $\Over{\BB}{-}$ on $\BB$ gives rise to a large $\BB$-category $\Univ[\BB]$ that we refer to as the \emph{universe for $\BB$-groupoids}. We will often omit the subscript if it is clear from the context. 
    By definition, the objects of $\Univ$ in context $A\in\BB$ are in bijection with the $\Over{\BB}{A}$-groupoids. Moreover, the mapping $\Over{\BB}{A}$-groupoid between two such objects can be identified with the internal mapping object of the $\infty$-topos $\Over{\BB}{A}$. Also, the poset of full subcategories of the universe $\Univ$ (i.e.\ of \emph{subuniverses}) is canonically equivalent to the poset of \emph{local classes} of morphisms in $\BB$.
    
    There is a fully faithful functor $\Univ[\BB]\into\Univ[\BBB]$ of $\BBBB$-categories ( we will call these \emph{very large} $\BB$-categories) that corresponds to the inclusion $\Over{\BB}{-}\into\Over{\BBB}{-}$. An object $g\colon A\to \Univ[\BBB]$ in context $A\in\BBB$ is contained in $\Univ[\BB]$ if and only if the associated map $P\to A\in \Over{\BBB}{A}$ is $\bU$-small, i.e.\ satisfies the condition that whenever $A^\prime\to A$ is a map in $\BBB$ where $A^\prime$ is contained in $\BB$, the fibre product $A^\prime \times_{A} P$ is contained in $\BB$ as well. 
    
    \item[The $\BB$-category of $\BB$-categories] The $\CatSS$-valued presheaf $\Cat(\Over{\BB}{-})$ defines a sheaf on $\BB$ and therefore a large $\BB$-category $\ICat_{\BB}$. The inclusion $\Over{\BB}{-}\into\Cat(\Over{\BB}{-})$ determines a fully faithful functor $\Univ\into\ICat_{\BB}$, and both the groupoidification functor and the core $\BB$-groupoid functor determine maps $(-)^\gp\colon\ICat_{\BB}\to\Univ$ and $(-)^\core\colon\ICat_{\BB}\to\Univ$. A full subcategory $\I{U}\into\ICat_{\BB}$ is referred to as an \emph{internal class} of $\BB$-categories.
    
    There is a fully faithful functor $\ICat_{\BB}\into\ICat_{\BBB}$ of very large $\BB$-categories that corresponds to the inclusion $\Over{\Cat(\BB)}{-}\into\Over{\Cat(\BBB)}{-}$. An object $g\colon A\to \ICat_{\BBB}$ in context $A\in\BBB$ is contained in $\ICat_{\BB}$ if and only if the associated map $\I{P}\to A\in \Cat(\Over{\BBB}{A})$ is $\bU$-small, i.e.\ satisfies the condition that whenever $A^\prime\to A$ is a map in $\BBB$ where $A^\prime$ is contained in $\BB$, the fibre product $A^\prime \times_{A} P$ is contained in $\Cat(\Over{\BB}{A})$. 
    
    \item[Left fibrations] A functor $p\colon \I{P}\to\I{C}$ between $\BB$-categories is called a left fibration if it is internally right orthogonal to the map $d^1\colon \Delta^0\into\Delta^1$. A functor that is contained in the internal saturation of this map is said to be initial. One obtains a $\CatSS$-valued sheaf $\LFib$ on $\Cat(\BB)$. Given any $\BB$-category $\I{C}$, the large $\BB$-category that corresponds to the sheaf $\LFib(-\times \I{C})$ on $\BB$ is denoted by $\ILFib_{\I{C}}$.
    
    Dually, a functor $p\colon \I{P}\to\I{C}$ is a right fibration if it is internally right orthogonal to $d^0\colon \Delta^0\into\Delta^1$, and a functor which is contained in the internal saturation of this map is called final. The associated sheaf of $\infty$-categories on $\Cat(\BB)$ is denoted by $\RFib$, and for $\I{C}\in\Cat(\BB)$ one obtains a large $\BB$-category $\IRFib_{\I{C}}$ via the sheaf $\RFib(\I{C}\times -)$ on $\BB$.
    
    \item[Slice categories] For any $\BB$-category $\I{C}$ and any object $c\colon A\to \I{C}$, one defines the slice $\BB$-category $\Under{\I{C}}{c}$ via the pullback
    \begin{equation*}
        \begin{tikzcd}
            \Under{\I{C}}{c}\arrow[d, "(\pi_c)_!"]\arrow[r] & \I{C}^{\Delta^1}\arrow[d, "{(d^1,d^0)}"]\\
            A\times\I{C}\arrow[r, "{c\times\id}"] & \I{C}\times\I{C}.
        \end{tikzcd}
    \end{equation*}
    The map $(\pi_c)_!$ turns out to be a left fibration. The slice $\BB$-category $\Over{\I{C}}{c}$ is defined in the evident dual way, and the map $(\pi_c)_!\colon \Over{\I{C}}{c}\to \I{C}\times A$ is a right fibration.
    
    \item[Initial objects] An object $c\colon A\to \I{C}$ in a $\BB$-category is said to be initial if the associated map $1\to \pi_A^\ast\I{C}$ in $\Cat(\Over{\BB}{A})$ is an initial functor. Dually, $c$ is final if the associated map in $\Cat(\Over{\BB}{A})$ is a final functor. For any $c\colon A\to \I{C}$ the object $\id_{c}\colon A\to \Under{\I{C}}{c}$ is initial, and the object $\id_c\colon A\to \Over{\I{C}}{c}$ is final.
    \item[The Grothendieck construction]
    There is an equivalence
    \begin{equation*}
        \ILFib_{\I{C}}\simeq \iFun{\I{C}}{\Univ}
    \end{equation*}
    of large $\BB$-categories that is natural in $\I{C}$ and that we call the \emph{Grothendieck construction}. 
    If $1_{\Univ}\colon 1\to \Univ$ denotes the object that corresponds to the final object in $\BB$, the map $\Under{\Univ}{1_{\Univ}}\to\Univ$ is the \emph{universal} left fibration: any left fibration $p\colon \I{P}\to \I{C}$ between $\BB$-categories arises in a unique way as the pullback of the universal left fibration along the functor $f\colon \I{C}\to\Univ$ that corresponds to $p$ by the Grothendieck construction.
    
    Dually, there is an equivalence
        \begin{equation*}
        \IRFib_{\I{C}}\simeq \iFun{\I{C}^{\op}}{\Univ}
    \end{equation*}
    where $\I{C}^{\op}$ is the $\BB$-category that is obtained from $\I{C}$ by precomposing the underlying simplicial object with the equivalence $\op\colon \Delta\simeq\Delta$ that carries a linearly ordered set to its opposite. Accordingly, the map $(\Under{\Univ}{1})^{\op}\to \Univ^{\op}$ is the universal right fibration.
    
    \item[Mapping bifunctors]
    For any $\BB$-category $\I{C}$ one defines the twisted arrow $\BB$-category $\Tw(\I{C})$ by means of the formula $\Tw(\I{C})_n=(\I{C}^{(\Delta^n)^{\op}\diamond\Delta^n})_0$, where $-\diamond -$ denotes the join of $\infty$-categories. By construction this $\BB$-category admits a natural left fibration $\Tw(\I{C})\to\I{C}^{\op}\times\I{C}$. By the Grothendieck construction, this left fibration is classified by a bifunctor $\I{C}^\op\times\I{C}\to\Univ$ which we denote by $\map{\I{C}}$ and that we refer to as the \emph{mapping $\BB$-groupoid} bifunctor. The image of a pair of objects $(c,d)\colon A\to \I{C}^{\op}\times\I{C}$ recovers the mapping $\BB$-groupoid $\map{\I{C}}(c,d)\in\Over{\BB}{A}$. The bifunctor $\map{\I{C}}$ transposes to a functor $h\colon \I{C}\to \IPSh_{\Univ}(\I{C})=\iFun{\I{C}^{\op}}{\Univ}$ that is called the \emph{Yoneda embedding}.

    \item[Yoneda's lemma]
    For any $\BB$-category $\I{C}$, the composition
    \begin{equation*}
            \I{C}^{\op}\times\IPSh_{\Univ}(\I{C})\xrightarrow{"h\times \id} \IPSh_{\Univ}(\I{C}) \times  \IPSh_{\Univ}(\I{C}) \xrightarrow{\map{\IPSh_{\Univ}(\I{C})}} \Univ
    \end{equation*}
    is equivalent to the evaluation functor $\ev\colon \I{C}^{\op}\times\IPSh_{\Univ}(\I{C})\to\Univ$. In particular, this implies that the Yoneda embedding $h\colon \I{C}\to\IPSh_{\Univ}(\I{C})$ is fully faithful. An object $A\to \IPSh_{\Univ}(\I{C})$ is contained in $\I{C}$ if and only if the associated right fibration $p\colon \I{P}\to \I{C}\times A$ admits a final section $A\to \I{P}$ over $A$. If this is the case, one obtains an equivalence $\Over{\I{C}}{c}\simeq \I{P}$ over $\I{C}\times A$ where $c$ is the image of the final section $A\to \I{P}$ along the functor $\I{P}\to \I{C}$.
    
    \item[Adjunctions]
    The bifunctor $\Fun_{\BB}(-,-)$ exhibits $\Cat(\BB)$ as a $\CatS$-enriched $\infty$-category and therefore in particular as an $(\infty,2)$-category. One can therefore make sense of the usual $2$-categorical definition of an adjunction in $\Cat(\BB)$. Explicitly, an adjunction between $\BB$-categories $\I{C}$ and $\I{D}$ is a pair of functors $(l,r)\colon \I{C}\leftrightarrows\I{D}$ together with maps $\eta\colon\id_{\I{D}}\to rl$ and $\epsilon\colon lr\to\id_{\I{C}}$ that satisfy the \emph{triangle identities}: the compositions $\epsilon l\circ l\eta$ and $r\epsilon\circ\eta r$ are equivalent to the identity. The datum of such an adjunction corresponds precisely to a \emph{relative} adjunction between the cartesian fibrations $\int \I{C}\to\BB$ and $\int\I{D}\to \BB$ that are classified by (the underlying $\CatS$-valued presheaves of) $\I{C}$ and $\I{D}$. An alternative and equivalent approach to adjunctions is one that is given in terms of mapping $\BB$-groupoids: a pair $(l,r)$ as above defines an adjunction precisely if there is an equivalence of functors
    \begin{equation*}
    \map{\I{D}}(,-r(-))\simeq\map{\I{C}}(l(-),-).
    \end{equation*}
    In particular, a functor $r\colon\I{C}\to\I{D}$ admits a left adjoint if and only if the copresheaf $\map{\I{D}}(d,r(-))$ is corepresentable for every object $d\colon A\to\I{D}$.
    
    \item[Limits and colimits]
    If $\I{I}$ and $\I{C}$ are $\BB$-categories, the \emph{colimit} functor $\colim_{\I{I}}$ is defined to be the left adjoint of the diagonal map $\diag\colon\I{C}\to\iFun{\I{I}}{\I{C}}$, provided that it exists. Dually, the \emph{limit} functor $\lim_{\I{I}}$ is defined as the right adjoint of the diagonal map. Even if such a colimit or limit functor does not exist, one can define the colimit $\colim d$ of a diagram $d\colon A\to\iFun{\I{I}}{\I{C}}$ as a corepresenting object of the copresheaf $\map{\iFun{\I{I}}{\I{C}}}(d,\diag(-))$, and the limit $\lim d$ of the diagram $d$ as a representing object of the presheaf $\map{\iFun{\I{I}}{\I{C}}}(\diag(-), d)$.
    
    A functor $f\colon\I{C}\to\I{D}$ \emph{preserves} $\I{I}$-indexed colimits if the map $\colim_{\I{I}} f_\ast\to f\colim_{\I{I}}$ that is defined as the mate transformation of the equivalence $f_\ast \diag\simeq\diag f$ is an equivalence. Dually, $f$ preserves $\I{I}$-indexed limits if the map $f\lim_{\I{I}}\to\lim_{\I{I}}f_\ast$ is an equivalence.
    
    \item[$\I{U}$-Cocompleteness]
    If $\I{U}$ is an internal class of $\BB$-categories, a (large) $\BB$-category $\I{C}$ is \emph{$\I{U}$-cocomplete} if for every $A\in\BB$ the $\Over{\BB}{A}$-category $\pi_A^\ast\I{C}$ admits all colimits that are indexed by objects in $\I{U}(A)$. Similarly, a functor $f\colon\I{C}\to\I{D}$ between $\I{U}$-cocomplete large $\BB$-categories is $\I{U}$-cocontinuous if $\pi_A^\ast f$ preserves all such colimits. If $\I{U}=\ICat_{\BB}$, we simply refer to such $\BB$-categories and functors as being \emph{cocomplete} and \emph{cocontinuous}, respectively. The universe $\Univ$ is a cocomplete $\BB$-category, and so is $\IPSh_{\Univ}(\I{C})$ for every $\I{C}\in\Cat(\BB)$. Since $\ICat_{\BB}$ arises as a reflective subcategory of $\IPSh_{\Univ}(\Delta)$, this implies that $\ICat_{\BB}$ is cocomplete as well.
    
    \item[Kan extensions]
    If $f\colon\I{C}\to\I{D}$ is a functor in $\Cat(\BB)$ and if $\I{E}$ is an arbitrary $\BB$-category, the functor of left Kan extension $f_!\colon\iFun{\I{C}}{\I{E}}\to\iFun{\I{D}}{\I{E}}$ is defined as the left adjoint of the functor $f^\ast$ that is given by precomposition with $f$. Such a functor of left Kan extension exists for example whenever $\I{C}$ is small, $\I{D}$ is locally small and $\I{E}$ is (possibly large and) cocomplete.
    
    \item[The universal property of presheaves]
    The Yoneda embedding $h\colon\I{C}\into\IPSh_{\Univ}(\I{C})$ exhibits $\IPSh_{\Univ}(\I{C})$ as the \emph{universal} cocomplete $\BB$-category that admits a map from $\I{C}$. More precisely, for any cocomplete large $\BB$-category $\I{E}$, the functor $h_!$ of left Kan extension determines an equivalence
    \begin{equation*}
    h_!\colon\iFun{\I{C}}{\I{E}}\simeq\iFun{\IPSh_{\Univ}(\I{C})}{\I{E}}^{\cc}
    \end{equation*}
    where the right-hand side denotes the full subcategory of $\iFun{\IPSh_{\Univ}(\I{C})}{\I{E}}$ that is spanned by the cocontinuous functors. More generally, if $\I{U}$ is an internal class of $\BB$-categories, the full subcategory $\IPSh_{\Univ}^{\I{U}}(\I{C})\into\IPSh_{\Univ}(\I{C})$ that is generated by $\I{C}$ under $\I{U}$-colimits is the free $\I{U}$-cocompletion of $\I{C}$, in that for every $\I{U}$-cocomplete $\BB$-category $\I{E}$ the functor of left Kan extension along $h\colon\I{C}\into\IPSh_{\Univ}^{\I{U}}(\I{C})$ exists and determines an equivalence
    \begin{equation*}
        h_!\colon\iFun{\I{C}}{\I{E}}\simeq\iFun{\IPSh_{\Univ}^{\I{U}}(\I{C})}{\I{E}}^{\cocont{\I{U}}}
     \end{equation*}
     in which the right-hand side denotes the full subcategory of $\iFun{\IPSh_{\Univ}^{\I{U}}(\I{C})}{\I{E}}$ that is spanned by the $\I{U}$-cocontinuous functors.
    
\end{description}

\section{Cocartesian fibrations}
\label{sec:CocartesianFibrations}
In this chapter, we define and study cocartesian fibrations between $\BB$-categories. We introduce the notion in \S~\ref{sec:Definitions}, where we also discuss how cocartesian fibrations can be detected from the perspective of maps between sheaves of $\infty$-categories. In \S~\ref{sec:cocartesianMorphisms}, we study cocartesian morphisms and show that cocartesian fibrations are precisely those functors with respect to which their domain has sufficiently  many cocartesian morphisms.
\subsection{Definition and section-wise characterisation}
\label{sec:Definitions}
If $p\colon\I{P}\to\I{C}$ is a functor in $\Cat(\BB)$, we obtain a functor $\res_p\colon\I{P}^{\Delta^1}\to\Comma{\I{P}}{\I{C}}{\I{C}}$ that makes the diagram
\begin{equation*}
	\begin{tikzcd}
	\I{P}^{\Delta^1}\arrow[d]\arrow[r, "\res_p"]\arrow[rr, bend left, "p_\ast"] & \Comma{\I{P}}{\I{C}}{\I{C}}\arrow[r]\arrow[d] & \I{C}^{\Delta^1}\arrow[d]\\
	\I{P}\times\I{P}\arrow[r, "\id\times p"] & \I{P}\times\I{C}\arrow[r, "p\times\id"] & \I{C}\times\I{C}
	\end{tikzcd}
\end{equation*}
commute. Here the right square is a pullback by definition of the comma $\BB$-category $\Comma{\I{P}}{\I{C}}{\I{C}}$ (cf.~\cite[\S~4.2]{Martini2021}).
\begin{definition}
	\label{def:cocartesianFibration}
	A functor $p\colon \I{P}\to\I{C}$ between $\BB$-categories is said to be a \emph{cocartesian fibration} if the functor $\res_p\colon \I{P}^{\Delta^1}\to \Comma{\I{P}}{\I{C}}{\I{C}}$ admits a fully faithful left adjoint $\lift_p\colon\colon \Comma{\I{P}}{\I{C}}{\I{C}}\into\I{P}^{\Delta^1}$. 
	
	If $p\colon \I{P}\to\I{C}$ and $q\colon \I{Q}\to\I{D}$ are cocartesian fibrations, a \emph{cocartesian functor} between $p$ and $q$ is a commutative square
	\begin{equation*}
	\begin{tikzcd}
		\I{P}\arrow[d, "p"]\arrow[r, "g"] & \I{Q}\arrow[d, "q"]\\
		\I{C}\arrow[r, "f"] & \I{D}
	\end{tikzcd}
	\end{equation*}
	such that the mate of the induced commutative square
	\begin{equation*}
			\begin{tikzcd}
				\I{P}^{\Delta^1}\arrow[d, "\res_p"] \arrow[r, "g_\ast"] & \I{Q}^{\Delta^1}\arrow[d, "\res_q"]\\
				\Comma{\I{P}}{\I{C}}{\I{C}}\arrow[r, "f_\ast"] & \Comma{\I{Q}}{\I{D}}{\I{D}}
			\end{tikzcd}
	\end{equation*}
	commutes as well.
\end{definition}

\begin{remark}
	\label{rem:CocartIsLocal}
	Recall from~\cite[Remark~3.3.6]{Martini2021a} that the condition of a functor to be a right adjoint is \emph{local} in $\BB$. Since similarly the condition of a functor to be fully faithful is local as well, we conclude that a functor $p\colon\I{P}\to\I{C}$ is a cocartesian fibration if and only if there is a cover $\bigsqcup_i A_i\onto 1$ such that $\pi_{A_i}^\ast(p)$ is a cocartesian fibration for all $i$. A similar observation can be made for cocartesian functors.
\end{remark}

\begin{remark}
	\label{rem:cocartesianFibrationsSpaces}
	In the case $\BB\simeq \SS$, our definition recovers the notion of cocartesian fibrations and cocartesian functors in $\CatS$ in the sense of~\cite{htt}, cf.\ Proposition~\ref{prop:characterisationCocartesianFibrationCocartesianMaps} below.
\end{remark}

\begin{remark}
	There is an evident dual notion of \emph{cartesian} fibrations, namely those maps $p\colon \I{P}\to\I{C}$ for which the restriction functor $\res_p\colon \I{P}^{\Delta^1}\to\Comma{\I{C}}{\I{C}}{\I{P}}$ admits a fully faithful \emph{right} adjoint. One defines cartesian functors between such cartesian fibrations in the obvious way. By~\cite[Proposition~3.1.14]{Martini2021a} a functor $p$ is a cartesian fibration if and only if $p^{\op}$ is a cocartesian fibration, and a map between cartesian fibrations defines a cartesian functor if and only if its opposite defines a cocartesian functor. In what follows, we will therefore restrict our attention to cocartesian fibrations, as every statement about those can be dualised in the appropriate sense to be turned into a statement about cartesian fibrations.
\end{remark}

\begin{remark}
	\label{rem:cocartesianSquaresSimplified}
	In the situation of Definition~\ref{def:cocartesianFibration}, the argumentation in~\cite[Remark~3.2.9]{Martini2021a} shows that the square
	\begin{equation*}
		\begin{tikzcd}
			\I{P}\arrow[d, "p"]\arrow[r, "g"] & \I{Q}\arrow[d, "q"]\\
			\I{C}\arrow[r, "f"] & \I{D}
		\end{tikzcd}
	\end{equation*}
	defines a cocartesian functor already when there is an \emph{arbitrary} equivalence $\lift_q f_\ast\simeq g_\ast\lift_p$.
\end{remark}

\begin{remark}
	\label{rem:equivalencesAreCocartesian}
	Suppose that $p\colon \I{P}\to\I{C}$ is a cocartesian fibration. Since the projection $\Comma{\I{P}}{\I{C}}{\I{C}}\to\I{P}$ is the pullback of $d_1\colon \I{C}^{\Delta^1}\to\I{C}$ along $p$ and since $d_1$ admits a fully faithful left adjoint $s_0$, the projection $\Comma{\I{P}}{\I{C}}{\I{C}}\to\I{P}$ also admits a fully faithful left adjoint~\cite[Lemma~6.3.9]{Martini2021a}, which we will denote by $s_0$ as well. By the uniqueness of adjoints, we thus obtain a commutative diagram
	\begin{equation*}
	\begin{tikzcd}
	\I{P} \arrow[d, "s_0", hookrightarrow] \arrow[dr, "s_0", hookrightarrow]& \\
	\Comma{\I{P}}{\I{C}}{\I{C}}\arrow[r, "\lift_p", hookrightarrow] & \I{P}^{\Delta^1}. 
	\end{tikzcd}
	\end{equation*}
\end{remark}

\begin{proposition}
	\label{prop:externalCharacterisationCocartesian}
	A functor $p\colon \I{P}\to\I{C}$ in $\BB$ is a cocartesian fibration if and only if
	\begin{enumerate}
	\item for every $A\in \BB$ the functor $p(A)$ is a cocartesian fibration of $\infty$-categories;
	\item for every $s\colon B\to A$ in $\BB$ the commutative square
	\begin{equation*}
		\begin{tikzcd}
		\I{P}(A)\arrow[r, "s^\ast"]\arrow[d, "p(A)"] & \I{P}(B)\arrow[d, "p(B)"]\\
		\I{C}(A)\arrow[r, "s^\ast"] &\I{C}(B)
		\end{tikzcd}
	\end{equation*}
	defines a cocartesian functor.
	\end{enumerate}
	Furthermore, if $p\colon \I{P}\to \I{C}$ and $q\colon \I{Q}\to\I{D}$ are cocartesian fibrations, a commutative square
	\begin{equation*}
		\begin{tikzcd}
			\I{P}\arrow[d, "p"]\arrow[r, "g"] & \I{Q}\arrow[d, "q"]\\
			\I{C}\arrow[r, "f"] & \I{D}
		\end{tikzcd}
	\end{equation*}
	defines a cocartesian functor precisely if it does so section-wise.
\end{proposition}
\begin{proof}
	Since the local sections functor $\Fun_{\BB}(A,-)$ commutes with limits and the powering functor and preserves full faithfulness, this statement is an immediate consequence of the section-wise characterisation of right adjoint functors (~\cite[Proposition~3.2.8]{Martini2021a} together with~\cite[Remark~3.2.9]{Martini2021a}) and the fact that full faithfulness can be detected section-wise as well.
\end{proof}

\begin{proposition}
	\label{prop:pullbackCocartesianFibration}
	Suppose that
	\begin{equation*}
		\begin{tikzcd}
			\I{P}\arrow[d, "p"]\arrow[r, "g"] & \I{Q}\arrow[d, "q"]\\
			\I{C}\arrow[r, "f"] & \I{D}
		\end{tikzcd}
	\end{equation*}
	is a pullback square in $\Cat(\BB)$ such that $q$ is a cocartesian fibration. Then $p$ is a cocartesian fibration, and the square itself defines a cocartesian functor.
\end{proposition}
\begin{proof}
	The pullback square gives rise to a commutative square
	\begin{equation*}
	\begin{tikzcd}
	\I{P}^{\Delta^1}\arrow[d, "\res_p"]\arrow[r, "g_\ast"] & \I{Q}^{\Delta^1}\arrow[d, "\res_q"]\\
	\Comma{P}{\I{C}}{\I{C}}\arrow[r, "f_\ast"] & \Comma{\I{P}}{\I{D}}{\I{D}}
	\end{tikzcd}
	\end{equation*}
	that is easily seen to be a pullback too. Thus the claim follows from~\cite[Lemma~6.3.9]{Martini2021a}.
\end{proof}
We denote by $\Cocart\into \Fun(\Delta^1,\Cat(\BB))$ the subcategory that is spanned by the cocartesian fibrations and cocartesian squares. By Proposition~\ref{prop:pullbackCocartesianFibration}, this defines a cartesian subfibration of the target fibration $d_0\colon \Fun(\Delta^1,\Cat(\BB))\to\Cat(\BB)$. For $\I{C}\in\Cat(\BB)$, we denote by $\Cocart(\I{C})$ the fibre of $\Cocart$ over $\I{C}$. Clearly we have $\Cocart(A)\simeq \Cat(\Over{\BB}{A})$ for any $A\in\BB$ since for any $\Over{\BB}{A}$-category $\I{P}$ and for $\pi_{\I{P}}\colon \I{P}\to A$ the structure map, the restriction functor $\res_{\pi_{\I{P}}}$ is an equivalence. In other words, the restriction of the presheaf $\Cocart$ along the inclusion $\BB\into\Cat(\BB)$ recovers the sheaf $\Cat(\Over{\BB}{-})$.

\begin{proposition}
	For every $\BB$-category $\I{C}$ and every simplicial object $K$ in $\BB$, the functor $\iFun{K}{-}$ restricts to a functor $\Cocart(\I{C})\to\Cocart(\iFun{K}{\I{C}})$.
\end{proposition}
\begin{proof}
	This follows from the observation that $\iFun{K}{-}$ commutes with limits and powering, carries adjunctions to adjunctions (see~\cite[Corollary~3.1.11]{Martini2021a}) and preserves the property of functors to be fully faithful.
\end{proof}

\subsection{Cocartesian morphisms}
\label{sec:cocartesianMorphisms}
Let $p\colon\I{P}\to\I{C}$ be a cocartesian fibration. Then $\lift_p\colon \Comma{\I{P}}{\I{C}}{\I{C}}\into\I{P}^{\Delta^1}$ determines a subobject of maps in $\I{P}$. Our goal in this section is to study these maps.

\begin{definition}
	\label{def:cocartesianMorphism}
	Let $p\colon \I{P}\to\I{C}$ be a functor between $\BB$-categories. A map $f\colon x\to y$ in $\I{P}$ (in context $A\in\BB$) is said to be \emph{$p$-cocartesian} if the commutative square
	\begin{equation*}
	\begin{tikzcd}
	\Under{\I{P}}{y}\arrow[r, "f^\ast"]\arrow[d, "p"] & \Under{\I{P}}{x}\arrow[d, "p"]\\
	\Under{\I{C}}{p(y)}\arrow[r, "p(f)^\ast"] & \Under{\I{C}}{p(x)}
	\end{tikzcd}
	\end{equation*}
	is a pullback square in $\Cat(\Over{\BB}{A})$.
\end{definition}
The commutative square in Definition~\ref{def:cocartesianMorphism} formally arises from evaluating the morphism of bifunctors $\map{\I{P}}(-,-)\to\map{\I{C}}(p(-),p(-))$ (which is itself constructed by using functoriality of the twisted arrow construction) at $f\colon \Delta^1\otimes A\to \I{C}^{\op}$. This produces a commutative square
\begin{equation*}
		\begin{tikzcd}
		\Under{\I{P}}{y}\arrow[r, "f^\ast"]\arrow[d, "p"] & \Under{\I{P}}{x}\arrow[d, "p"]\\
		\Under{\I{C}}{p(y)}\times_{\I{C}}\I{P}\arrow[r, "p(f)^\ast"] & \Under{\I{C}}{p(x)}\times_{\I{C}}\I{{P}},
		\end{tikzcd}
\end{equation*}
that recovers the square from Definition~\ref{def:cocartesianMorphism} upon pasting with the pullback square
	\begin{equation*}
			\begin{tikzcd}
			\Under{\I{C}}{p(y)}\times_{\I{C}}\I{P}\arrow[r, "p(f)^\ast"]\arrow[d] & \Under{\I{C}}{p(x)}\times_{\I{C}}\I{{P}}\arrow[d]\\
			\Under{\I{C}}{p(y)}\arrow[r, "p(f)^\ast"] & \Under{\I{C}}{p(x)}.
			\end{tikzcd}
	\end{equation*}

\begin{remark}
	\label{rem:stabilityCocartesianMaps}
	In the situation of Definition~\ref{def:cocartesianMorphism}, the pasting lemma for pullback squares implies that the subobject of $\I{P}_1$ that is spanned by the cocartesian morphisms in $\I{P}$ is closed under composition and equivalences in the sense of~\cite[Proposition~2.9.8]{Martini2021a}. Furthermore, the pasting lemma shows that if $f\colon x\to y$ is a cocartesian map and $g\colon y\to z$ is an arbitrary map, then $gf$ is cocartesian if and only if $g$ is cocartesian. 
\end{remark}

Let $p\colon \I{P}\to \I{C}$ be a functor in $\Cat(\BB)$ and let $f\colon x\to y$ be a map in $\I{P}$ in context $A\in\BB$ with image $\alpha\colon c\to d$ in $\I{C}$. Let $\I{P}^{\Delta^2}\vert_f$ and $\I{P}^{\Lambda^2_0}\vert_f$ be the fibres of $d_{\{0,1\}}\colon \I{P}^{\Delta^2}\to\I{P}^{\Delta^1}$ and $d_{\{0,1\}}\colon \I{P}^{\Lambda^2_0}\to\I{P}^{\Delta^1}$ over $f\colon A\to \I{P}^{\Delta^1}$. Define $\I{C}^{\Delta^2}\vert_\alpha$ and $\I{C}^{\Lambda^2_0}\vert_\alpha$ likewise. One then obtains:
\begin{proposition}
	\label{prop:characterisationCocartesianMorphismLifting}
	Let $p\colon \I{P}\to \I{C}$ be a functor in $\Cat(\BB)$ and let $f\colon x\to y$ be a map in $\I{P}$ in context $A\in\BB$. Let $\alpha$ be the image of $f$ along $p$. Then $f$ is cocartesian if and only if the commutative square
	\begin{equation*}
	\begin{tikzcd}
	\I{P}^{\Delta^2}\vert_f\arrow[r]\arrow[d] & \I{P}^{\Lambda^2_0}\vert_f\arrow[d]\\
	\I{C}^{\Delta^2}\vert_{\alpha}\arrow[r] & \I{C}^{\Lambda^2_0}\vert_{\alpha}
	\end{tikzcd}
	\end{equation*}
	is cartesian.
\end{proposition}
\begin{proof}
	We may assume without loss of generality $A\simeq 1$. Moreover, note that there is a commutative diagram
	\begin{equation*}
	\begin{tikzcd}[column sep=small]
	\I{P}^{\Delta^2}\vert_f\arrow[rr]\arrow[dr, "d_{\{2\}}"'] && \I{P}^{\Lambda^2_0}\vert_f\arrow[dl, "d_{\{2\}}"]\\
	& \I{P}. &
	\end{tikzcd}
	\end{equation*}
	in which the diagonal maps are left fibrations. Hence the map $\I{P}^{\Delta^2}\vert_f\to\I{P}^{\Lambda^2_0}\vert_f$ is a left fibration as well. As the same is true for the map $\I{C}^{\Delta^2}\vert_\alpha\to\I{C}^{\Lambda^2_0}\vert_\alpha$, the square in the statement of the proposition is a pullback if and only if it is carried to a pullback square by the core $\BB$-groupoid functor~\cite[Proposition~4.1.18]{Martini2021}. Let $z$ be the tautological object in $\I{P}$ in context $\I{P}_0$, i.e.\ the object that is determined by the identity $\I{P}_0\simeq\I{P}_0$. We then obtain a commutative diagram
	\begin{equation*}
		\begin{tikzcd}[column sep={5.8em,between origins}, row sep={4em,between origins}]
		& \map{\I{P}}(\pi_{\I{P}_0}^\ast y,z)\arrow[from=rr] \arrow[dl]\arrow[dd]&& (\I{P}^{\Delta^2}\vert_f)_0\arrow[rr, dashed] \arrow[dd]\arrow[dl]&& (\I{P}^{\Lambda^2_0}\vert_f)_0\arrow[rr]\arrow[dl]\arrow[dd]&& \map{\I{P}}(\pi_{\I{P}_0}^\ast x,z)\arrow[dl]\arrow[dd]\\
		\I{P}_1\arrow[from=rr, "d_{\{1,2\}}"', crossing over, near start]\arrow[dd, "{(d_1,d_0)}", near start] && \I{P}_2 \arrow[rr, crossing over]&& (\I{P}^{\Lambda^2_0})_0\arrow[rr, "d_{\{0,2\}}", crossing over, near end] && \I{P}_1 & \\
		& \I{P}_0 \arrow[dl, "y\times\id"]\arrow[from=rr, "\id"', near end]&&\I{P}_0 \arrow[dl, "f\times\id"]\arrow[rr, "\id", near start]&& \I{P}_0\arrow[rr, "\id", near start]\arrow[dl, "f\times\id"]&&\I{P}_0\arrow[dl, "x\times\id"]\\
		\I{P}_0\times\I{P}_0\arrow[from=rr, "d_0\times\id"', near start] && \I{P}_1\times\I{P}_0\arrow[rr, "\id", near end]\arrow[from=uu, "{(d_{\{0,1\}}, d_{\{2\}})}", crossing over, near start]&& \I{P}_1\times\I{P}_0\arrow[rr, "d_1\times\id", near end]\arrow[from=uu, "{(d_{\{0,1\}}, d_{\{2\}})}", crossing over, near start] && \I{P}_0\times\I{P}_0 \arrow[from=uu, "{(d_1,d_0)}", crossing over, near start]&
		\end{tikzcd}
	\end{equation*}
	in which the dotted arrow is the map that is induced by the functor $\I{P}^{\Delta^2}\vert_f\to \I{P}^{\Lambda^2_0}\vert_f$ upon applying the core $\BB$-groupoid functor. Note that both the front left and the front right square is cartesian, hence both  $(\I{P}^{\Delta^2}\vert_f)_0\to \map{\I{P}}(\pi_{\I{P}_0}^\ast y,z)$ and $(\I{P}^{\Lambda^2_0}\vert_f)_0\to \map{\I{P}}(\pi_{\I{P}_0}^\ast x,z)$ must be an equivalence. Now by using the argument in~\cite[Remark~4.7.7]{Martini2021}, the composition
	\begin{equation*}
	\map{\I{P}}(\pi_{\I{P}_0}^\ast y,z)\xrightarrow{\simeq} (\I{P}^{\Delta^2}\vert_f)_0\rightarrow (\I{P}^{\Lambda^2_0}\vert_f)_0\xrightarrow{\simeq} \map{\I{P}}(\pi_{\I{P}_0}^\ast x,z)
	\end{equation*}
	recovers the map $f^\ast$. We now note that the above construction is natural in $\I{P}$, in that we may identify the commutative square that arises from applying $(-)^\simeq$ to the square in the statement of the proposition with the commutative diagram
	\begin{equation*}
	\begin{tikzcd}
	\map{\I{P}}(\pi_{\I{P}_0}^\ast y,z)\arrow[r, "(\pi_{\I{P}_0}^\ast f)^\ast"] \arrow[d]&  \map{\I{P}}(\pi_{\I{P}_0}^\ast x,z)\arrow[d]\\
	\map{\I{C}}(\pi_{\I{P}_0}^\ast d,p(z))\arrow[r, "(\pi_{\I{P}_0}^\ast \alpha)^\ast"] & \map{\I{C}}(\pi_{\I{P}_0}^\ast c,p(z))
	\end{tikzcd}
	\end{equation*}
	that is obtained by evaluating the square from Definition~\ref{def:cocartesianMorphism} at $z$. As $z$ is the tautological object, we conclude that this diagram is a pullback if and only if $f$ is cocartesian, as desired.
\end{proof}

Let $\I{C}$ be a $\BB$-category. Observe that the evaluation functor $\ev\colon \Delta^1\otimes\I{C}^{\Delta^1}\to\I{C}$ can be regarded as a morphism $\phi\colon d_1\to d_0$ in $\iFun{\I{C}^{\Delta^1}}{\I{C}}$. By postcomposition with the Yoneda embedding, one thus obtains a map
\begin{equation*}
\phi_\ast\colon \map{\I{C}}(-,d_1(-))\to \map{\I{C}}(-,d_0(-)).
\end{equation*}
Dually, one obtains a map
\begin{equation*}
\phi^\ast\colon \map{\I{C}}(d_0(-),-)\to\map{\I{C}}(d_1(-),-).
\end{equation*}

\begin{lemma}
	\label{lem:mappingGroupoidArrowCategory}
	For any $\BB$-category $\I{C}$, there is a cartesian square
	\begin{equation*}
	\begin{tikzcd}
	\map{\I{C}^{\Delta^1}}(-,-)\arrow[r]\arrow[d] & \map{\I{C}}(d_0(-),d_0(-))\arrow[d, "\phi^\ast"]\\
	\map{\I{C}}(d_1(-),d_1(-))\arrow[r, "\phi_\ast"] & \map{\I{C}}(d_1(-),d_0(-))
	\end{tikzcd}
	\end{equation*}
	in which the left vertical and the upper horizontal map are given by the action of the functors $d_1,d_0\colon \I{C}^{\Delta^1}\rightrightarrows\I{C}$ on mapping groupoids.
\end{lemma}
\begin{proof}
	Let $\epsilon \colon s_0d_1\to \id$ be the counit of the adjunction $s_0\dashv d_1$ and let $\eta\colon \id\to s_0 d_0$ be the unit of the adjunction $d_0\dashv s_0$. Then $\phi\colon \Delta^1\otimes\I{C}^{\Delta^1}\to\I{C}$ can be recovered both by postcomposing $\eta$ with $d_1$ and $\epsilon$ with $d_0$. We may therefore construct a commutative square as in the statement of the lemma as the unique square that makes the diagram
	\begin{equation*}
	\begin{tikzcd}[column sep={between origins,7em}, row sep={between origins,3em}]
	&\map{\I{C}^{\Delta^1}}(-,-)\arrow[rr]\arrow[dd]\arrow[dl, "\simeq"'] && \map{\I{C}}(d_0(-),d_0(-))\arrow[dd, "\phi^\ast"]\arrow[dl, "\simeq"']\\
	\map{\I{C}^{\Delta^1}}(-,-)\arrow[rr, "\eta_\ast", crossing over]\arrow[dd, "\epsilon^\ast"] && \map{\I{C}^{\Delta^1}}(-, s_0 d_0(-))&\\
	&\map{\I{C}}(d_1(-),d_1(-))\arrow[rr, "\phi_\ast", near start]\arrow[dl, "\simeq"'] && \map{\I{C}}(d_1(-),d_0(-))\arrow[dl, "\simeq"']\\
	\map{\I{C}^{\Delta^1}}(s_0d_1(-),-)\arrow[rr, "\eta_\ast"] && \map{\I{C}^{\Delta^1}}(s_0 d_1(-), s_0 d_0(-))\arrow[from=uu, "\epsilon^\ast", crossing over, near start]&
	\end{tikzcd}
	\end{equation*}
	commute. We still need to show that this square is cartesian, for which it suffices to show that it becomes a pullback after being evaluated at an arbitrary pair of maps $f\colon c\to d$ and $g\colon c^\prime\to d^\prime$ in $\I{C}$ in context $A\in\BB$, see~\cite[Proposition~4.3.2]{Martini2021a}. This in turn allows us to argue section-wise in $\BB$, which by using~\cite[Corollary~4.6.8]{Martini2021} lets us further reduce the statement to its analogue for $\infty$-categories. This appears (in a more general form) as~\cite[Proposition~2.3]{glasman2016}.
\end{proof}

\begin{proposition}
	\label{prop:characterisationCocartesianFibrationCocartesianMaps}
	A functor $p\colon \I{P}\to\I{C}$ in $\Cat(\BB)$ is cocartesian if and only if for every object $x$ in $\I{P}$ in context $A\in\BB$ and every map $\alpha\colon c\simeq p(x)\to d$ in $\I{C}$, there exists a cocartesian lift of $\alpha$, i.e.\ a cocartesian morphism $f\colon x\to y$ such that $p(f)\simeq \alpha$.
\end{proposition}
\begin{proof}
	The datum of an object $x\colon A\to\I{P}$ and a map $\alpha\colon c\simeq p(x)\to d$ in $\I{C}$ is tantamount to an object $w\colon A\to \Comma{\I{P}}{\I{C}}{\I{C}}$. In light of this observation, the datum of a lift $f\colon x\to y$ of $\alpha$ is equivalent to a lift of $w$ along $\res_p$. Given such a lift, the definition of comma $\BB$-categories and Lemma~\ref{lem:mappingGroupoidArrowCategory} provide a commutative diagram
	\begin{equation*}
	\begin{tikzcd}[column sep={6em,between origins},row sep={3em,between origins}]
	& && \map{\I{P}^{\Delta^1}}(f, -)\arrow[rr]\arrow[dl]\arrow[dd]&& \map{\I{P}}(y, d_0(-))\arrow[dl]\arrow[dd, "f^\ast"] \\
	\map{\Comma{\I{P}}{\I{C}}{\I{C}}}(w,\res_p(-))\arrow[dd]\arrow[rr]\arrow[from=urrr, dashed, "\rho"'] && \map{\I{C}^{\Delta^1}}(\alpha, p_\ast(-))\arrow[dd, crossing over]\arrow[rr, crossing over] && \map{\I{C}}(d,pd_0(-)) &\\
	& && \map{\I{P}}(x,d_1(-))\arrow[rr, "\phi_\ast", near start]\arrow[dl]&& \map{\I{P}}(x,d_0(-))\arrow[dl]\\
	\map{\I{P}}(x,d_1(-))\arrow[rr]\arrow[from=urrr, "\id"'] && \map{\I{C}}(c, pd_1(-))\arrow[rr, "p\phi_\ast"] && \map{\I{C}}(c,pd_0(-))\arrow[from=uu, crossing over, "\alpha^\ast", near start] &
	\end{tikzcd}
	\end{equation*}
	in which the two squares in the front and the one in the back are cartesian and in which the dotted arrow $\rho$ is given by the composition
	\begin{equation*}
	\map{\I{P}^{\Delta^1}}(f,-)\to \map{\Comma{\I{P}}{\I{C}}{\I{C}}}(\res_p(f),\res_p(-))\xrightarrow{\eta^\ast} \map{\Comma{\I{P}}{\I{C}}{\I{C}}}(w,\res_p(-))
	\end{equation*}
	in which $\eta\colon w\simeq\res_p(f)$ is the specified equivalence that exibits $f$ as a lift of $w$. Now $p$ is cocartesian precisely if every object $w$ in $\Comma{\I{P}}{\I{C}}{\I{C}}$ admits a lift $f$ along $\res_p$ such that the induced map $\rho$ is an equivalence. The proof is thus finished once we show that $\rho$ is an equivalence if and only if $f$ is a cocartesian morphism. If $f$ is a cocartesian morphism, then the right square in the above diagram is a pullback, which clearly implies that $\rho$ is a pullback of the identity on $\map{\I{P}}(x,d_1(-))$ and therefore an equivalence as well. Conversely, if $\rho$ is an equivalence, one obtains a pullback square
	\begin{equation*}
	\begin{tikzcd}
	\map{\I{P}^{\Delta^1}}(f,-)\arrow[r]\arrow[d] & \map{\I{C}}(d,pd_0(-))\arrow[d, "\alpha^\ast"]\\
	\map{\I{P}}(x,d_1(-))\arrow[r] & \map{\I{C}}(c,pd_0(-))
	\end{tikzcd}
	\end{equation*}
	that recovers the square from Definition~\ref{def:cocartesianMorphism} upon precomposition with $s_0\colon \I{P}\into\I{P}^{\Delta^1}$.
	\end{proof}
\begin{remark}
	\label{rem:characterisationCocartesianMorphisms}
	The proof of Proposition~\ref{prop:characterisationCocartesianFibrationCocartesianMaps} shows that if $p\colon \I{P}\to\I{C}$ is a cocartesian fibration, a map $f\colon x\to y$ in $\I{P}$ in context $A\in\BB$ is contained in the subobject $(\Comma{\I{P}}{\I{C}}{\I{C}})_0\into \I{P}_1$ if and only if it is a cocartesian morphism. In particular, the map $f$ is cocartesian with respect to $p$ if and only if it is cocartesian with respect to $p(A)$ when viewed as a map in the $\infty$-category $\I{P}(A)$.
\end{remark}

\begin{remark}
	\label{rem:CocartLiftIsLocal}
	The proof of Proposition~\ref{prop:characterisationCocartesianFibrationCocartesianMaps} shows that if $x\colon A\to\I{P}$ is an arbitrary object, a map $\alpha\colon c\simeq p(x)\to d$ admits a cocartesian lift $f\colon x\to y$ if and only if the copresheaf $\map{\Comma{\I{P}}{\I{C}}{\I{C}}}(w,\res_p(-))$ is corepresentable, where $w\colon A\to \Comma{\I{P}}{\I{C}}{\I{C}}$ is the object that corresponds to the datum $(x,\alpha\colon c\simeq p(x)\to d)$. By making use of~\cite[Remark~3.3.6]{Martini2021a}, we thus conclude that $\alpha$ admits a cocartesian lift $f\colon x\to y$ if and only if there is a cover $(s_i)\colon \bigsqcup_i A_i\onto A$ in $\BB$ such that $s_i^\ast\alpha$ admits a cocartesian lift $f_i\colon s_i^\ast x\to y_i$ for each $i$.
\end{remark}

\begin{corollary}
	\label{cor:fibrewiseCriterionEquivalencesCocartesian}
	Let $p\colon \I{P}\to\I{C}$ and $q\colon \I{Q}\to\I{C}$ be cocartesian fibrations and let $h\colon \I{P}\to\I{Q}$ be a cocartesian functor over $\I{C}$. Then $h$ is an equivalence precisely if it is a fibrewise equivalence, i.e.\ if for every object $c\colon A\to \I{C}$ in context $A\in\BB$ the induced map $h\vert_c\colon \I{P}\vert_c\to\I{Q}\vert_c$ between the fibres is an equivalence.
\end{corollary}
\begin{proof}
	Suppose that $h$ is a fibrewise equivalence. Then $h$ is certainly essentially surjective, so it suffices to show that it is fully faithful as well. To that end, let $x,y\colon A\rightrightarrows \I{P}$ be two objects in $\I{P}$. We wish to show that the induced map
	\begin{equation*}
	\map{\I{P}}(x,y)\to\map{\I{Q}}(h(x),h(y))
	\end{equation*}
	is an equivalence in $\Over{\BB}{A}$. Let $c=p(x)$ and $d=p(y)$. Then the above map lies over $\map{\I{C}}(c,d)$, so it suffices to show that it induces an equivalence after being pulled back along an arbitrary map $B\to \map{\I{C}}(c,d)$ for any $B\in\Over{\BB}{A}$. Upon replacing $\Over{\BB}{A}$ with $\Over{\BB}{B}$, we may thus assume without loss of generality $B\simeq A$ and can assume that there exists a map $\alpha\colon c\to d$ in context $A$. 
	 Using Proposition~\ref{prop:characterisationCocartesianFibrationCocartesianMaps}, we may choose a cocartesian lift $f\colon x\to z$ of $\alpha$ in $\I{P}$. By Remark~\ref{rem:characterisationCocartesianMorphisms} and the assumption that $h$ is a cocartesian functor, the map $h(f)$ must be cocartesian as well. Therefore, we obtain a pullback square
	\begin{equation*}
	\begin{tikzcd}
	\map{\I{P}}(z,y)\arrow[d, "f^\ast"] \arrow[r] & \map{\I{Q}}(h(z),h(y))\arrow[d, "h(f)^\ast"]\\
	\map{\I{P}}(x,y)\arrow[r] & \map{\I{Q}}(h(x),h(y))
	\end{tikzcd}
	\end{equation*}
	in which the upper horizontal map lies over $\map{\I{C}}(d,d)$, such that its pullback along $\id_d\colon A\to \map{\I{C}}(d,d)$ coincides with the pullback of $\map{\I{P}}(x,y)\to\map{\I{Q}}(h(x),h(y))$ over $\alpha\colon A\to \map{\I{C}}(c,d)$. We may therefore assume without loss of generality $c=d$ and need only show that the pullback of $\map{\I{P}}(x,y)\to\map{\I{Q}}(h(x),h(y))$ along $\id_d\colon A\to\map{\I{C}}(d,d)$ is an equivalence. Since this is precisely the morphism that is induced by the functor $h\vert_d$ on mapping $\BB$-groupoids, the result follows.
\end{proof}

\begin{remark}
	\label{rem:cartesianMorphism}
	If $p\colon\I{P}\to\I{C}$ is a functor between $\BB$-categories, we can define a map $f\colon y\to x$ in $\I{P}$ to be \emph{cartesian} if it is cocartesian when viewed as a map $x\to y$ in $\I{P}^\op$. Explicitly, this amounts to the condition that the commutative square
	\begin{equation*}
	\begin{tikzcd}
	\Over{\I{P}}{y}\arrow[r, "f_!"]\arrow[d, "p"] & \Over{\I{P}}{x}\arrow[d, "p"]\\
	\Over{\I{C}}{d}\arrow[r, "\alpha_!"] & \Over{\I{C}}{c}
	\end{tikzcd}
	\end{equation*}
	(where $\alpha\colon d\to c$ is the image of $f$ along $p$) is a pullback. The results in this section can therefore be dualised to cartesian fibrations, with cartesian maps in place of cocartesian maps.
\end{remark}

\section{The marked model for cocartesian fibrations}
\label{sec:markedModel}
As opposed to left fibrations between $\BB$-categories, cocartesian fibration do not arise as the right complement of a factorisation system in $\Cat(\BB)$. In order to rectify this, one must treat cocartesian maps as extra data. This naturally leads us to the study of \emph{marked simplicial objects} in $\BB$, which are an internal (and higher-categorical) analogue of marked simplicial sets as studied in~\cite[\S~3.1]{htt}. In \S~\ref{sec:markedSimplicialObjects} we introduce the $\infty$-topos of marked simplicial objects in $\BB$ and study its basic properties. In \S~\ref{sec:markedAnodyne} and \S~\ref{sec:markedCocartesianFibrations}, we study the factorisation system in the $\infty$-topos of marked simplicial objects that gives rise to the desired model for cocartesian fibrations. In \S~\ref{sec:markedLeftFibrations}, we discuss how left fibrations can be recovered in the marked model. Finally, \S~\ref{sec:markedProper} features a discussion of the notion of marked \emph{proper} and marked \emph{smooth} maps, which will be important for our study of the $\BB$-category of cocartesian fibrations over a fixed base $\BB$-category.

\subsection{Marked simplicial objects}
\label{sec:markedSimplicialObjects}
Recall that we denote by $\sigma_0\colon\ord{1}\to\ord{0}$ the unique map in $\Delta$. Let $\sigma_0\colon \Delta^1\to \Delta$ be the associated functor that picks out this map, and let $\Delta_+$ be the $\infty$-category that arises as the pushout
\begin{equation*}
\begin{tikzcd}
\Delta^1\arrow[d, "\sigma_0"]\arrow[r, hookrightarrow, "d_{\{0,2\}}"] & \Delta^2\arrow[d, "\nu"]\\
\Delta\arrow[r, hookrightarrow, "\iota"] & \Delta_+
\end{tikzcd}
\end{equation*}
in $\CatS$. The fact that the functor $\iota\colon\Delta\into \Delta_+$ is fully faithful follows from~\cite[Lemma~6.3.9]{Martini2021a}. Note that an object in $\Delta_+$ is either of the form $\iota\ord{n}$ for some $\ord{n}\in \Delta$ or the image of $\{1\}\in\Delta^2$ along the map $\nu\colon\Delta^2\to \Delta_+$, which we will denote by $+$. 

A priori, the $\infty$-category $\Delta_+$ need not be a $1$-category, so it is not clear that this definition recovers the usual marked simplex $1$-category (which can be defined as the homotopy $1$-category of $\Delta_+$, i.e.\ as the pushout $\Delta\sqcup_{\Delta^1}\Delta^2$ in $\Cat_1$). However, the following lemma shows that this is nonetheless the case.
\begin{lemma}
	\label{lem:markedSimplexOrdinaryCategory}
	The $\infty$-category $\Delta_+$ is a $1$-category.
\end{lemma}
The proof of Lemma~\ref{lem:markedSimplexOrdinaryCategory} requires some machinery that is not needed for the rest of this paper and is therefore deferred to appendix~\ref{sec:appA}.

\begin{definition}
\label{def:markedSimplicialObjects}
A \emph{marked simplicial object} in $\BB$ is a functor $\Delta_+^{\op}\to\BB$. We denote by $\mSimp\BB = \Fun(\Delta_+^{\op},\BB)$ the $\infty$-topos of marked simplicial objects in $\BB$.
\end{definition}

\begin{remark}
	\label{tensoringCotensoringMarked}
	Analogous to the case of simplicial objects in $\BB$, postcomposition with the global sections functor induces a geometric morphism $\Gamma\colon\mSimp\BB\to\mSimp\SS$, which in particular implies that the $\infty$-topos $\mSimp\BB$ is tensored and powered over $\mSimp\SS$. For $\ord{n}\in\Delta$, we will denote by $\Delta_+^n$ the object in $\mSimp\SS$ that is represented by $\iota\ord n$, and we will denote by $\Delta_+^+$ the marked simplicial $\infty$-groupoid that is represented by $+$. As usual, we will implicitly identify such marked simplicial $\infty$-groupoids with the associated constant marked simplicial objects in $\BB$. Note that analogously as in the case of simplicial objects in $\BB$, the identity functor on $\mSimp\BB$ is equivalent to the composition $\ev_0\circ(-)^{\Delta^\bullet_+}$. In other words, for every marked simplicial object $P$ in $\BB$ there is an equivalence $P_\bullet\simeq (P^{\Delta^\bullet_+})_0$ which is natural in $P$.
\end{remark}

Observe that the functor $d_{\{0,2\}}$ admits both a left adjoint $s_{\{1,2\}}$ and a right adjoint $s_{\{0,1\}}$. The pushout of these adjoints along $\Delta^1\into\Delta$ then define a left adjoint $\flat\colon \Delta_+\to\Delta$ and a right adjoint $\sharp\colon\Delta_+\to\Delta$ to the inclusion $\iota$. This follows from the following lemma (and its dual version):
\begin{lemma}
If
\begin{equation*}
\begin{tikzcd}
	\I{C}\arrow[d, hookrightarrow, "f"] \arrow[r, "h"] & \I{E}\arrow[d, "g", hookrightarrow]\\
	\I{D}\arrow[r, "k"] & \I{F}
\end{tikzcd}
\end{equation*}
is a pushout square in $\Cat(\BB)$ such that $f$ is fully faithful and admits a left adjoint $l$, then the pushout of $l$ along $k$ defines a left adjoint of $g$.
\end{lemma}
\begin{proof}
Let $l\colon \I{D}\to\I{C}$ be the left adjoint of $f$ and let $\eta\colon \Delta^1\otimes \I{D}\to\I{D}$ be the adjunction unit. We define $l^\prime\colon \I{F}\to\I{E}$ to be the pushout of $l$ along $k$. Consider the commutative diagram
\begin{equation*}
\begin{tikzcd}[column sep={4.7em,between origins}, row sep={2.5em,between origins}]
& \Delta^1\otimes \I{E}\arrow[rr, dashed, "\id\otimes g"]\arrow[dd, "s^0", near end] \arrow[rrrr, bend left=20, "s^0\otimes g"]&& \Delta^1\otimes\I{F}\arrow[dd, "\eta^\prime", near end, dashed]\arrow[rr, "s^0\otimes\id", dashed] && \I{F}\arrow[dd, "l^\prime"]\\
\Delta^1\otimes\I{C}\arrow[rr, "\id\otimes f", crossing over, near end]\arrow[dd, "s^0"]\arrow[ur, "\id\otimes h"] && \Delta^1\otimes\I{D}\arrow[ur, "\id\otimes k",near start, dashed]\arrow[rr, crossing over, "s^0\otimes\id", near end]&& \I{D}\arrow[ur, "k"]&\\
& \I{E} \arrow[rr, "g", near start]&& \I{F}\arrow[rr, "l^\prime", near start]&& \I{E}\\
\I{C}\arrow[ur, "h"]\arrow[rr, "f"] && \I{D}\arrow[from=uu, crossing over, near start, "\eta"] \arrow[ur, "k"]\arrow[rr, "l"]&&\I{C}\arrow[from=uu, crossing over, "l", near start] \arrow[ur, "h"]&
\end{tikzcd}
\end{equation*}
which is constructed as the left Kan extension of its solid part.
Since $\eta\circ (d^1\otimes \id)$ is equivalent to the identity and $\eta\circ (d^0\otimes \id)$ is equivalent to $fl$, the same must be true when replacing $\eta$ by $\eta^\prime$ and $fl$ by $gl^\prime$. In other words, $\eta^\prime$ encodes a map $\id\to gl^\prime$. The two squares in the back of the above diagram now precisely express the two conditions that $\eta^\prime k$ and $l^\prime\eta^\prime$ are equivalences. Using~\cite[Corollary~3.4.3]{Martini2021a}, this shows that $l^\prime$ is left adjoint to $g$.
\end{proof}

By precomposition, the restriction functor $(-)\vert_{\Delta}=\iota^\ast\colon \mSimp{\BB}\to\Simp\BB$ admits both a left adjoint $(-)^\flat$ and a right adjoint $(-)^\sharp$, both of which are fully faithful. We denote by $(-)_\sharp$ the right adjoint of $(-)^\sharp$ that is given by right Kan extension along $\sharp$. There is also a further left adjoint $(-)_\flat$ of $(-)^\flat$ given by left Kan extension along $\flat$, but we will not need this functor.
Note that applying the unit of the adjunction $(-)\vert_{\Delta}\dashv (-)^\sharp$ to $(-)^\flat$ gives rise to a canonical morphism $(-)^\flat\to (-)^\sharp$. Explicitly, this map is given by precomposition with $\flat \epsilon\colon \sharp\simeq \flat \iota \sharp\to \flat$, where $\epsilon$ is the counit of the adjunction $\iota\dashv \sharp$.
\begin{remark}
	\label{rem:FlatSharpRestrictionGroupoids}
	Since the map $\flat\epsilon\colon\sharp\to\flat$ evaluates to the identity on $\ord{0}$, the natural morphism $(-)^\flat\to (-)^\sharp$ is an equivalence when restricted to $\BB\into\Simp\BB$.
\end{remark}

Observe that the fact that $(-)^\flat$ is left adjoint to $(-)\vert_{\Delta}$ and therefore equivalent to the functor of left Kan extension $\iota_!$ implies that there is a canonical equivalence $\Delta_+^\bullet\simeq (\Delta^\bullet)^\flat$ of functors $\Delta\to \mSimp\BB$. We will also need to identify the marked simplicial object $\Delta^+_+$. To that end, observe that there is an equivalence $(\Delta_+^+)\vert_{\Delta}\simeq\Delta^1$ and therefore a canonical morphism $\Delta_+^+\to (\Delta^1)^\sharp$ in $\mSimp\BB$.
\begin{lemma}
	\label{lem:walkingMarkedEdge}
	The map $\Delta^+_+\to (\Delta^1)^\sharp$ is an equivalence.
\end{lemma}
\begin{proof}
	We can assume $\BB\simeq\SS$. By construction, the restriction functor $(-)\vert_{\Delta}$ carries the map $\Delta^+_+\to (\Delta^1)^\sharp$ to an equivalence. As equivalences in $\mSimp\SS$ are detected object-wise, it therefore suffices to show that the evaluation of this map at $+\in\Delta_+$ is an equivalence as well. On account of Yoneda's lemma, this amounts to showing that the morphism
	\begin{equation*}
	\map{\Delta_+}(+,+)\to\map{\Delta}(\ord{1},\ord{1})
	\end{equation*}
	that is induced by the action of the functor $\sharp$ on mapping $\infty$-groupoids is an equivalence. In light of the explicit computation of $\map{\Delta_+}(+,+)$ in appendix~\ref{sec:appA}, this is immediate.
	\end{proof}
\begin{remark}
	\label{rem:identificationMapFlatSharp}
	The canonical map $\Delta^1_+\to\Delta^+_+$ gives rise to a commutative diagram
	\begin{equation*}
	\begin{tikzcd}
	(\Delta^1_+)\vert_\Delta^\flat\arrow[r] \arrow[d] & \Delta^1_+\arrow[r]\arrow[d] & (\Delta^1_+)\vert_{\Delta}^\sharp\arrow[d]\\
	(\Delta^+_+)\vert_{\Delta}^\flat\arrow[r]& \Delta^+_+\arrow[r]& (\Delta^+_+)\vert_{\Delta}^\sharp
	\end{tikzcd}
	\end{equation*}
	in which the two horizontal maps on the left are given by the counit of the adjunction $(-)^\flat\dashv (-)\vert_{\Delta}$ and the ones on the right are given by the unit of the adjunction $(-)\vert_{\Delta}\dashv (-)^\sharp$. As $\Delta^1_+$ is in the essential image of $(-)^\flat$, the upper left horizontal map is an equivalence, and by Lemma~\ref{lem:walkingMarkedEdge} the lower right horizontal map is an equivalence too. Hence the morphism $(\Delta^1)^\flat\to(\Delta^1)^\sharp$ recovers the canonical map $\Delta^1_+\to\Delta^+_+$ upon identifying $(\Delta^1)^\flat\simeq \Delta^1_+$ and $\Delta^+_+\simeq (\Delta^1)^\sharp$.
\end{remark}

\subsection{Marked left anodyne morphisms}
\label{sec:markedAnodyne}
The goal of this section is to construct a saturated class of maps in $\mSimp\BB$ whose right complement ought to model cocartesian fibrations. Our approach is in large parts an adaptation of Lurie's construction of the cocartesian model structure in~\cite[\S~3.1]{htt}, but as we work internally there will be some deviations. In particular, the generators that we list in~\ref{def:markedAnodyne} are slightly different from the class of marked anodyne morphisms as defined in~\cite[Definition~3.1.1.1]{htt}.
\begin{definition}
\label{def:markedAnodyne}
	A map in $\mSimp{\BB}$ is said to be \emph{marked left anodyne} if it is contained in the internal saturation of the following collection of maps:
	\begin{enumerate}
	\item $(I^2)^\flat\into(\Delta^2)^\flat$;
	\item $(E^1)^\flat\to 1$;
	\item $(\Delta^1)^\sharp\sqcup_{(\Delta^1)^\flat}(\Delta^1)^\sharp\to (\Delta^1)^\sharp$;
	\item $d^1\colon 1\into(\Delta^1)^\sharp$.
\end{enumerate}
\end{definition}

For practical purposes, we will need a slightly smaller set of generators for the collection of marked left anodyne maps. In what follows, we shall adopt Jay Shah's notation in~\cite{Shah2018} and let $\prescript{}{\natural}(\Delta^n)^\flat=(\Delta^1)^{\sharp}\sqcup_{(\Delta^1)^\flat}(\Delta^n)^\flat$ denote the pushout of $(\Delta^1)^\flat\to (\Delta^1)^\sharp$ along $d^{\{0,1\}}\colon(\Delta^1)^\flat\into(\Delta^n)^\flat$ for every $n\geq 2$. We will use the same notation for any subobject of $\Delta^n$ that contains the edge $\{0,1\}$.
\begin{proposition}
\label{prop:markedAnodyneGenerators}
	A map in $\mSimp{\BB}$ is marked left anodyne if and only if it is contained in the saturation of the following collection of maps:
	\begin{enumerate}
	\item $(I^2\otimes K)^\flat\into(\Delta^2\otimes K)^\flat$ for all $K\in\Simp\BB$;
	\item $(E^1\otimes K)^\flat\to K^{\flat}$ for all $K\in\Simp\BB$;
	\item $(\Delta^1\otimes A)^\sharp\sqcup_{(\Delta^1\otimes A)^\flat}(\Delta^1\otimes A)^\sharp\to (\Delta^1\otimes A)^\sharp$ for all $A\in\BB$;
	\item $\prescript{}{\natural}(\Lambda^2_0)^\flat\otimes A\into \prescript{}{\natural}(\Delta^2)^\flat\otimes A$ for all $A\in\BB$;
	\item $d^1\colon A^\sharp\into(\Delta^1\otimes A)^\sharp$ for all $A\in\BB$;
	\item $(I^2\otimes A)^\sharp\into (\Delta^2\otimes A)^\sharp$ for all $A\in\BB$.
\end{enumerate}
\end{proposition}
We will spread out the proof of Proposition~\ref{prop:markedAnodyneGenerators} over the two combinatorial lemmas~\ref{lem:markedAnodyneGenerators1} and~\ref{lem:markedAnodyneGenerators2}. Both of them will make repeated use of the following basic observation:

\begin{lemma}
	\label{lem:pushoutProductEquivalenceMarked}
	Let
	\begin{equation*}
	\begin{tikzcd}
	K\arrow[r]\arrow[d, "f"] & M\arrow[d, "g"]\\
	L\arrow[r] & N
	\end{tikzcd}
	\end{equation*}
	be a commutative square in $\mSimp\BB$ such that $f\vert_{\Delta}$ and $g\vert_{\Delta}$ are equivalences. Then the square is a pushout if and only if it becomes a pushout after evaluation at $+\in\Delta_+$. 
	In particular, if $C\to D$ is a map in $\Simp\BB$, the map $C^\sharp\sqcup_{C^\flat}D^\flat\to D^\sharp$ is an equivalence if and only if $C_1\sqcup_{C_0}D_0\to D_1$ is an equivalence. An analogous result holds for pullbacks.
\end{lemma}
\begin{proof}
	The square is a pushout if and only if its evaluation at each object in $\Delta_+$ is a pushout in $\BB$. Since all but the object $+\in\Delta_+$ are contained in the essential image of the inclusion $\Delta\into\Delta_+$ and since the functor $(-)\vert_{\Delta}$ by assumption carries the vertical maps to equivalences, the first claim follows. As for the second claim, it suffices to observe that the map $C_1\sqcup_{C_0}D_0\to D_1$ is precisely the evaluation of $C^\sharp\sqcup_{C^\flat}D^\flat\to D^\sharp$ at $+\in\Delta_+$.
\end{proof}

\begin{lemma}
	\label{lem:markedAnodyneGenerators1}
	The internal saturation of the maps in~(1)--(3) in Definition~\ref{def:markedAnodyne} is equal to the saturation of the following maps:
	\begin{enumerate}
	\item $(I^2\otimes K)^\flat\into(\Delta^2\otimes K)^\flat$ for all $K\in\Simp\BB$;
	\item $(E^1\otimes K)^\flat\to K^{\flat}$ for all $K\in\Simp\BB$;
	\item $(\Delta^1\otimes A)^\sharp\sqcup_{(\Delta^1\otimes A)^\flat}(\Delta^1\otimes A)^\sharp\to (\Delta^1\otimes A)^\sharp$ for all $A\in\BB$.
	\end{enumerate}
\end{lemma}
\begin{proof}
	 Let $S$ be the saturation of the maps in~(1)--(3) in the lemma. As the internal saturation of the maps in~(1)--(3) in Definition~\ref{def:markedAnodyne} clearly contains S, it suffices to prove the converse direction. We need to show that for every marked simplicial object $K$, the map $f\otimes\id_K$ is contained in $S$, where $f$ is one of the maps in~(1)--(3) in Definition~\ref{def:markedAnodyne}. As every marked simplicial object can be obtained as a small colimit of objects of the form $\Delta_+^n\otimes A$, where $A\in\BB$ and either $n\geq 0$ or $n=+$, we need only show this for $K\in\{(\Delta^n)^\flat~\vert~n\geq 0\}\cup\{(\Delta^1)^\sharp\}$. There are therefore six cases:
	 	\begin{enumerate}
	 	\item For $n\geq 0$, the map $(I^2\times \Delta^n)^\flat \into (\Delta^2\times \Delta^n)^\flat$ is by definition contained in $S$.
	 	\item In order to show that the map $(I^2)^\flat\times (\Delta^1)^\sharp\into (\Delta^2)^\flat\times(\Delta^1)^\sharp$ is contained in $S$, it suffices to show that the map $((\Delta^2)^\flat\times(\Delta^1)^\sharp)\sqcup_{(I^2)^\flat\times(\Delta^1)^\sharp}((I^2)^\flat\times (\Delta^1)^\sharp)\to (\Delta^2)^\flat\times(\Delta^1)^\sharp$ is contained in $S$. Using Lemma~\ref{lem:pushoutProductEquivalenceMarked}, one easily verifies that this map is an equivalence.
	 	\item The maps $(E^1\times \Delta^n)^\flat \to (\Delta^n)^\flat$ are by definition contained in $S$.
	 	\item Consider the commutative diagram
	 	\begin{equation*}
	 	\begin{tikzcd}[column sep={6.5em,between origins}, row sep={3em,between origins}]
	 		& (\Delta^1\sqcup\Delta^1)^\flat\arrow[rr]\arrow[dd] && (\Delta^1\sqcup\Delta^1)^\flat\arrow[dd]\\
	 		(\Delta^1\sqcup\Delta^1)^\flat\arrow[ur, "\id"]\arrow[rr, crossing over]\arrow[dd] && (\Delta^1\sqcup\Delta^1)^\sharp \arrow[ur, "\id"]&\\
	 		& (\Delta^1)^\flat\arrow[rr] && (\Delta^1)^\flat\sqcup_{(\Delta^1\sqcup\Delta^1)^\flat}(\Delta^1\sqcup\Delta^1)^\sharp\\
	 		(E^1\times\Delta^1)^\flat\arrow[rr] \arrow[ur]&& (E^1)^\flat\times(\Delta^1)^\sharp\arrow[ur, "\phi"']\arrow[from=uu, crossing over] &
	 	\end{tikzcd}
	 	\end{equation*}
	 	in which the two vertical maps in the front square are induced by the inclusion of the two points of $E^1$. Since Lemma~\ref{lem:pushoutProductEquivalenceMarked} implies that the front square in this diagram is a pushout, the map $\phi$ is obtained as a pushout of maps that are contained in $S$ and must therefore be in $S$ too. Hence, to show that $(E^1)^\flat\times(\Delta^1)^\sharp\to(\Delta^1)^\sharp$ is contained in $S$, it suffices to show that $(\Delta^1)^\flat\sqcup_{(\Delta^1\sqcup\Delta^1)^\flat}(\Delta^1\sqcup\Delta^1)^\sharp\to (\Delta^1)^\sharp$ is in $S$, which follows from the observation that this is precisely the map in~(3) in the case where $A\simeq 1$.
	 	\item	Let us set $K= (\Delta^1)^\sharp\sqcup_{(\Delta^1)^\flat}(\Delta^1)^\sharp$. We have a commutative diagram
	 	\begin{equation*}
	 		\begin{tikzcd}[column sep={8.5em,between origins}, row sep={3.5em,between origins}]
	 			& \bigsqcup_{i\in \ord{n}}(\Delta^1)^\flat \sqcup \bigsqcup_{i\in \ord{n}}(\Delta^1)^\flat \arrow[rr]\arrow[dl]\arrow[dd] && \bigsqcup_{i\in \ord{n}}(\Delta^1)^\sharp \sqcup \bigsqcup_{i\in \ord{n}}(\Delta^1)^\sharp\arrow[dl]\arrow[dd]\\
	 			(\Delta^1\times\Delta^n)^\flat\sqcup (\Delta^1\times\Delta^n)^\flat\arrow[rr, crossing over]\arrow[dd] && ((\Delta^1)^\sharp\times(\Delta^n)^\flat)\sqcup ((\Delta^1)^\sharp\times(\Delta^n)^\flat)\arrow[dd] & \\
	 			&\bigsqcup_{i\in \ord{n}}(\Delta^1)^\flat \arrow[rr] \arrow[dl] && \bigsqcup_{i\in \ord{n}} K\arrow[dl]\\
	 			(\Delta^1\times\Delta^n)^\flat\arrow[rr] && K\times(\Delta^n)^\flat\arrow[from=uu, crossing over]
	 		\end{tikzcd}
	 	\end{equation*}
	 	in which both the front and the back square is a pushout. Using Lemma~\ref{lem:pushoutProductEquivalenceMarked}, one moreover easily verifies that the top square is a pushout too, which implies that the bottom square is one as well. As a consequence, we obtain a commutative diagram
	 	\begin{equation*}
	 		\begin{tikzcd}
	 		\bigsqcup_{i\in \ord{n}} (\Delta^1)^\flat\arrow[r]\arrow[d] & \bigsqcup_{i\in \ord{n}} K\arrow[r] \arrow[d]& \bigsqcup_{i\in \ord{n}} (\Delta^1)^\sharp\arrow[d]\\
	 		(\Delta^1\times\Delta^n)^\flat\arrow[r] & K\times(\Delta^n)^\flat\arrow[r] & (\Delta^1)^\sharp\times(\Delta^n)^\flat
	 		\end{tikzcd}
	 	\end{equation*}
	 	in which the left square is cocartesian. Since the map $K\to (\Delta^1)^\sharp$ is contained in $S$, we conclude that the map $K\times(\Delta^n)^\flat\to (\Delta^1)^\sharp\times(\Delta^n)^\flat$ is an element of $S$ whenever the right square is a pushout diagram. This follows from the observation that the outer square of this diagram is cocartesian, which is easily verified using Lemma~\ref{lem:pushoutProductEquivalenceMarked}.
	 	\item Let again $K= (\Delta^1)^\sharp\sqcup_{(\Delta^1)^\flat}(\Delta^1)^\sharp$ and consider the commutative diagram
	 	\begin{equation*}
	 		\begin{tikzcd}
	 		(\Delta^1\sqcup\Delta^1)^\flat\arrow[r]\arrow[d] & K\times(\Delta^1)^\flat\arrow[d] \arrow[r] & (\Delta^1\times\Delta^1)^\flat\arrow[d]\\
	 		(\Delta^1\sqcup\Delta^1)^\sharp\arrow[r] & K\times(\Delta^1)^\sharp\arrow[r] & (\Delta^1)^\flat\times(\Delta^1)^\sharp
	 		\end{tikzcd}
	 	\end{equation*}
	 	in which the two horizontal maps on the left are induced by the inclusion of the two points of $K_0$. Using Lemma~\ref{lem:pushoutProductEquivalenceMarked}, one finds that the composite square is cocartesian, and the fact that the two horizontal maps on the left induce an equivalence when evaluated at $+\in \Delta_+$ similarly implies that the left square is a pushout too. We thus conclude that the right square is cocartesian. As the upper right horizontal morphism is contained in $S$, this shows that the map $K\times(\Delta^1)^\sharp\to (\Delta^1)^\flat\times(\Delta^1)^\sharp$ is in $S$ as well.\qedhere
	 	\end{enumerate}
\end{proof}

\begin{remark}
	\label{rem:markedMonomorphismGenerator}
	Note that the internal saturation of the maps in~(1)--(3) in Definition~\ref{def:markedAnodyne} also contains the map $K^\sharp\sqcup_{K^\flat} K^\sharp\to K^\sharp$ for every simplicial object $K$. In fact, this follows from the observation that this map arises as a retract of $(\Delta^1\otimes K)^\sharp\sqcup_{(\Delta^1\otimes K)^\flat}(\Delta^1\otimes K)^\sharp\to(\Delta^1\otimes K)^\sharp$.
\end{remark}

\begin{lemma}
	\label{lem:markedAnodyneGenerators2}
	Let $S$ be a saturated class of maps in $\mSimp\BB$ that contains the internal saturation of the maps in~(1)--(3) in Definition~\ref{def:markedAnodyne}. Then $S$ contains the internal saturation of $d^1\colon 1\into (\Delta^1)^\sharp$ if and only if it contains the following maps:
	\begin{enumerate}
	\item $\prescript{}{\natural}(\Lambda^2_0)^\flat\otimes A\into \prescript{}{\natural}(\Delta^2)^\flat\otimes A$ for all $A\in\BB$;
	\item $d^1\colon A^\sharp\into(\Delta^1\otimes A)^\sharp$ for all $A\in\BB$;
	\item $(I^2\otimes A)^\sharp\into (\Delta^2\otimes A)^\sharp$ for all $A\in\BB$.
	\end{enumerate}
\end{lemma}
\begin{proof}
	Suppose first that $S$ contains the internal saturation of $d^1\colon 1\into (\Delta^1)^\sharp$. There are now three cases to consider:
	\begin{enumerate}
	\item  Let $K\to L$ be the unique map in $\mSimp\BB$ that fits into the diagram
	\begin{equation*}
	\begin{tikzcd}[column sep={3.5em,between origins},row sep={2.5em,between origins}]
	& (\Lambda^2_0)^\flat\arrow[dd, "d^1", near start, hookrightarrow]\arrow[rr, hookrightarrow]\arrow[dl] && (\Delta^2)^\flat\arrow[dd, hookrightarrow]\arrow[dl]\arrow[ddrr, "d^1", bend left, hookrightarrow] &&\\
	\prescript{}{\natural}{(\Lambda^2_0)^\flat}\arrow[rr, crossing over, hookrightarrow] \arrow[dd, "d^1", near start, hookrightarrow]&& \prescript{}{\natural}{(\Delta^2)^\flat}\arrow[ddrr, "d^1", bend left, hookrightarrow]&& &\\ 
	& (\Delta^1\times\Lambda^2_0)^\flat\arrow[rr, hookrightarrow] \arrow[dl]&& K\arrow[dl]\arrow[rr, hookrightarrow] && (\Delta^1\times\Delta^2)^\flat \arrow[dl]\\
	(\Delta^1)^\sharp\times\prescript{}{\natural}{(\Lambda^2_0)^\flat}\arrow[rr, hookrightarrow] && L\arrow[from=uu, crossing over, hookrightarrow]\arrow[rr, hookrightarrow]&& (\Delta^1)^\sharp\times\prescript{}{\natural}{(\Delta^2)^\flat}&
	\end{tikzcd}
	\end{equation*}
	such that both the front and the back square is a pushout.
	Then the map $L\to (\Delta^1)^\sharp\times\prescript{}{\natural}{(\Delta^2)^\flat}$ is contained in $S$. We claim that the inclusion $\prescript{}{\natural}(\Lambda^2_0)^\flat\into \prescript{}{\natural}(\Delta^2)^\flat$ is a retract of this map. To see this, first note that the two squares on the bottom of the above diagram are cocartesian by Lemma~\ref{lem:pushoutProductEquivalenceMarked}. Now let $r\colon \Delta^1\times\Delta^2\to\Delta^2$ be the map given by $r(0,1)=0$ and $r(k,l)=l$ else. The $r^\flat$ restricts to a map $(r^\prime)^\flat\colon K\to(\Lambda^2_0)^\flat$. We obtain a commutative diagram
	\begin{equation*}
	\begin{tikzcd}[column sep={between origins,4em}, row sep={between origins,3em}]
	& (\Lambda^2_0)^\flat\arrow[rr, "d^0"]\arrow[dd]\arrow[dl] && K\arrow[rr, "(r^\prime)^\flat"]\arrow[dd]\arrow[dl] && (\Lambda^2_0)^\flat\arrow[dd]\arrow[dl]\\
	(\Delta^2)^\flat\arrow[rr, crossing over, "d^0"]\arrow[dd, dashed] && (\Delta^1\times\Delta^2)^\flat\arrow[rr, "r^\flat"] && (\Delta^2)^\flat &\\
	&\prescript{}{\natural}{(\Lambda^2_0)^\flat} \arrow[rr, "d^0", near start] \arrow[dl, dashed]&& L\arrow[rr]\arrow[dl, dashed] && \prescript{}{\natural}{(\Lambda^2_0)^\flat}\arrow[dl, dashed]\\
	\prescript{}{\natural}{(\Delta^2)^\flat}\arrow[rr, dashed]&& (\Delta^1)^\sharp\times \prescript{}{\natural}{(\Delta^2)^\flat}\arrow[rr, dashed] \arrow[from=uu, crossing over, dashed] && \prescript{}{\natural}{(\Delta^2)^\flat}\arrow[from=uu, crossing over, dashed] &
	\end{tikzcd}
	\end{equation*}
	in which the upper row is a retract diagram. Since the lower row is obtained as a pushout of the upper row, the claim follows. As a consequence, the maps in~(1) can be realised as retracts of maps in $S$, which shows that they too must be contained in $S$. 
	
	\item The maps $d^1\colon A^\sharp\into (\Delta^1\otimes A)^\sharp$ are by definition contained in $S$.
	\item Note that the map $d^{\{0,1\}}\colon (\Delta^1\otimes A)^\sharp\into (I^2\otimes A)^\sharp$ is a pushout of $d^1\colon A^\sharp\into(\Delta^1\otimes A)^\sharp$ and therefore contained in $S$. Hence, to show that the maps in~(3) are in $S$, it suffices to prove that the map $d^{\{0,1\}}\colon (\Delta^1\otimes A)^\sharp\into(\Delta^2\otimes A)^\sharp$ is an element of $S$. This in turn follows from the observation that this map is a retract of the morphism $d^1\colon (\Delta^1\otimes A)^\sharp\into(\Delta^1\otimes(\Delta^1\otimes A))^\sharp$.
	\end{enumerate}
	We now show the converse inclusion. As in the proof of Lemma~\ref{lem:markedAnodyneGenerators1}, we only need to show that the map $d^1\colon K\into (\Delta^1)^\sharp\otimes K$ is contained in $S$ for $K\in\{(\Delta^n\otimes A)^\flat~\vert~ n\geq 0,~A\in\BB\}\cup\{(\Delta^1\otimes A)^\sharp~\vert~A\in\BB\}$. As $d^1\colon A^\sharp\into(\Delta^1\otimes A)^\sharp$ is contained in $S$, we can replace $\BB$ by $\Over{\BB}{A}$ and therefore always assume $A\simeq 1$. Moreover, since $S$ by assumption contains the internal saturation of $(I^2)^\flat\into(\Delta^2)^\flat$, the map $(I^n)^\flat\otimes K\into (\Delta^n)^\flat\otimes K$ is in $S$ too, for every integer $n\geq 2$~\cite[Lemma~3.2.5]{Martini2021}. Thus, if $f\in S$ is an arbitrary map such that $\id_{(\Delta^1)^\flat}\otimes f$ is contained in $S$, the map $\id_{(\Delta^n)^\flat}\otimes f$ must be in $S$ too for every integer $n\geq 0$. In total, these considerations allow us to assume $K\in\{(\Delta^1)^\flat,(\Delta^1)^\sharp\}$. There are therefore two cases:
	\begin{enumerate}
		\item To show that $d^1\colon (\Delta^1)^\flat\into(\Delta^1)^\sharp\times(\Delta^1)^\flat$ is contained in $S$, first note that the codomain of this map is given by the pushout
		\begin{equation*}
		\begin{tikzcd}
		(\Delta^1\sqcup\Delta^1)^\flat\arrow[r]\arrow[d, "{(d^1\times\id, d^0\times\id)}"] & (\Delta^1\sqcup\Delta^1)^\sharp\arrow[d]\\
		(\Delta^1\times\Delta^1)^\flat\arrow[r] & (\Delta^1)^\flat\times(\Delta^1)^\sharp.
		\end{tikzcd}
		\end{equation*}
		Therefore, by using the decomposition $\Delta^1\times\Delta^1\simeq\Delta^2\sqcup_{\Delta^1}\Delta^2$, we obtain an equivalence of marked simplicial objects $(\Delta^1)^\flat \times(\Delta^1)^\sharp\simeq H\sqcup_{(\Delta^1)^\flat}K$, where $H$ and $K$ are defined as the pushouts
		\begin{equation*}
		\begin{tikzcd}
		(\Delta^1)^\flat\arrow[r]\arrow[d, hookrightarrow, "d^{\{1,2\}}"] & (\Delta^1)^\sharp\arrow[d] && (\Delta^1)^\flat\arrow[r] \arrow[d, "d^{\{0,1\}}"]& (\Delta^1)^\sharp\arrow[d]\\
		(\Delta^2)^\flat\arrow[r] & H && (\Delta^2)^\flat\arrow[r] & K.
		\end{tikzcd}
		\end{equation*}
		With respect to this identification, the inclusion $d^1\colon (\Delta^1)^\flat\into(\Delta^1)^\sharp\times(\Delta^1)^\flat$ is obtained by the composition
		\begin{equation*}
		(\Delta^1)^\flat\xhookrightarrow{d^{\{0,1\}}} H\into H\sqcup_{(\Delta^1)^\flat} K.
		\end{equation*}
		It therefore suffices to show that both $d^{\{0,1\}}\colon (\Delta^1)^\flat\into H$ and $d^{\{0,2\}}\colon (\Delta^1)^\flat\into K$ are contained in $S$. We begin with the first map. Observe that this morphism is equivalent to the composition
		\begin{equation*}
		(\Delta^1)^\flat \xhookrightarrow{d^{\{0,1\}}} (\Delta^1)^\sharp\sqcup_{(\Delta^1)^\flat}(I^2)^\flat\\into (\Delta^1)^\sharp\sqcup_{(\Delta^1)^\flat}(\Delta^2)^\flat.
		\end{equation*}
		Here the right map is of the form~(1) in Definition~\ref{def:markedAnodyne} and therefore included in $S$. The left map, on the other hand, is obtained as a pushout of $d^1\colon( \Delta^0)^\sharp\into (\Delta^1)^\sharp$, hence contained in $S$ too. In order to show that $d^{\{0,2\}}\colon (\Delta^1)^\flat\into K$ defines an element of $S$, it suffices to observe that this map can be obtained as the composition
		\begin{equation*}
		(\Delta^1)^\flat\xhookrightarrow{d^{\{0,2\}}"} \prescript{}{\natural}(\Lambda^2_0)^\flat\into \prescript{}{\natural}(\Delta^2)^\flat
		\end{equation*}
		in which the right map is of the form~(1) and therefore in $S$ and in which the left map is a pushout of $d^1\colon (\Delta^0)^\sharp\into (\Delta^1)^\sharp$, so contained in $S$ as well.
		
		\item Finally, we show that the map $d^1\colon(\Delta^1)^\sharp\into(\Delta^1\times\Delta^1)^\sharp$ is contained in $S$. On account of the commutative diagram
		\begin{equation*}
		\begin{tikzcd}[column sep={4em,between origins}, row sep={2.5em,between origins}]
		& (\Delta^1)^\sharp\arrow[rr, "d^{\{0,2\}}"]\arrow[dd, "d^{\{0,2\}}", near end] && (\Delta^2)^\sharp\arrow[dd]\\
		(\Delta^0)^\sharp\arrow[dd, "\id",near start]\arrow[rr, "d^1", crossing over, near end] \arrow[ur, "d^1"]&& (\Delta^1)^\sharp\arrow[ur, "d^{\{0,1\}}"] & \\
		& (\Delta^2)^\sharp\arrow[rr, hookrightarrow] && (\Delta^1\times\Delta^1)^\sharp\\
		(\Delta^0)^\sharp\arrow[rr, "d^1"] \arrow[ur, "d^{\{0\}}", near end]&& (\Delta^1)^\sharp\arrow[from=uu, "\id", crossing over, near start]\arrow[ur] &
		\end{tikzcd}
		\end{equation*}
		in which both the front and the back square is a pushout and in which the map $d^1\colon (\Delta^0)^\sharp\into(\Delta^1)^\sharp$ is contained in $S$, it suffices to show that the two maps $d^{\{0\}}\colon (\Delta^0)^\sharp\into(\Delta^2)^\sharp$ and $d^{\{0,1\}}\colon(\Delta^1)^\sharp\into(\Delta^2)^\sharp$ are contained in $S$ as well.
		As the first of these two maps can be factored into $d^1\colon (\Delta^0)^\sharp\into(\Delta^1)^\sharp$ followed by $d^{\{0,1\}}\colon (\Delta^1)^\sharp\into(\Delta^2)^\sharp$, we only need to prove this for the second map. By in turn factoring this morphism as
		\begin{equation*}
		(\Delta^1)^\sharp\xhookrightarrow{d^{\{0,1\}}} (I^2)^\sharp\into (\Delta^2)^\sharp, 
		\end{equation*}
		this is a consequence of the observation that the map $(\Delta^1)^\sharp\into(I^2)^\sharp$ is obtained as a pushout of $d^1\colon(\Delta^0)^\sharp\into(\Delta^1)^\sharp$.\qedhere
	\end{enumerate}
\end{proof}

\begin{proof}[{Proof of Proposition~\ref{prop:markedAnodyneGenerators}}]
	Combine Lemma~\ref{lem:markedAnodyneGenerators1} and Lemma~\ref{lem:markedAnodyneGenerators2}.
\end{proof}

\subsection{Marked cocartesian fibrations}
\label{sec:markedCocartesianFibrations}
In this section we turn to studying the right complement of the class of marked left anodyne morphisms in $\mSimp\BB$.
\begin{definition}
	\label{def:cocartesianFibrationsMarked}
	A map in $\mSimp\BB$ is a \emph{marked cocartesian fibration} if it is right orthogonal to the class of marked left anodyne maps. We write $\Cocart^+\into\Fun(\Delta^1,\mSimp\BB)$ for the full cartesian subfibration over $\mSimp\BB$ that is spanned by the marked cocartesian fibrations.
\end{definition}

The following proposition shows that marked cocartesian fibrations faithfully generalise cocartesian fibrations of $\BB$-categories. The analogous result for cocartesian fibrations of $\infty$-categories appears as (the dual of)~\cite[Proposition~3.1.1.6]{htt}.
\begin{proposition}
	\label{prop:comparisonCocartesianFibrationMarkedUnmarked}
	For any $\BB$-category $\I{C}$, a map $p\colon P\to \I{C}^{\sharp}$ is a marked cocartesian fibration if and only if $P\vert_{\Delta}$ is a $\BB$-category, the map $p\vert_{\Delta}$ is a cocartesian fibration in $\Cat(\BB)$, and the map $P_+\to P_1$ is a monomorphism that identifies $P_+$ with the subobject of cocartesian morphisms of $p\vert_{\Delta}$.
\end{proposition}
\begin{proof}
	The map $p$ being right orthogonal to the maps of the form~(1) and~(2) in Proposition~\ref{prop:markedAnodyneGenerators} is equivalent to $P\vert_{\Delta}$ being a $\BB$-category.
	Moreover, $p$ is right orthogonal to the maps in~(3) in Proposition~\ref{prop:markedAnodyneGenerators} if and only if $P_+\to P_1$ is a monomorphism in $\BB$.
	
	Suppose now that $p$ is right orthogonal to the maps listed in~(1)--(3) in Proposition~\ref{prop:markedAnodyneGenerators}, and let us denote by $\I{P}=P\vert_{\Delta}$ the underlying $\BB$-category of $P$. The condition of $p$ to be right orthogonal to the maps in~(4) in Proposition~\ref{prop:markedAnodyneGenerators} is now equivalent to the commutative diagram
	\begin{equation*}
		\begin{tikzcd}
		P_+\times_{\I{P}_1}\I{P}_2\arrow[r] \arrow[d] & \I{C}_1\times_{\I{C}_1}\I{C}_2\arrow[d]\\
		P_+\times_{\I{P}_1}(\I{P}^{\Lambda^2_0})_0\arrow[r] & \I{C}_1\times_{\I{C}_1}(\I{C}^{\Lambda^2_0})_0
		\end{tikzcd}
	\end{equation*}
	to be cartesian. By employing Proposition~\ref{prop:characterisationCocartesianMorphismLifting}, this is equivalent to the condition that the inclusion $P_+\into \I{P}_1$ defines a cocartesian morphism in $\I{P}=P\vert_{\Delta}$. Therefore, if $p$ is in addition right orthogonal to the maps in~(5) in Proposition~\ref{prop:markedAnodyneGenerators}, we conclude from Proposition~\ref{prop:characterisationCocartesianFibrationCocartesianMaps} that $p\vert_{\Delta}$ must be a cocartesian fibration. Furthermore, under these conditions \emph{every} cocartesian map factors through $P_+\into \I{P}_1$. To see this, suppose that $f\colon x\to y$ is a cocartesian morphism in $\I{P}$ in context $A\in\BB$, and let $\alpha\colon c\to d$ be the image of $f$ along $p\vert_{\Delta}$. Using the maps in~(5) in Proposition~\ref{prop:markedAnodyneGenerators}, there exists a marked lift of $\alpha$, i.e.\ a map $g\colon x\to z$ in $\I{P}$ that is contained in $P_+\into \I{P}_1$ and that is sent to $\alpha$ by $p\vert_{\Delta}$. Since $g$ is marked and therefore cocartesian, Proposition~\ref{prop:characterisationCocartesianMorphismLifting} implies that one can find a map $h\colon z\to y$ in $\I{P}$ that is sent to $\id_d$ by $p\vert_{\Delta}$ such that $hg\simeq f$. This implies that $h$ must be cocartesian as well and therefore an equivalence. Thus $f$ is marked too, i.e\ contained in the image of $P_+\into \I{P}_1$.
	
	So far, we have shown that if $p$ is a marked cocartesian fibration, then the simplicial object $P\vert_{\Delta}$ is a $\BB$-category and $p\vert_{\Delta}$ is a cocartesian fibration such that $P_+\to  P_1$ is a monomorphism that identifes $P_+$ with the subobject of cocartesian maps in $P\vert_{\Delta}$. Conversely, if the map $p$ satisfies these conditions, the above argumentation shows that the proof is complete once we show that $p$ is right orthogonal to the maps in~(5) and~(6) in Proposition~\ref{prop:markedAnodyneGenerators}. Since orthogonality to the maps in~(5)
	precisely means that the map $P_+\to P_0\times_{C_0}C_1$ is an equivalence, this is immediate by the assumption that the inclusion $P_+\into P_1$ identifies $P_+$ with the subobject of cocartesian maps in $P\vert_{\Delta}$, cf.\ Remark~\ref{rem:characterisationCocartesianMorphisms}. Orthogonality to the maps in~(6), on the other hand, translates into the condition that the map $(P_\sharp)_2\to P_+\times_{P_0}P_+$ is an equivalence. To show this, first note that by Lemma~\ref{lem:pushoutProductEquivalenceMarked} the commutative square
	\begin{equation*}
	\begin{tikzcd}
	(\Delta^1\sqcup I^2)^\flat\arrow[r]\arrow[d] & (\Delta^1\sqcup I^2)^\sharp\arrow[d]\\
	(\Delta^2)^\flat\arrow[r] & (\Delta^2)^\sharp 
	\end{tikzcd}
	\end{equation*}
	is a pushout. Here the map $\Delta^1\sqcup I^2\to \Delta^2$ is given by $d^{\{0,2\}}$ on the first summand and by the canonical inclusion on the second one. One therefore obtains a pullback square
	\begin{equation*}
	\begin{tikzcd}
	(P_\sharp)_2\arrow[r]\arrow[d, hookrightarrow] & P_+\arrow[d, hookrightarrow]\\
	P_+\times_{P_0}P_+\arrow[r] &P_1
	\end{tikzcd}
	\end{equation*}
	in which the lower horizontal map is given by the composing the inclusion $P_+\times_{P_0}P_+\into P_1\times_{P_0}P_1\simeq P_2$ with $d_{\{0,2\}}\colon P_2\to P_1$. Since cocartesian maps in $P\vert_{\Delta}$ are closed under composition (see Remark~\ref{rem:stabilityCocartesianMaps}), we thus conclude that the lower horizontal map factors through the inclusion $P_+\into P_1$. This shows that the map $(P_\sharp)_2\into P_+\times_{P_0}P_+$ is an equivalence, as desired.
\end{proof}

As a consequence of Proposition~\ref{prop:comparisonCocartesianFibrationMarkedUnmarked}, the restriction of the functor $(-)\vert_{\Delta}\colon \Fun(\Delta^1,\mSimp\BB)\to\Fun(\Delta^1,\Simp\BB)$ along the inclusion $\Cocart^+\times_{\mSimp\BB}\Cat(\BB)\into\Fun(\Delta^1,\mSimp\BB)$ takes values in $\Cocart$ and therefore induces a functor
\begin{equation*}
(-)\vert_{\Delta}\colon \Cocart^+\times_{\mSimp\BB}\Cat(\BB)\to\Cocart
\end{equation*}
that commutes with the projections to $\Cat(\BB)$. Out next goal is to show:

\begin{proposition}
	\label{prop:cocartesianFibrationsMarkedUnmarkedEquivalence}
	The functor $(-)\vert_{\Delta}\colon \Cocart^+\times_{\mSimp\BB}\Cat(\BB)\to\Cocart$ is an equivalence.
\end{proposition}
The proof of Proposition~\ref{prop:cocartesianFibrationsMarkedUnmarkedEquivalence} will need the following lemma:
\begin{lemma}
	\label{lem:restrictionSimplicesMonomorphism}
	Given two presheaves $\sigma,\tau\in\PSh_{\BB}(\Delta^2)$ such that the map $\tau(1)\to\tau(0)$ is a monomorphism in $\BB$, the map
	\begin{equation*}
	\map{\PSh_{\BB}(\Delta^2)}(\sigma,\tau)\to\map{\PSh_{\BB}(\Delta^1)}(d_{\{0,2\}}^\ast\sigma, d_{\{0,2\}}^\ast\tau)
	\end{equation*}
	is a monomorphism in $\SS$ whose image consists of those maps $d_{\{0,2\}}^\ast\sigma\to d_{\{0,2\}}^\ast\tau$ for which the composition $\sigma(1)\to\sigma(0)\to\tau(0)$ takes values in $\tau(1)\into\tau(0)$.
\end{lemma}
\begin{proof}
	By making use of the adjunction $s_{\{0,1\}}^\ast\dashv d_{\{0,2\}}^\ast$, the map is equivalently given by postcomposition with the adjunction unit $\eta\colon \tau\to s_{\{0,1\}}^\ast d_{\{0,2\}}^\ast\tau$ which is explicitly given by the commutative diagram
	\begin{equation*}
	\begin{tikzcd}[row sep=small]
		& & \tau(0)\arrow[ddr, "\id"] & \\
		& \tau(1) \arrow[ur, hookrightarrow]& & \\
		& \tau(2) \arrow[rr]\arrow[uur]&& \tau(0)\\
		\tau(2) \arrow[ur, "\id"']\arrow[uur, crossing over]\arrow[rr]&& \tau(0). \arrow[ur, "\id"']\arrow[from=uul, crossing over]&
	\end{tikzcd}
	\end{equation*}
	Since $\tau(1)\into\tau(0)$ is by assumption a monomorphism, the entire map $\eta$ must be a monomorphism too. As a consequence, postcomposition with $\eta$ defines a monomorphism in $\SSS$, and it is clear from the description of $\eta$ that the image of this map is of the desired form.
\end{proof}

\begin{proof}[{Proof of Proposition~\ref{prop:cocartesianFibrationsMarkedUnmarkedEquivalence}}]
	We first show that the functor is fully faithful. To that end, let us fix two objects $p\colon P\to \I{C}^\sharp$ and $q\colon Q\to \I{D}^\sharp$ in $\Cocart^+\times_{\mSimp\BB}\Cat(\BB)$. We then obtain a pullback square
	\begin{equation*}
	\begin{tikzcd}
	\map{\Cocart^+}(p,q)\arrow[r]\arrow[d] & \map{\Fun(\Delta^1,\Simp{\BB})}(p\vert_{\Delta}, q\vert_{\Delta})\arrow[d]\\
	\map{\Fun(\Delta^1, \PSh_{\BB}(\Delta^2))}(\nu^\ast p,\nu^\ast q)\arrow[r] & \map{\Fun(\Delta^1, \PSh_{\BB}(\Delta^1))}(\sigma_0^\ast p\vert_{\Delta}, \sigma_0^\ast q\vert_{\Delta}).
	\end{tikzcd}
	\end{equation*}
	By Lemma~\ref{lem:restrictionSimplicesMonomorphism}, the lower horizontal map is a monomorphism, hence the upper horizontal map is one as well. The lemma furthermore implies that a map $ p\vert_{\Delta}\to q\vert_\Delta$ is contained in the image of the upper horizontal map if and only if the composition $P_+\into P_1\to Q_1$ takes values in $Q_+$, which by Proposition~\ref{prop:comparisonCocartesianFibrationMarkedUnmarked} is the case precisely when the map $ p\vert_{\Delta}\to q\vert_\Delta$ is a cocartesian functor. Hence the restriction functor $(-)\vert_{\Delta}$ induces an equivalence
	\begin{equation*}
	\map{\Cocart^+}(p,q)\simeq \map{\Cocart}(p\vert_{\Delta}, q\vert_{\Delta})
	\end{equation*}
	and is thus fully faithful.
	
	We complete the proof by showing that the functor is essentially surjective. If $p\colon\I{P}\to\I{C}$ is a cocartesian fibration in $\Cat(\BB)$ and if $Z\into \I{P}_1$ denotes the subobject that is spanned by the cocartesian maps, Remark~\ref{rem:stabilityCocartesianMaps} implies that the map $s_0\colon \I{P}_0\to\I{P}_1$ factors through $Z$ and therefore determines a map $(\Delta^2)^{\op}\to \BB$ whose restriction along $d_{\{0,2\}}\colon\Delta^1\into\Delta^2$ recovers $s_0\colon\I{P}_0\to\I{P}_1$. In that way, one obtains a marked simplicial object $\I{P}^\natural\in\mSimp\BB\simeq \Simp\BB\times_{\PSh_{\BB}(\Delta^1)}\PSh_{\BB}(\Delta^2)$ with $\I{P}^\natural_+=Z$ such that $\I{P}^\natural\vert_\Delta\simeq \I{P}$. By construction, the object $\I{P}^\natural$ comes equipped with a map $p^\natural\colon\I{P}^\natural\to \I{C}^\sharp$. By Proposition~\ref{prop:comparisonCocartesianFibrationMarkedUnmarked}, we now conclude that $p^\natural$ defines the desired object of $\Cocart^+$ that satisfies $p^\natural\vert_{\Delta}\simeq p$.
\end{proof}
\begin{corollary}
	\label{cor:pullbackSquareCocartesianMarkedUnmarked}
	There is a pullback square
	\begin{equation*}
	\begin{tikzcd}
	\Cocart\arrow[r, hookrightarrow, "(-)^\natural"]\arrow[d] & \Cocart^+\arrow[d]\\
	\Cat(\BB)\arrow[r, hookrightarrow, "(-)^\sharp"] & \mSimp\BB
	\end{tikzcd}
	\end{equation*}
	of $\infty$-categories.\qed
\end{corollary}

\begin{remark}
	\label{rem:cartesianFibrationsAnodyneGenerators}
	There is an evident way to dually define a factorisation system of \emph{marked right anodyne maps} and \emph{marked cartesian fibrations} in $\mSimp\BB$. Since the equivalence $\op\colon \Delta\simeq\Delta$ can be uniquely extended to an equivalence $\op\colon\Delta_+\simeq\Delta_+$ upon specifying that $\op$ carries the factorisation $\ord{1}\to +\to \ord{0}$ to itself, we may simply define a map $f$ in $\mSimp\BB$ to be marked right anodyne if $f^{\op}$ is marked left anodyne. Explicitly, the class of marked right anodyne maps is the internal saturation of the maps in~(1)--(3) in Definition~\ref{def:markedAnodyne} together with the map
	\begin{enumerate}
	\item[4$^\prime$)] $d^0\colon 1\into (\Delta^1)^\sharp$.
	\end{enumerate}
	A map $f$ in $\mSimp\BB$ is then a marked cartesian fibration if it is right orthogonal to the class of marked right anodyne maps, or equivalently if $f^\op$ is a marked cocartesian fibration. We denote by $\Cart^+$ the associated cartesian fibration over $\mSimp\BB$. Note that by similarly replacing the maps in ~(4) and~(5) in Proposition~\ref{prop:markedAnodyneGenerators}, one obtains an analogous collection of generators for the dual case. In particular, Proposition~\ref{prop:comparisonCocartesianFibrationMarkedUnmarked} carries over to the case of cartesian fibrations, which implies that we also have a pullback square
	\begin{equation*}
	\begin{tikzcd}
	\Cart\arrow[r, hookrightarrow, "(-)^\natural"]\arrow[d] & \Cart^+\arrow[d]\\
	\Cat(\BB)\arrow[r, hookrightarrow, "(-)^\sharp"] & \mSimp\BB.
	\end{tikzcd}
	\end{equation*}
\end{remark}

\subsection{Marked left fibrations}
\label{sec:markedLeftFibrations}
In this section we discuss the marked analogue of the class of left fibrations between simplicial objects in $\BB$. We will use this notion to show that left fibrations form a coreflective subcategory of cocartesian fibrations.
\begin{definition}
	\label{def:markedInitial}
	A map in $\mSimp\BB$ is \emph{marked initial} if it is contained in the internal saturation of the two maps $d^1\colon 1\into (\Delta^1)^\flat$ and $d^1\colon 1\into (\Delta^1)^\sharp$.
\end{definition}
\begin{remark}
	\label{rem:markedInitialAlternativeDefinition}
	On account of the commutative diagram
	\begin{equation*}
	\begin{tikzcd}
	1\arrow[r, "d^1"]\arrow[dr, "d^1"']& (\Delta^1)^\flat\arrow[d]\\
	& (\Delta^1)^\sharp,
	\end{tikzcd}
	\end{equation*}
	the class of marked initial maps in $\mSimp\BB$ is equivalent to the internal saturation of $d^1\colon 1\into(\Delta^1)^\flat$ and $(\Delta^1)^\flat\into(\Delta^1)^\sharp$.
\end{remark}

\begin{proposition}
	\label{lem:markedInitialMarkedAnodyne}
	Every marked left anodyne map is marked initial.
\end{proposition}
\begin{proof}
	We only need to show that the maps in~(1)--(3) in Definition~\ref{def:markedAnodyne} are marked initial. By Lemma~\ref{lem:markedAnodyneGenerators1}, we may equivalently show this for the maps in~(1)--(3) in Lemma~\ref{lem:markedAnodyneGenerators1}. The case of the first two maps is an immediate consequence of~\cite[Lemma~4.1.4]{Martini2021}. As for the map $(\Delta^1)^\sharp\sqcup_{(\Delta^1)^\flat}(\Delta^1)^\sharp\to(\Delta^1)^\sharp$, this follows from the fact that $(\Delta^1)^\flat\into(\Delta^1)^\sharp$ is marked initial.
\end{proof}

\begin{definition}
	\label{def:markedRightFibration}
	A map $p\colon P\to C$ in $\mSimp\BB$ is called a \emph{marked left fibration} if it is internally right orthogonal to both $d^1\colon 1\into (\Delta^1)^\flat$ and $d^1\colon 1\into (\Delta^1)^\sharp$. We write $\LFib^+\into \Fun(\Delta^1,\mSimp\BB)$ for the full cartesian subfibration over $\mSimp\BB$ that is spanned by the marked left fibrations.
\end{definition}

As a consequence of Proposition~\ref{prop:markedAnodyneGenerators}, one has:
\begin{proposition}
	\label{prop:markedLeftCocartesian}
	Every marked left fibration is marked cocartesian.\qed
\end{proposition}

\begin{lemma}
	\label{lem:generatorsMarkedInitial}
	Let $S$ be the saturation of the maps $d^1\colon K^\flat\into (\Delta^1\otimes K)^\flat$ for every $K\in\Simp\BB$ and $(\Delta^1\otimes A)^\flat\to(
	\Delta^1\otimes A)^\sharp$ for all $A\in\BB$. Then $S$ contains every marked initial map.
\end{lemma}
\begin{proof}
	We begin by showing that $S$ contains the internal saturation of $d^1\colon 1\into(\Delta^1)^\flat$. To that end, note that $S$ is stable under taking products with any object $A\in\BB$. Therefore, it suffices to show that $S$ contains the map $d^1\colon (\Delta^1)^\sharp\into (\Delta^1)^\flat\times(\Delta^1)^\sharp$. The pushout square
	\begin{equation*}
		\begin{tikzcd}
		(\Delta^1\sqcup\Delta^1)^\flat\arrow[r]\arrow[d, "{(d^1,d^0)}"] & (\Delta^1\sqcup\Delta^1)^\sharp\arrow[d]\\
		(\Delta^1\times\Delta^1)^\flat\arrow[r] & (\Delta^1)^\flat\times(\Delta^1)^\sharp
		\end{tikzcd}
	\end{equation*}
	implies that $(\Delta^1\times\Delta^1)^\flat\into(\Delta^1)^\flat\times(\Delta^1)^\sharp$ is in $S$. On account of the commutative square
	\begin{equation*}
		\begin{tikzcd}
		(\Delta^1)^\flat\arrow[r]\arrow[d, "d^1"] & (\Delta^1)^\sharp\arrow[d, "d^1"] \\
		(\Delta^1\times\Delta^1)^\flat\arrow[r] & (\Delta^1)^\flat\times(\Delta^1)^\sharp,
		\end{tikzcd}
	\end{equation*}
	we therefore conclude that the right vertical map must be contained in $S$ as well. 

	We still need to show that $S$ also contains the internal saturation of $d^1\colon 1\into(\Delta^1)^\sharp$. By Lemma~\ref{lem:pushoutProductEquivalenceMarked}, the commutative square
	\begin{equation*}
	\begin{tikzcd}
	\Sk_1(\Delta^n)^\flat\arrow[r]\arrow[d, hookrightarrow] & \Sk_1(\Delta^n)^\sharp\arrow[d, hookrightarrow]\\
	(\Delta^n)^\flat\arrow[r] & (\Delta^n)^\sharp
	\end{tikzcd}
	\end{equation*}
	(where $\Sk_1(\Delta^n)$ is the $1$-skeleton of $\Delta^n$,~cf.~\cite[\S~2.9]{Martini2021a}) is a pushout for every $n\geq 2$. Hence $S$ contains the map $(\Delta^n\otimes A)^\flat\to(\Delta^n\otimes A)^\sharp$ for all $n\geq 0$ and all $A\in\BB$ and therefore also the maps $d^1\colon A^\sharp\into(\Delta^n\otimes A)^\sharp$. Using~\cite[Lemma~4.1.2]{Martini2021}, we conclude that for every $K\in\Simp\BB$ the map $K^\sharp\into(\Delta^1\otimes K)^\sharp$ is an element of $S$. To finish the proof, we now only need to verify that the morphism $d^1\colon (\Delta^n)^\flat\into(\Delta^1)^\sharp\times(\Delta^n)^\flat$ is in $S$ too. To that end,~\cite[Lemma~4.1.4]{Martini2021} implies that the maps $(I^n)^\flat\into (\Delta^n)^\flat$ are contained in $S$ and that we can therefore assume $n\in\{0,1\}$. By Remark~\ref{rem:markedInitialAlternativeDefinition}, we can further reduce this to $n=1$. 
	Now the commutative square
	\begin{equation*}
	\begin{tikzcd}
	1\arrow[d, "d^1"]\arrow[r, "d^1"] & (\Delta^1)^\sharp\arrow[d, "\id\times d^1"]\\
	(\Delta^1)^\flat\arrow[r, "d^1\times \id"] & (\Delta^1)^\sharp\times(\Delta^1)^\flat
	\end{tikzcd}
	\end{equation*}
	and the first part of the proof show that the lower horizontal map is contained in $S$, as desired.
\end{proof}
\begin{proposition}
	\label{prop:markedLeftFibrationSharp}
	Let $C$ be a simplicial object in $\BB$. Then a map $p\colon P\to C^\sharp$ is a marked left fibration if and only if the map $P\to P\vert_{\Delta}^\sharp$ is an equivalence and $p\vert_{\Delta}\colon P\vert_{\Delta}\to C$ is a left fibration.
\end{proposition}
\begin{proof}
	The map $p\vert_{\Delta}$ being a left fibration is equivalent to $p$ being right orthogonal to $d^1\colon K^\flat\into (\Delta^1\times K)^\flat$ for all $K\in\Simp\BB$. Also, since $C^\sharp\to (C^\sharp)\vert_{\Delta}^\sharp$ is trivially an equivalence, the map $P\to P\vert_{\Delta}^\sharp$ is an equivalence if and only if $p$ is right orthogonal to $(\Delta^1\otimes A)^\flat\to (\Delta^1\otimes A)^\sharp$ for all $A\in\BB$. Hence the result follows from Lemma~\ref{lem:generatorsMarkedInitial}.
\end{proof}

\begin{remark}
	\label{rem:markedLeftFibrationFlat}
	There is a dual version of Proposition~\ref{prop:markedLeftFibrationSharp} with $(-)^\flat$ in place of $(-)^\sharp$: a map $p\colon P\to C^\flat$ is a marked left fibration if and only if $p\vert_{\Delta}$ is a left fibration and $P\vert_{\Delta}^\flat\to P$ is an equivalence. To see this, note that $P\vert_{\Delta}^\flat\to P$ is an equivalence if and only if the unique map $P_0\to P_+$ is one, which is in turn equivalent to $p$ being local with respect to $s^0\colon(\Delta^1\otimes A)^\sharp\to A^\sharp$ for all $A\in\BB$. As the latter map is a retraction of $d^1\colon A^\sharp\to (\Delta^1\otimes A)^\sharp$, this condition is equivalent to $P$ being local with respect to $d^1\colon A\to (\Delta^1\otimes A)^\sharp$. Since $C^\flat$ is local with respect to this map as well, we conclude that $P\vert_{\Delta}^\flat\to P$ is an equivalence if and only if $p$ is right orthogonal to $d^1\colon A\to (\Delta^1\otimes A)^\sharp$. By applying Lemma~\ref{lem:generatorsMarkedInitial}, the claim now follows.
\end{remark}

Recall from~\cite[\S~4.1]{Martini2021} that the collection of left fibrations in $\Simp\BB$ determines a cartesian fibration $\LFib\to\Simp\BB$. Proposition~\ref{prop:markedLeftFibrationSharp} now implies:
\begin{corollary}
	\label{cor:pullbackSquareLFibMarkedUnmarked}
	The commutative square
	\begin{equation*}
	\begin{tikzcd}
	\LFib\arrow[r, hookrightarrow, "(-)^\sharp"]\arrow[d] & \LFib^+\arrow[d]\\
	\Simp\BB\arrow[r, hookrightarrow, "(-)^\sharp"] & \mSimp\BB
	\end{tikzcd}
	\end{equation*}
	is a pullback diagram of $\infty$-categories.\qed
\end{corollary}
Note that since marked cocartesian fibrations are internally right orthogonal to $d^1\colon 1\into (\Delta^1)^\sharp$, the adjunction $(-)^\sharp\dashv (-)_\sharp\colon\Fun(\Delta^1,\mSimp\BB)\leftrightarrows\Fun(\Delta^1,\Simp\BB)$ restricts to an adjunction
\begin{equation*}
(-)^\sharp\dashv (-)_\sharp\colon \Cocart^+\leftrightarrows\LFib.
\end{equation*}
Upon restriction to the full subcategory $\Cocart=\Cocart^+\times_{\mSimp\BB}\Simp\BB\into\Cocart^+$, this yields:
\begin{proposition}
	\label{prop:coreflectionLeftFibrationsCocartesianFibrations}
	The inclusion $\LFib\into\Cocart$ admits a relative right adjoint $(-)_\sharp$ over $\Simp\BB$.\qed
\end{proposition}
\begin{remark}
	In light of Corollary~\ref{cor:pullbackSquareCocartesianMarkedUnmarked}, the cartesian fibration $\Cocart\to\Cat(\BB)$ as defined in~\S~\ref{sec:Definitions} arises as the pullback of the cartesian fibration $\Cocart\to\Simp\BB$ along the inclusion $\Cat(\BB)\into\Simp\BB$. Therefore, our choice of using the same notation for both fibrations should not lead to confusion.
\end{remark}

\begin{remark}
	\label{rem:markedRightFibrations}
	As usual, one can dualise the notion of marked initial maps and marked left fibrations in the evident way to obtain marked final maps and marked right fibrations. All statements about marked left fibrations carry over to analogous statements about marked right fibrations. In particular, upon defining $\Cart=\Cart^+\times_{\mSimp\BB}\Simp\BB$, one obtains an inclusion $\RFib\into\Cart$ that admits a relative right adjoint $(-)_\sharp\colon \Cart\to \RFib$ over $\Simp\BB$ as well.
\end{remark}

\begin{remark}
	\label{rem:coreflectionCocartLFibExplicitly}
	If $p\colon \I{P}\to\I{C}$ is a cocartesian fibration, Proposition~\ref{prop:comparisonCocartesianFibrationMarkedUnmarked} implies that the adjunction counit $\I{P}_\sharp\to \I{P}$ is a monomorphism that identifies $\I{P}_\sharp$ with the subcategory of $\I{P}$ that is spanned by the subobject $(\Comma{\I{P}}{\I{C}}{\I{C}})_0\into\I{P}_1$ of cocartesian maps. See~\cite[\S~2.9]{Martini2021a} for a discussion of subcategories.
\end{remark}

\subsection{Proper maps of marked simplicial objects}
\label{sec:markedProper}
Recall from~\cite[\S~4.4]{Martini2021} that a map $p\colon P\to C$ between simplicial objects in $\BB$ is \emph{proper} if for every base change $q\colon Q\to D$ of $p$ along some map $f\colon D\to C$ the lax square
\begin{equation*}
	\begin{tikzcd}
		\Over{\LFib}{D}\arrow[d, "p^\ast"]\arrow[from=r, "\Over{L}{D}"'] & \Over{(\Simp\BB)}{D}\arrow[d, "q^\ast"]\\
		\Over{\LFib}{Q}\arrow[from=r, "\Over{L}{Q}"'] & \Over{(\Simp\BB)}{Q}\arrow[ul, Rightarrow, shorten=3mm]
	\end{tikzcd}
\end{equation*}
commutes. In this section we will discuss the analogous notion of proper maps between \emph{marked} simplicial objects.

\begin{definition}
	\label{def:proper}
	A map $p\colon P\to C$ in $\mSimp\BB$ is \emph{marked proper} if for every cartesian square
	\begin{equation*}
			\begin{tikzcd}
				Q\arrow[d, "q"]\arrow[r] & P\arrow[d, "p"]\\
				D\arrow[r] & C
			\end{tikzcd}
	\end{equation*}
	in $\mSimp\BB$ the left lax square
	\begin{equation*}
		\begin{tikzcd}
				\Over{\Cocart^+}{D}\arrow[d, "q^\ast"]\arrow[from=r, "\Over{L}{D}"'] & \Over{(\mSimp\BB)}{D}\arrow[d, "q^\ast"]\\
				\Over{\Cocart^+}{Q}\arrow[from=r, "\Over{L}{Q}"'] & \Over{(\mSimp\BB)}{Q}\arrow[ul, Rightarrow, shorten=3mm]
			\end{tikzcd}
	\end{equation*}
	(where $\Over{L}{D}$ and $\Over{L}{Q}$ are the localisation functors) commutes. The full cartesian subfibration of $\Fun(\Delta^1,\mSimp\BB)$ that is spanned by the marked proper maps is denoted by $\Prop^+$.
\end{definition}
Note that as in~\cite[\S~4.4]{Martini2021} a map $p\colon P\to C$ in $\mSimp\BB$ is marked proper if and only if for every base change $q\colon Q\to D$ of $p$ the pullback functor $q^\ast\colon \Over{(\mSimp\BB)}{D}\to\Over{(\mSimp\BB)}{Q}$ preserves marked left anodyne morphisms.

\begin{proposition}
	\label{prop:projectionsAreProper}
	For every two marked simplicial objects $C$ and $D$, the projection $C\times D\to D$ is marked proper.
\end{proposition}
\begin{proof}
	It suffices to show that the terminal map $\pi_C\colon C\to 1$ is marked proper, which follows immediately from the fact that marked left anodyne morphisms are \emph{internally} saturated and therefore preserved by $\pi_C^\ast$.
\end{proof}
\begin{remark}
	\label{rem:properMapBC}
	Given any $A\in\BB$, the forgetful functor $(\pi_A)_!\colon \mSimp{(\Over{\BB}{A})}\to\mSimp\BB$ preserves marked proper maps. In fact, since $(\pi_A)_!$ commutes with pullbacks, this follows from the straightforward observation that this functor also preserves the property of a map to be marked left anodyne. As a consequence, Proposition~\ref{prop:projectionsAreProper} also implies that the projection $C\times_A D\to D$ is marked proper for all $C,D\in\mSimp{(\Over{\BB}{A})}$.
\end{remark}

In~\cite[Proposition~4.4.7]{Martini2021} we showed that every right fibration between simplicial objects in $\BB$ is proper. Our next goal is to generalise this result to marked simplicial objects. We begin with the following lemma, the proof of which we learned from Denis-Charles Cisinski~\cite[Proposition~5.3.5]{cisinski2019a}.

\begin{lemma}
	\label{lem:pullbackInnerAnodyneRightFibration}
	Let $\I{C}$ be a $\BB$-category and let
	\begin{equation*}
		\begin{tikzcd}
		Q\arrow[d, "q"]\arrow[r, "j", hookrightarrow] & \I{P}\arrow[d, "p"]\\
		I^2\otimes \I{C}\arrow[r, hookrightarrow, "i"] & \Delta^2\otimes\I{C}
		\end{tikzcd}
	\end{equation*}
	be a pullback square in $\Simp\BB$ in which $p$ is a right fibration. Then $j$ is contained in the internal saturation of $I^2\into\Delta^2$ and $E^1\to 1$.
\end{lemma}
\begin{proof}
	Let $L\colon\Simp\BB\to\Cat(\BB)$ be the localisation functor. Since $\I{P}$ is a $\BB$-category, we obtain a factorisation $Q\to L(Q)\to \I{P}$ of $j$, and our task is to show that the second map is an equivalence. Note that as $i$ is an equivalence on level $0$, so is $j$. As a consequence, the map $L(Q)\to \I{P}$ is essentially surjective. Let us show that it is fully faithful too. Consider the commutative square
	\begin{equation*}
		\begin{tikzcd}
			\RFib(Q)\arrow[r, "j_!"] \arrow[d, "q_!"]& \RFib(\I{P})\arrow[d, "p_!"]\\
			\RFib(I^2\otimes \I{C})\arrow[r, "i_!"] & \RFib(\Delta^2\otimes \I{C})
		\end{tikzcd}
	\end{equation*}
	in which each arrow is the left adjoint of the corresponding pullback functor. We claim that $j_!$ is an equivalence.  To see this, note that applying the functor $p_!$ to the adjunction counit $j_! j^\ast\to\id$ recovers the adjunction counit of $i_!\dashv i^\ast$. Since the Grothendieck construction~\cite[Theorem~4.5.1]{Martini2021} implies that $i^\ast$ is an equivalence and as $p_!$ is conservative since $p$ is a right fibration, we thus find that $j_! j^\ast\to\id$ is an equivalence. As a consequence, $j^\ast$ is fully faithful. But $i_!$ being an equivalence and both $p_!$ and $q_!$ being conservative also implies that $j_!$ is conservative. As a result, $j_!$ must be an equivalence. Upon applying the functor $-\times A$ to the original pullback square for any $A\in\BB$, the above argumentation also shows that $g_!\colon \RFib(Q\times A)\to\RFib(\I{P}\times A)$ is an equivalence. Together with the Grothendieck construction, this shows that restriction along $L(Q)\to \I{P}$ induces an equivalence $\IPSh_{\Univ}(\I{P})\simeq \IPSh_{\Univ}(L(Q))$ of $\BB$-categories. In light of~\cite[Corollary~3.3.3]{Martini2021a}, this implies that $L(Q)\to \I{P}$ is fully faithful, as desired.
\end{proof}

\begin{proposition}
	\label{prop:rightFibrationsProper}
	Every marked right fibration is marked proper.
\end{proposition}
\begin{proof}
	As marked right fibrations form a local class in $\mSimp\BB$, it suffices to prove that whenever there is a pullback square
	\begin{equation*}
	\begin{tikzcd}
	Q\arrow[d, "q"]\arrow[r, "g"] & P\arrow[d, "p"]\\
	D\arrow[r, "f"] & C
	\end{tikzcd}
	\end{equation*}
	in which $f$ is one of the maps in Proposition~\ref{prop:markedAnodyneGenerators} and $p$ is a marked right fibration, the map $g$ is marked left anodyne. We will first go through the maps listed in Lemma~\ref{lem:markedAnodyneGenerators1}:
	\begin{enumerate}
	\item We begin with the case where $f$ is given by the inclusion $(I^2\otimes K)^\flat\into(\Delta^2\otimes K)^\flat$. As every simplicial object in $\BB$ is a colimit of $\BB$-categories, we can assume that $K$ is a $\BB$-category. By Remark~\ref{rem:markedLeftFibrationFlat}, the map $P\vert_{\Delta}^\flat\to P$ is an equivalence. As a consequence, to show that $g$ marked left anodyne, it suffices to show that $g\vert_\Delta$ is contained in the internal saturation of $I^2\into\Delta^2$ and $E^1\to 1$, which is a consequence of Lemma~\ref{lem:pullbackInnerAnodyneRightFibration}.

	\item The case where $f$ is the map $(E^1\otimes K)^\flat\to K^\flat$ follows immediately from the fact that marked left anodyne maps are stable under products.
	
	\item Finally, we prove the case where $f$ is the map $(\Delta^1\otimes A)^\sharp\sqcup_{(\Delta^1\otimes A)^\flat}(\Delta^1\otimes A)^\sharp\to(\Delta^1\otimes A)^\sharp$. By Proposition~\ref{prop:markedLeftFibrationSharp}, the morphism $P\to P\vert_{\Delta}^\sharp$ is an equivalence. Let us use the notation $P^\prime=P\vert_{\Delta}$ and $p^\prime=p\vert_{\Delta}$. Then $p\simeq(p^\prime)^\sharp$. By Lemma~\ref{lem:pushoutProductEquivalenceMarked} and the fact that left fibrations are conservative, the map $(P^\prime)^\flat\to (P^\prime)^\sharp\times_{(\Delta^1\otimes A)^\sharp} (\Delta^1\otimes A)^\flat$ is an equivalence. Therefore, the map $g\colon Q\to P$ is equivalent to $(P^\prime)^\sharp\sqcup_{(P^\prime)^\flat}(P^\prime)^\sharp\to (P^\prime)^\sharp$. By Remark~\ref{rem:markedMonomorphismGenerator}, this map is marked left anodyne.
	\end{enumerate}
	By using Lemma~\ref{lem:markedAnodyneGenerators2}, it now suffices to prove the case where $f$ is of the form $d^1\colon K\into(\Delta^1)^\sharp\otimes K$ for an arbitrary $K\in \mSimp\BB$. This is done in the same way as in the proof of~\cite[Proposition~4.4.7]{Martini2021}. Namely, the map $g\colon Q\to P$ can be shown to arise as a retract of the marked left anodyne map $((\Delta^1)^\sharp\otimes Q)\sqcup_{Q} P\to (\Delta^1)^\sharp\otimes P$ and is therefore marked left anodyne itself.
\end{proof}

\begin{remark}
	\label{rem:cartesianPartiallyProper}
	In the situation of Proposition~\ref{prop:rightFibrationsProper}, note that the argument in the last paragraph of its proof also works when $p$ is only a marked cartesian fibration, as this argument only requires $p$ to be internally right orthogonal to $d^0\colon 1\into(\Delta^1)^\sharp$. We will need this observation later for the proof of Theorem~\ref{thm:StraighteningEquivalence}.
\end{remark}

\begin{remark}
	\label{rem:markedSmoothMaps}
	One can dualise the discussion in this section to \emph{marked smooth maps}: a map $f$ in $\mSimp\BB$ is said to be smooth if $f^\op$ is proper. Then Proposition~\ref{prop:rightFibrationsProper} dualises to the statement that marked \emph{left} fibrations are smooth.
\end{remark}

\section{The $\BB$-category of cocartesian fibrations}
\label{sec:CategoryOfCocartesianFibrations}
The goal of this chapter is to construct and study the $\BB$-category of cocartesian fibrations over a $\BB$-category $\I{C}$. It will be useful to first adopt a slightly more global perspective, i.e.\ to allow $\I{C}$ to vary. Therefore, we begin in \S~\ref{sec:Cocart+} by defining and studying the $\mSimp\BB$-category of cocartesian fibrations. In~\S~\ref{sec:Cocart} we make use of this $\mSimp\BB$-category to obtain the $\BB$-category of cocartesian fibrations over a fixed base $\BB$-category and to show that it is tensored and powered over $\ICat_{\BB}$. Lastly, \S~\ref{sec:limitsColimitsCocart} contains a discussion of the existence of limits and colimits in this $\BB$-category.

\subsection{The global $\mSimp\BB$-category of cocartesian fibrations}
\label{sec:Cocart+}
Recall that since $\Cocart^+\subset\Fun(\Delta^1,\mSimp\BB)$ is defined as the right complement of a factorisation system, this defines  a local class in the $\infty$-topos $\mSimp\BB$. Therefore, the associated $\CatSS$-valued presheaf is a sheaf on $\mSimp\BB$. We may therefore define:
\begin{definition}
	The large $\mSimp\BB$-category $\ICocart^+$ is defined as the full subcategory of the universe $\Univ[\mSimp\BB]$ that is associated to the $\CatSS$-valued sheaf $\Cocart^+$.
\end{definition}
Being a full subuniverse of $\Univ[\mSimp\BB]$, the $\mSimp\BB$-category $\ICocart^+$ inherits a number of nice properties of $\Univ[\mSimp\BB]$.
For example, since $\ICocart^+$ is defined by the right complement of a factorisation system in $\mSimp\BB$, we deduce from~\cite[Example~5.3.3]{Martini2021a}:
\begin{lemma}
	\label{lem:CocartCocartCocomplete}
	The $\mSimp\BB$-category $\ICocart^+$ is  closed under $\ICocart^+$-colimits in $\Univ[\mSimp\BB]$. That is, $\ICocart^+$ is $\ICocart^+$-cocomplete and the inclusion $\ICocart^+\into\Univ[\mSimp\BB]$ is $\ICocart^+$-cocontinuous.\qed
\end{lemma}
Dually, by the discussion in~\cite[\S~5.4]{Martini2021a}, the class of marked proper maps in $\mSimp\BB$ is local and therefore determines a subuniverse $\IProp\into\Univ[\mSimp\BB]$. Furthermore,~\cite[Proposition~5.4.2]{Martini2021a} shows:
\begin{lemma}
	\label{lem:CocartPrpComplete}
	The $\mSimp\BB$-category $\ICocart^+$ is closed under $\IProp^+$-limits in $\Univ[\mSimp\BB]$.\qed
\end{lemma}

Our next goal is to show that $\ICocart^+$ is tensored and powered over $\BB$-categories. To make this precise, we first need some preparations. Observe that the inclusion $(-)^\sharp\colon\Simp\BB\into\mSimp\BB$ can be regarded as an algebraic morphism whose right adjoint is given by $(-)_\sharp\colon \mSimp\BB\to\Simp\BB$. Similarly,  the diagonal embedding $\BB\into\Simp\BB$ is an algebraic morphism whose right adjoint is the functor $(-)_0$ of evaluation at $0$. To avoid confusion, we will denote the extension of the latter to the level of categories by $(-)\vert_\BB\colon \Cat(\Simp\BB)\to\Cat(\BB)$ and its left adjoint by $(-)\vert^{\Simp\BB}$. We now obtain adjunctions
\begin{equation*}
\begin{tikzcd}
	\Cat(\mSimp\BB)\arrow[r, shift right=1mm, "(-)_\sharp"']\arrow[r, shift left=1mm, hookleftarrow, "(-)^\sharp"]& \Cat(\Simp\BB)\arrow[r, shift right=1mm, "(-)\vert_\BB"']\arrow[r, shift left=1mm, hookleftarrow, "(-)\vert^{\Simp\BB}"] & \Cat(\BB).
\end{tikzcd}
\end{equation*}
Note that on the level of $\CatS$-valued sheaves, the two colocalisation functors $(-)_\sharp$ and $(-)\vert_\BB$ are given by precomposition with the inclusions $(-)^\sharp\colon \Simp\BB\into\mSimp\BB$ and $\BB\into\Simp\BB$, respectively.
\begin{remark}
	\label{rem:Warning}
	Be aware that there are now \emph{two} distinct inclusions $\Cat(\BB)\into\Cat(\Simp\BB)$: the first is given by $(-)^{\Simp\BB}$, and the second is given by the composition $\Cat(\BB)\into\Simp\BB\into\Cat(\Simp\BB)$ in which the second map is the diagonal embedding. The latter inclusion identifies $\Cat(\BB)$ with a class of $\Simp\BB$-groupoids, whereas this is not the case for the former map.
\end{remark}
By~\cite[Lemma~3.7.4]{Martini2021} there is a cartesian square
\begin{equation*}
\begin{tikzcd}
\int \IPSh_{\Univ}(\Delta)\arrow[r, hookrightarrow] \arrow[d] & \Fun(\Delta^1, \Simp\BB)\arrow[d]\\
\BB\arrow[r, hookrightarrow] & \Simp\BB, 
\end{tikzcd}
\end{equation*}
where $\int \IPSh_{\Univ}(\Delta)$ is the cartesian fibration over $\BB$ that corresponds to the $\BB$-category $\IPSh_{\Univ}(\Delta)$. In other words, there is an equivalence $\IPSh_{\Univ}(\Delta)\simeq( \Univ[\Simp\BB])\vert_{\BB}$ and therefore an inclusion $\ICat_{\BB}\into \Univ[\Simp\BB]\vert_{\BB}$. Similarly, the inclusion
$(-)^\flat\colon \Simp\BB\into\mSimp\BB$ determines an embedding $\Univ[\Simp\BB]\into (\Univ[\mSimp\BB])\vert_{\Delta}$ in $\Cat(\Simp\BB)$ and therefore in particular an inclusion $\ICat_{\BB}\into( \Univ[\mSimp\BB]\vert_{\Delta})\vert_{\BB}$ of large $\BB$-categories. By defining  $\ICat_{\BB}^+=(\ICat_{\BB}\vert^{\Simp\BB})^\sharp$ and by making use of the equivalence $((-)\vert_{\Delta})\vert_{\BB}\simeq ((-)_\sharp)\vert_{\BB}$ (see Remark~\ref{rem:FlatSharpRestrictionGroupoids}), we therefore obtain a functor $(-)^\flat\colon\ICat_{\BB}^+\to\Univ[\mSimp\BB]$ that carries a $\Over{\BB}{A}$-category $\I{E}\to A$ to the map $\I{E}^\flat\to A^\flat$. In light of Remark~\ref{rem:properMapBC}, this functor therefore takes values in $\IProp^+\cap\ICocart^+$. 
Consequently, we deduce from~\cite[Proposition~7.3.6]{Martini2021a} and~\cite[Remark~7.3.7]{Martini2021a}:
\begin{proposition}
	\label{prop:Cocart+TensoredPowered}
	The large $\mSimp\BB$-category $\ICocart^+$ is tensored and powered over $\ICat_{\BB}^+$. In other words, there are bifunctors
	\begin{equation*}
	-\otimes -\colon \ICat_{\BB}^+\times \ICocart^+\to\ICocart^+
	\end{equation*}
	and
	\begin{equation*}
	(-)^{(-)}\colon(\ICat_{\BB}^+)^{\op}\times\ICocart^+\to\ICocart^+
	\end{equation*}
	together with an equivalence
	\begin{equation*}
	\map{\ICocart^+}(-\otimes -,-)\simeq\map{\ICocart^+}(-, (-)^{(-)})
	\end{equation*}
	of functors in $\Cat(\mSimp\BBB)$.\qed
\end{proposition}
\begin{remark}
	\label{rem:tensoringPoweringCocart+Explicitly}
	Explicitly, the tensoring of $\ICocart^+$ over $\ICat_{\BB}^+$ is given by the restriction of the product functor $-\times -\colon \Univ[\mSimp\BB]\times\Univ[\mSimp\BB]\to\Univ[\mSimp\BB]$ along $(-)^\flat\times\id\colon\ICat_{\BB}^+\times\ICocart^+\to\ICocart^+\times\ICocart^+$.
\end{remark}

\begin{remark}
	\label{rem:Cart+}
	One can dualise the discussion in this section to marked cartesian fibrations in the evident way: the sheaf $\Cart^+$ on $\mSimp\BB$ determines a large $\mSimp\BB$-category $\ICart^+$ of marked cartesian fibrations, which is closed under $\ICart^+$-colimits and $\mSmooth^+$-limits in $\Univ[\mSimp\BB]$, where $\mSmooth^+$ denotes the full subcategory of $\Univ[\mSimp\BB]$ that is determined by the local class of marked smooth maps. Hence $\ICart^+$ is tensored and cotensored over $\ICat_{\BB}^+$ as well.
\end{remark}
\subsection{$\BB$-categories of cocartesian fibrations}
\label{sec:Cocart}
In this section we make use of the preparations made in \S~\ref{sec:Cocart+} to define the $\BB$-category of cocartesian fibrations over a fixed $\BB$-category $\I{C}$. As an intermediate step, it will be useful to go through the large $\Simp\BB$-category $\ICocart=(\ICocart^+)_\sharp$.
\begin{lemma}
	\label{lem:geometricMorphismPowering}
	Let $f_\ast\colon \BB\to \AA$ be a geometric morphism of $\infty$-topoi and let $\I{C}$ be a $\BB$-category. Then there is an equivalence $\Tw(f_\ast\I{C})\simeq f_\ast\Tw(\I{C})$ of left fibrations over $f_\ast\I{C}^{\op}\times f_\ast\I{C}$.
\end{lemma}
\begin{proof}
	The equivalence $f^\ast(-\otimes -)\simeq -\otimes f^\ast(-)$ implies that $f_\ast$ commutes with powering over $\CatS$. In other words, there is a natural equivalence
	\begin{equation*}
		f_\ast(-)^{(-)}\simeq f_\ast((-)^{(-)})
	\end{equation*}
	of bifunctors $\CatS^{\op}\times\Cat(\BB)\to\Cat(\AA)$. In particular, this observation implies that there is a natural equivalence $\Tw(f_\ast\I{C})\simeq f_\ast\Tw(\I{C})$ of left fibrations over $f_\ast\I{C}^{\op}\times f_\ast\I{C}$.
\end{proof}
By applying $(-)_\sharp$ to the bifunctors from Proposition~\ref{prop:Cocart+TensoredPowered} and by using Lemma~\ref{lem:geometricMorphismPowering} we now deduce:
\begin{corollary}
	\label{cor:CocartPoweredTensored}
	The large $\Simp\BB$-category $\ICocart$ is tensored and powered over $\ICat_{\BB}\vert^{\Simp\BB}$.\qed
\end{corollary}

\begin{definition}
	\label{def:categoryOfCocartesianFibrations}
	Let $\I{C}$ be a $\BB$-category, viewed as a $\Simp\BB$-groupoid (cf.~Remark~\ref{rem:Warning}). We then define the large $\BB$-category of cocartesian fibrations over $\I{C}$ as $\ICocart_{\I{C}}=\iFun{\I{C}}{\ICocart}\vert_{\BB}$.
\end{definition}
\begin{remark}
	\label{rem:cocartCUnderlyingSheaf}
	Explicitly, the $\CatSS$-valued sheaf associated with $\ICocart_{\I{C}}$ is given by $\Cocart^+(-\times \I{C}^\sharp)$ and therefore by $\Cocart(-\times \I{C})$, using Corollary~\ref{cor:pullbackSquareCocartesianMarkedUnmarked}.
\end{remark}

Note that owing to the fact that the functor $\iFun{\I{C}}{-}\colon \Cat(\Simp\BBB)\to\Cat(\Simp\BBB)$  commutes with the powering of $\Cat(\Simp\BBB)$ over $\CatSS$, there is a natural equivalence $\iFun{\I{C}}{\Tw(-)}\simeq \Tw(\iFun{\I{C}}{-})$. Together with Lemma~\ref{lem:geometricMorphismPowering}, this observation implies that by applying the functor $\iFun{\I{C}}{-}\vert_{\BB}$ to the two maps from Corollary~\ref{cor:CocartPoweredTensored} and by precomposing with the map $(\pi_{\I{C}})^\ast\colon \ICat_{\BB}\simeq\iFun{1}{\ICat_{\BB}\vert^{\Simp\BB}}\vert_{\BB}\to\iFun{\I{C}}{\ICat_{\BB}\vert^{\Simp\BB}}\vert_{\BB}$, one obtains bifunctors
	\begin{equation*}
	-\otimes -\colon \ICat_{\BB}\times \ICocart_{\I{C}}\to\ICocart_{\I{C}}
	\end{equation*}
	and
	\begin{equation*}
	(-)^{(-)}\colon\ICat_{\BB}^{\op}\times\ICocart_{\I{C}}\to\ICocart_{\I{C}}
	\end{equation*}
	that are natural in $\I{C}$ and that fit into an equivalence
	\begin{equation*}
	\map{\ICocart_{\I{C}}}(-\otimes -,-)\simeq\map{\ICocart_{\I{C}}}(-, (-)^{(-)}).
	\end{equation*}
	We conclude:
\begin{corollary}
	\label{cor:CocartCTensoredPowered}
	For any $\I{C}\in\Cat(\BB)$, the large $\BB$-category $\ICocart_{\I{C}}$ is both tensored and powered over the $\BB$-category $\ICat_{\BB}$, and both tensoring and powering is natural in $\I{C}$.\qed
\end{corollary}

\begin{remark}
	\label{rem:CocartCTensoringPoweringExplicitly}
	Using Remark~\ref{rem:tensoringPoweringCocart+Explicitly}, the tensoring of $\ICocart_{\I{C}}$ is given by the restriction of the product functor $-\times-\colon\ICocart_{\I{C}}\times\ICocart_{\I{C}}\to\ICocart_{\I{C}}$ along the map $\ICat_{\BB}\times\ICocart_{\I{C}}\to\ICocart_{\I{C}}\times\ICocart_{\I{C}}$ that is given on the first factor by the functor $\pi_{\I{C}}^\ast\colon \ICat_{\BB}\simeq\ICocart_{1}\to\ICocart_{\I{C}}$.
\end{remark}

\begin{remark}
	\label{rem:CocartCTensoringPoweringSectionwise}
	By yet another application of Lemma~\ref{lem:geometricMorphismPowering} and by using~\cite[Corollary~4.6.8]{Martini2021}, applying the global sections functor $\Gamma$ to the tensoring and powering bifunctors of $\ICocart_{\I{C}}$ exhibits $\Cocart(\I{C})$ as being tensored and powered over $\Cat(\BB)$. Moreover, Remark~\ref{rem:CocartCTensoringPoweringExplicitly} implies that the tensoring
	\begin{equation*}
	-\otimes -\colon \Cat(\BB)\times\Cocart(\I{C})\to\Cocart(\I{C})
	\end{equation*}
	is given by $\pi_{\I{C}}^\ast(-)\times_{\I{C}} -$. As a consequence, the cotensoring $(-)^{(-)}\colon \Cat(\BB)^{\op}\times\Cocart(\I{C})\to\Cocart(\I{C})$ can be explicitly described as the functor that carries a pair $(\I{D},(\I{P}\to\I{C}))$ to the pullback
	\begin{equation*}
	\begin{tikzcd}
	\I{P}^{\I{D}}\arrow[r]\arrow[d] & \iFun{\I{D}}{\I{P}}\arrow[d]\\
	\I{C}\arrow[r, "\diag"] & \iFun{\I{D}}{\I{C}}.
	\end{tikzcd}
	\end{equation*}
\end{remark}

\begin{remark}
	\label{rem:BCCocart}
	Let $A\in\BB$ be an an arbitrary object. Note that by postcomposition, the adjunction $(\pi_A)_!\dashv\pi_A^\ast\colon \BB\leftrightarrows\Over{\BB}{A}$ induces adjunctions $(\pi_A)_!\dashv\pi_A^\ast\colon\Simp\BB\leftrightarrows\Simp{(\Over{\BB}{A})}$ and $(\pi_A)_!\dashv\pi_A^\ast\colon\mSimp\BB\leftrightarrows\mSimp{(\Over{\BB}{A})}$ that give rise to a diagram
	\begin{equation*}
	\begin{tikzcd}[column sep={9em,between origins}, row sep={5em,between origins}]
		\Cat(\mSimp\BB)\arrow[r, shift right=1mm, "(-)_\sharp"']\arrow[r, shift left=1mm, hookleftarrow, "(-)^\sharp"]\arrow[d, shift right=1mm, "\pi_A^\ast"']\arrow[from=d, shift right=1mm, "(\pi_A)_!"'] & \Cat(\Simp\BB)\arrow[r, shift right=1mm, "(-)\vert_\BB"']\arrow[r, shift left=1mm, hookleftarrow, "(-)\vert^{\Simp\BB}"] \arrow[d, shift right=1mm, "\pi_A^\ast"']\arrow[from=d, shift right=1mm, "(\pi_A)_!"']& \Cat(\BB)\arrow[d, shift right=1mm, "\pi_A^\ast"']\arrow[from=d, shift right=1mm, "(\pi_A)_!"'] \\
		\Cat(\mSimp{(\Over{\BB}{A})})\arrow[r, shift right=1mm, "(-)_\sharp"']\arrow[r, shift left=1mm, hookleftarrow, "(-)^\sharp"]& \Cat(\Simp{(\Over{\BB}{A})})\arrow[r, shift right=1mm, "(-)\vert_{\Over{\BB}{A}}"']\arrow[r, shift left=1mm, hookleftarrow, "(-)\vert^{\Simp{(\Over{\BB}{A})}}"] & \Cat(\Over{\BB}{A})
	\end{tikzcd}
	\end{equation*}
	that commutes in every direction.
	Since the terminal map $\pi_A\colon A\to 1$ is a marked cocartesian fibration, a map $p\colon P\to C$ in $\mSimp{(\Over{\BB}{A})}$ is marked cocartesian if and only if $(\pi_A)_!(p)$ is marked cocartesian in $\mSimp{\BB}$. Therefore, the equivalence $\pi_A^\ast\Univ[\mSimp\BB]\simeq\Univ[\mSimp{(\Over{\BB}{A})}]$ restricts to an equivalence of large $\mSimp{(\Over{\BB}{A})}$-categories $\pi_A^\ast\ICocart^+\simeq\ICocart^+$. Also, there is an equivalence $\pi_A^\ast\ICat_{\BB}^+\simeq\ICat_{\Over{\BB}{A}}^+$ with respect to which the functor $(-)^\flat\colon\ICat_{\Over{\BB}{A}}^+\to \Univ[\mSimp{(\Over{\BB}{A})}]$ corresponds to the image of the map $(-)^\flat\colon\ICat_{\BB}^+\to\Univ[\mSimp\BB]$ along $\pi_A^\ast$. Since $\pi_A^\ast$ moreover carries the mapping bifunctor of a $\mSimp\BB$-category $\I{C}$ to the mapping bifunctor of $\pi_A^\ast\I{C}$, we find that applying $\pi_A^\ast$ to the tensoring and powering of the $\mSimp\BB$-category $\ICocart^+$ yields the tensoring and powering of the $\mSimp{(\Over{\BB}{A})}$-category $\ICocart^+$. The above commutative diagram and the fact that $\pi_A^\ast$ also commutes with taking functor $\Simp\BB$-categories now implies that for every $\BB$-category $\I{C}$ there is a canonical equivalence $\pi_A^\ast\ICocart_{\I{C}}\simeq\ICocart_{\pi_A^\ast\I{C}}$ with respect to which the tensoring and powering of $\ICocart_{\pi_A^\ast\I{C}}$ over $\ICat_{\Over{\BB}{A}}$ are obtained as the image of the tensoring and powering bifunctors of $\ICocart_{\I{C}}$ along $\pi_A^\ast$.
\end{remark}

\begin{remark}
	\label{rem:definitionCartC}
	If we set $\ICart=(\ICart^+)_\sharp$, we can define the large $\BB$-category of cartesian fibrations over a $\BB$-category $\I{C}$ as $\ICart_{\I{C}}=\iFun{\I{C}}{\ICart}\vert_{\BB}$, analogous to how we defined the large $\BB$-category $\ICocart_{\I{C}}$. Using this definition, it is immediate that the tensoring and powering of $\ICart^+$ over $\ICat_{\BB}^+$ induces tensoring and powering bifunctors of $\ICart_{\I{C}}$ over $\ICat_{\BB}$ as well.
\end{remark}

\subsection{Limits and colimits in $\ICocart_{\I{C}}$}
\label{sec:limitsColimitsCocart}
In this section we will study limits and colimits in the $\BB$-categories of cocartesian fibrations that we defined in the previous section and the behaviour of these under base change. To that end, recall that since  $\Cocart^+(C)$ is a reflective subcategory of the $\infty$-topos $\Over{(\mSimp\BB)}{C}$ for every marked simplicial object $C\in\mSimp\BB$, the $\infty$-category $\Cocart^+(C)$ admits small limits and colimits. Moreover, if $f\colon C\to D$ is a map of marked simplicial objects in $\BB$, the restriction functor $f^\ast\colon \Cocart^+(D)\to\Cocart^+(C)$ has a left adjoint $f_!$ that can be explicitly computed by means of the composition
		\begin{equation*}
		\Cocart^+(C)\into \Over{(\mSimp\BB)}{C}\xrightarrow{f_!} \Over{(\mSimp\BB)}{D} \xrightarrow{\Over{L}{D}}  \Cocart^+(D)
		\end{equation*}
in which $\Over{L}{D}$ denotes the localisation functor. In particular, $f^\ast$ preserves small limits. If $f$ happens to be marked proper, then  $f^\ast$ also has a right adjoint $f_\ast$ by Lemma~\ref{lem:CocartPrpComplete}, hence $f^\ast$ also preserves small colimits in this case. Now if $s\colon B\to A$ is an arbitrary map in $\BB$, the associated map $s^\sharp$ is both marked proper and a marked cocartesian fibration, hence so is its product with any marked simplicial object $C$. Given any map $f\colon C\to D$ of marked simplicial objects, the fact that $s^\sharp$ is marked cocartesian implies that the natural morphism $(s^\sharp)_! f^\ast\to f^\ast (s^\sharp)_!$ is an equivalence, see Lemma~\ref{lem:CocartCocartCocomplete}. The map $s^\sharp$ being marked proper, on the other hand, implies that the map $(s^\sharp)_\ast f^\ast\to f^\ast (s^\sharp)_\ast$ is an equivalence which is in turn equivalent to $f_! (s^\sharp)^\ast\to (s^\sharp)^\ast f_!$ being an equivalence. By using~\cite[Corollary~5.3.6]{Martini2021a}, we thus conclude:
\begin{proposition}
	\label{prop:CocartCocomplete}
	For every $\I{C}\in\Cat(\BB)$, the large $\BB$-category $\ICocart_{\I{C}}$ is complete and cocomplete.\qed
\end{proposition}

In light of~\cite[Proposition~3.2.8]{Martini2021a}, the same argumentation shows:
\begin{proposition}
	\label{prop:CocartBaseChangeLeftAdjoint}
	For every functor $f\colon \I{C}\to\I{D}$ in $\Cat(\BB)$, the pullback functor $f^\ast\colon \ICocart_{\I{D}}\to\ICocart_{\I{C}}$ admits a left adjoint $f_!$. In particular, $f^\ast$ is continuous.\qed
\end{proposition}

\begin{proposition}
	\label{prop:CocartBaseChangeRightAdjoint}
	Let $p\colon \I{P}\to\I{C}$ be a functor in $\Cat(\BB)$ such that the associated map $p^\sharp$ of marked simplicial objects is proper. Then the pullback functor $p^\ast\colon \ICocart_{\I{C}}\to\ICocart_{\I{P}}$ admits a right adjoint $p_\ast$. In particular, $p^\ast$ is cocontinuous.\qed
\end{proposition}

\begin{remark}
	\label{rem:LeftAdjointPullbackCocartPoint}
	If $\I{C}$ is a $\BB$-category, the functor $(\pi_{\I{C}})_!\colon\ICocart_{\I{C}}\to\ICocart_{1}\simeq\ICat_{\BB}$ is explicitly given by sending a cocartesian fibration $p\colon \I{P}\to A\times\I{C}$ to the $\Over{\BB}{A}$-category $\I{P}_\sharp^{-1}\I{P}$, i.e.\ to the pushout
	\begin{equation*}
	\begin{tikzcd}
	\I{P}_\sharp\arrow[r]\arrow[d, hookrightarrow] & \I{P}_\sharp^\gp\arrow[d]\\
	\I{P}\arrow[r] & \I{P}_\sharp^{-1}\I{P}
	\end{tikzcd}
	\end{equation*}
	in $\Cat(\Over{\BB}{A})$. To see this, Remark~\ref{rem:BCCocart} implies that we may assume without loss of generality $A\simeq 1$. Consider the pushout square
	\begin{equation*}
	\begin{tikzcd}
		(\I{P}^\natural_\sharp)^\sharp\arrow[r]\arrow[d, hookrightarrow] & ((\I{P}^\natural_\sharp)^\gp)^\sharp\arrow[d]\\
		\I{P}^\natural\arrow[r] & Z
	\end{tikzcd}
	\end{equation*}
	in $\mSimp{\BB}$. Note that the span in the upper left corner of the first square is obtained by applying the functor $(-)\vert_{\Delta}$ to the span in the upper left corner of the second square. We claim that $Z\vert_{\Delta}^\flat\to Z$ is an equivalence. In fact, since the object $Z_+\in\BB$ is computed as the pushout
	\begin{equation*}
	\begin{tikzcd}
	P_+\arrow[d, "\id"]\arrow[r] & (\I{P}^\natural_\sharp)^\gp\arrow[d]\\
	P_+\arrow[r] & Z_+,
	\end{tikzcd}
	\end{equation*}
	the map $(\I{P}^\natural_\sharp)^\gp\to Z_+$ must be an equivalence. But since by the same argument the map $(\I{P}^\natural_\sharp)^\gp\to Z_0$ is an equivalence as well, the claim holds. Now by construction, the map $\I{P}^\natural\to Z$ is contained in the internal saturation of $s^0\colon (\Delta^1)^\sharp\to 1$ and is therefore in particular marked left anodyne. Moreover, if we denote by $\I{Z}\in\Cat(\BB)$ the image of $Z\vert_{\Delta}$ along the localisation functor $L\colon\Simp{\BB}\to\Cat(\BB)$, the associated localisation map $Z\vert_{\Delta}\to\I{Z}$ induces a marked left anodyne map $Z\to \I{Z}^\flat$. In total, we therefore obtain a marked left anodyne map $\I{P}^\natural\to \I{Z}^\flat$, which implies that the image of $p$ along $(\pi_{\I{C}})_!$ is given by $\I{Z}$. As the $\BB$-category $\I{P}_\sharp^{-1}\I{P}$ is precisely computed by applying the functor $L(-)\vert_{\Delta}$ to the pushout square that defines $Z$, we obtain the desired equivalence $\I{P}_\sharp^{-1}\I{P}\simeq \I{Z}$.
\end{remark}

\begin{remark}
	\label{rem:pushforwardCocartPoint}
	If $\I{C}$ is an arbitrary $\BB$-category, Proposition~\ref{prop:CocartBaseChangeRightAdjoint} and the fact that $\pi_{\I{C}^\sharp}\colon \I{C}^\sharp\to 1$ is proper imply that the functor $(\pi_{\I{C}})^\ast\colon \ICat_{\BB}\simeq\ICocart_{1}\to \ICocart_{\I{C}}$ admits a right adjoint $(\pi_{\I{C}})_\ast$. On global sections, this functor is given by restricting the geometric morphism
	\begin{equation*}
	(\pi_{\I{C}^\sharp})_\ast\colon \Over{(\mSimp\BB)}{\I{C}^\sharp}\to\mSimp\BB
	\end{equation*}
	to marked cocartesian fibrations. Recall that this map is equivalently given by the functor $\Over{\iFun{\I{C}^\sharp}{-}}{\I{C}^\sharp}$ that sends a map $p\colon P\to \I{C}^\sharp$ to the pullback
	\begin{equation*}
	\begin{tikzcd}
	\Over{\iFun{\I{C}^\sharp}{P}}{\I{C}^\sharp}\arrow[r]\arrow[d] & {\iFun{\I{C}^\sharp}{P}}\arrow[d, "p_\ast"]\\
	1\arrow[r, "\id_{\I{C}^\sharp}"] & {\iFun{\I{C}^\sharp}{\I{C}^\sharp}}.
	\end{tikzcd}
	\end{equation*}
	We thus conclude that the functor $(\pi_{\I{C}})_\ast$ carries a cocartesian fibration $p\colon\I{P}\to\I{C}$ to the associated $\BB$-category $(\Over{\iFun{\I{C}^\sharp}{\I{P}^\natural}}{\I{C}^\sharp})\vert_{\Delta}$ of \emph{cocartesian sections} of $p$.
\end{remark}

\begin{remark}
	\label{rem:pushforwardCocartRightFibration}
	By combining Proposition~\ref{prop:CocartBaseChangeRightAdjoint} with Proposition~\ref{prop:rightFibrationsProper}, one deduces in particular that the base change functor $p^\ast\colon \ICocart_{\I{C}}\to\ICocart_{\I{P}}$ along any right fibration $p\colon \I{P}\to\I{C}$ in $\Cat(\BB)$ admits a right adjoint $p_\ast$.
\end{remark}

\begin{remark}
	\label{rem:CartLimitsColimits}
	Note that taking opposite $\BB$-categories defines an equivalence $(-)^\op\colon \ICart_{\I{C}}\simeq\ICocart_{\I{C}^\op}$ that is natural in $\I{C}\in\Cat(\BB)$. Therefore, the results that have been established in this section can also be dualised to cartesian fibrations.
\end{remark}

\section{Straightening and unstraightening}
\label{sec:SU}
The main goal of this chapter is to construct an equivalence $\ICocart_{\I{C}}\simeq\iFun{\I{C}}{\ICat_{\BB}}$ that is natural in $\I{C}\in\Cat(\BB)$. For $\infty$-categories, such an equivalence has first been established by Lurie in~\cite{htt}, who referred to the functor from cocartesian fibrations to $\CatS$-valued functors as \emph{straightening} and to its inverse as \emph{unstraightening}. We will make use of the same terminology here, although our constructions will be substantially different from those in Lurie's approach. We construct a straightening functor in \S~\ref{sec:straightening} and its left adjoint in \S~\ref{sec:unstraightening}. In \S~\ref{sec:straighteningEquivalence}, we prove that this adjunction defines an equivalence of $\BB$-categories. As a consequence, one obtains a \emph{universal} cocartesian fibration over $\ICat_{\BB}$ which is studied in \S~\ref{sec:universalCocartesianFibration}. We close this chapter by giving an explicit description of the straightening functor in the special case where the base $\BB$-category is the interval $\Delta^1$ in~\S~\ref{sec:straighteningInterval}.

\subsection{The straightening functor}
\label{sec:straightening}
Recall from~\cite[Proposition~3.2.12]{Martini2021a} that $\ICat_{\BB}$ is a reflective subcategory of $\IPSh_{\Univ}(\Delta)$. Let $\Delta^{\bullet}\colon\Delta\into\IPSh_{\Univ}(\Delta)$ be the Yoneda embedding. For any $n\geq 0$, the presheaf represented by $\ord{n}\colon 1\to \Delta$ is given by $\Delta^n\in\Simp\BB$ and therefore by a $\BB$-category. Since $\Delta$ is a constant $\BB$-category and therefore generated by the collection of global objects $\ord{n}\colon 1\to \Delta$, this shows that the Yoneda embedding defines a functor $\Delta^\bullet\colon\Delta\into\ICat_{\BB}$. Therefore, Corollary~\ref{cor:CocartPoweredTensored} implies that $\ICocart$ is both tensored and powered over $\Delta\vert^{\Simp\BB}$ as well. Let $\ILFib$ be the large $\Simp\BB$-category that corresponds to the sheaf of left fibrations $\LFib$ on $\Simp\BB$. By~\cite[Corollary~3.2.4]{Martini2021a}, the coreflection $(-)_\sharp\colon\Cocart\to\LFib$ from Proposition~\ref{prop:coreflectionLeftFibrationsCocartesianFibrations} determines an internal right adjoint $(-)_\sharp\colon \ICocart\to\ILFib$ to the inclusion $\ILFib\into\ICocart$. We define the \emph{straightening} functor $\St$ as the composition
\begin{equation*}
\St\colon \ICocart\to \iFun{(\Delta\vert^{\Simp\BB})^{\op}}{\ICocart}\to \iFun{(\Delta\vert^{\Simp\BB})^{\op}}{\ILFib}
\end{equation*}
in which the first map is the transpose of the powering bifunctor and the second map is given by postcomposition with $(-)_\sharp$. Given $\I{C}\in\Cat(\BB)$, applying the functor $\iFun{\I{C}}{-}\vert_{\BB}$ to $\St$ gives rise to a functor
\begin{equation*}
\St_{\I{C}}\colon \ICocart_{\I{C}}\to \iFun{\Delta^{\op}}{\ILFib_{\I{C}}}\simeq  \iFun{\I{C}}{\IPSh_{\Univ}(\Delta)}.
\end{equation*}

\begin{remark}
	\label{rem:straighteningTranposePowering}
	The straightening functor $\St_{\I{C}}\colon \ICocart_{\I{C}}\to\iFun{\Delta^{\op}}{\ILFib_{\I{C}}}$ is transpose to the composition
	\begin{equation*}
	\Delta^{\op}\times\ICocart_{\I{C}}\xrightarrow{(-)^{\Delta^{\bullet}}} \ICocart_{\I{C}}\xrightarrow{(-)_\sharp}\ILFib_{\I{C}}
	\end{equation*}
	in which the first map is the restriction of the powering of $\ICocart_{\I{C}}$ over $\ICat_{\BB}$ along the Yoneda embedding $\Delta^{\bullet}$ and where the second map is the coreflection functor.
\end{remark}

\begin{remark}
 	In the case $\BB=\SS$, the above definition of the straightening functor has previously appeared in lecture notes by Hinich~\cite{Hinich2017}.
\end{remark}

\begin{remark}
	\label{rem:StraigeningPoint}
	For the special case $\I{C}= 1$, the straightening functor is given by the inclusion 
	\begin{equation*}
	\ICocart_{1}\simeq\ICat_{\BB}\into\IPSh_{\Univ}(\Delta).
	\end{equation*}
	In fact, since Remark~\ref{rem:CocartCTensoringPoweringExplicitly} implies that in this case the tensoring is simply given by the product bifunctor $-\times-\colon\ICat_{\BB}\times\ICat_{\BB}\to\ICat_{\BB}$, the powering functor is the internal mapping object of $\ICat_{\BB}$, see~\cite[Corollary~4.5.5]{Martini2021a}. Since the coreflection $(-)_\sharp\colon \ICocart_{\I{C}}\to\ILFib_{\I{C}}$ reduces to the core $\BB$-groupoid functor when $\I{C}=1$, Remark~\ref{rem:straighteningTranposePowering} and~\cite[Remark~4.5.6]{Martini2021a} imply that the straightening functor $\St_1\colon \ICat_{\BB}\to \IPSh_{\Univ}(\Delta)$ is transpose to
	\begin{equation*}
	\map{\ICat_{\BB}}(\Delta^{\bullet},-)\colon \Delta^{\op}\times\ICat_{\BB}\to\Univ.
	\end{equation*}
	and is therefore given by the inclusion $\ICat_{\BB}\into\IPSh_{\Univ}(\Delta)$ on account of Yoneda's lemma.
\end{remark}

\begin{proposition}
	\label{prop:StraighteningIsCategory}
	For any $\BB$-category $\I{C}$, the staightening functor $\St_{\I{C}}\colon\ICocart_{\I{C}}\to\iFun{\I{C}}{\IPSh_{\Univ}(\Delta)}$ takes values in the full subcategory $\iFun{\I{C}}{\ICat_{\BB}}$.
\end{proposition}
\begin{proof}
	Let $p\colon \I{P}\to A\times\I{C}$ be a cocartesian fibration in context $A\in\BB$. We need to show that the functor $\St_{\I{C}}(p)\colon A\times\I{C}\to\IPSh_{\Univ}(\Delta)$ takes values in $\ICat_{\BB}$. Upon replacing $\BB$ with $\Over{\BB}{A}$, we may assume without loss of generality $A\simeq 1$, cf.\ Remark~\ref{rem:BCCocart}. We can argue object-wise in $\I{C}$, i.e.\ it suffices to show that for every object $c\colon A\to \I{C}$ the simplicial object $\St_{\I{C}}(p)(c)\in\Simp{(\Over{\BB}{A})}$ is a $\Over{\BB}{A}$-category. Again, we can assume $A\simeq 1$. In light of the naturality of the straightening functor, this argument implies that we may reduce to the case $\I{C}\simeq 1$. In this case, Remark~\ref{rem:StraigeningPoint} shows that the straightening functor is simply the inclusion $\ICat_{\BB}\into\IPSh_{\Univ}(\Delta)$, hence the claim follows.
\end{proof}

We conclude this section with the observation that for every $\BB$-category $\I{C}$, restricting the straightening functor $\St_{\I{C}}$ along the inclusion $\ILFib_{\I{C}}\into\ICocart_{\I{C}}$ recovers the equivalence $\ILFib_{\I{C}}\simeq\IPSh_{\Univ}(\I{C}^{\op})$ that is induced by (the inverse of) the Grothendieck construction~\cite[Theorem~4.5.1]{Martini2021}. More precisely, one has:
\begin{proposition}
	\label{prop:StraighteningGrothendieckConstruction}
	There is a commutative square
	\begin{equation*}
	\begin{tikzcd}
	\ILFib_{\I{C}}\arrow[d, hookrightarrow]\arrow[r, "\simeq"] & \IPSh_{\Univ}(\I{C}^{\op})\arrow[d, "\iota_\ast", hookrightarrow]\\
	\ICocart_{\I{C}}\arrow[r, "\St_{\I{C}}"] & \iFun{\I{C}}{\IPSh_{\Univ}(\Delta)}
	\end{tikzcd}
	\end{equation*}
	in which $\iota\colon \Univ\into\IPSh_{\Univ}(\Delta)$ denotes the diagonal embedding.
\end{proposition}
\begin{proof}
	We need to show that for every left fibration $p\colon \I{P}\to A\times\I{C}$ in context $A\in\BB$ the straightening $\St_{\I{C}}(p)\colon A\times\I{C}\to\IPSh_{\Univ}(\Delta)$ factors through $\iota\colon\Univ\into\IPSh_{\Univ}(\Delta)$. As in the proof of Proposition~\ref{prop:StraighteningIsCategory}, the fact that the straightening functor is natural in $\I{C}$ and the fact that we may work object-wise in $\I{C}$ allows us to reduce to the case where $\I{C}\simeq 1\simeq A$, in which case the result also immediately follows from Remark~\ref{rem:StraigeningPoint}.
\end{proof}

\begin{remark}
	\label{rem:straighteningCart}
	Dually, by making use of the powering of $\ICart$ over $\ICat_{\BB}\vert^{\Simp\BB}$ (Remark~\ref{rem:definitionCartC}), one can construct a straightening functor
	\begin{equation*}
	\St\colon \ICart\to \iFun{(\Delta\vert^{\Simp\BB})^{\op}}{\IRFib}
	\end{equation*}
	which induces functors $\St_{\I{C}}\colon \ICart_{\I{C}}\to\iFun{\Delta^{\op}}{\IRFib_{\I{C}}}$ for every $\BB$-category $\I{C}$. Note that by making use of the equivalence $(-)^\op\colon\IRFib_{\I{C}}\simeq\ILFib_{\I{C}^\op}$, the codomain of $\St_{\I{C}}$ can be identified with $\iFun{\I{C}^{\op}}{\IPSh_{\Univ}(\Delta)}$.
	Now in light of Remark~\ref{rem:tensoringPoweringCocart+Explicitly}, the equivalence $(-)^\op\colon\ICart_{\I{C}}\simeq\ICocart_{\I{C}^\op}$ from Remark~\ref{rem:definitionCartC} fits into a commutative square
	\begin{equation*}
	\begin{tikzcd}
	\Delta\times\ICart_{\I{C}}\arrow[r, "\Delta^\bullet\otimes -"] \arrow[d, "\op\times (-)^\op"] & \ICocart_{\I{C}}\arrow[d, "(-)^\op"]\\
	\Delta\times\ICocart_{\I{C}^\op}\arrow[r, "\Delta^\bullet\otimes -"] & \ICocart_{\I{C}^\op}.
	\end{tikzcd}
	\end{equation*}
	By making use of the adjunction between tensoring and powering, this square in turn induces a commutative diagram
	\begin{equation*}
	\begin{tikzcd}
	\ICart_{\I{C}}\arrow[d, "(-)^\op"]\arrow[r, "(-)^{\Delta^\bullet}"] & \iFun{\Delta^{\op}}{\ICart_{\I{C}}}\arrow[d, "(-)^\op_\ast"]\\
	\ICocart_{\I{C}^\op}\arrow[r, "(-)^{\Delta^{\bullet,\op}}"] & \iFun{\Delta^{\op}}{\ICocart_{\I{C}^\op}}.
	\end{tikzcd}
	\end{equation*}
	Upon combining this observation with Remark~\ref{rem:straighteningTranposePowering} (and its dual) and the evident fact that the coreflection of (co)cartesian fibrations into right (left) fibrations commutes with taking opposite $\BB$-categories, we conclude that there is a commutative square
	\begin{equation*}
	\begin{tikzcd}
	\ICart_{\I{C}}\arrow[r, "\St_{\I{C}}"]\arrow[d, "(-)^\op"] & \iFun{\I{C}^\op}{\ICat_{\BB}}\arrow[d, "(-)^\op_\ast"]\\
	\ICocart_{\I{C}^\op}\arrow[r, "\St_{\I{C}^\op}"] & \iFun{\I{C}^{\op}}{\ICat_{\BB}}.
	\end{tikzcd}
	\end{equation*}
\end{remark}

\subsection{The unstraightening functor}
\label{sec:unstraightening}
In this section we construct a left adjoint to the straightening functor $\St_{\I{C}}$ for each $\BB$-category $\I{C}$. Given $\I{C}\in\Cat(\BB)$, restricting the tensoring $-\otimes -\colon\ICat_{\BB}\times\ICocart_{\I{C}}\to \ICocart_{\I{C}}$ along the inclusion
\begin{equation*}
\Delta^{\bullet}\times h_{\I{C}^{\op}}\colon\Delta\times\I{C}^{\op}\into \ICat_{\BB}\times\ILFib_{\I{C}}\into\ICat_{\BB}\times\ICocart_{\I{C}}
\end{equation*}
that is induced by the Yoneda embedding on either factor gives rise to a functor
\begin{equation*}
\Delta^{\bullet}\otimes h_{\I{C}^{\op}}(-)\colon \Delta\times\I{C}^{\op}\to\ICocart_{\I{C}}.
\end{equation*}
In light of Proposition~\ref{prop:CocartCocomplete} and the universal property of presheaf $\BB$-categories~\cite[Theorem~7.1.1]{Martini2021a}, we may now define:
\begin{definition}
	\label{def:unstraightening}
	For every $\BB$-category $\I{C}$, the \emph{unstraightening} functor $\Un_{\I{C}}\colon \iFun{\I{C}}{\IPSh_{\Univ}(\Delta)}\to\ICocart_{\I{C}}$ is defined as the left Kan extension of the tensoring $\Delta^{\bullet}\otimes h_{\I{C}^{\op}}(-)\colon \Delta\times\I{C}^{\op}\to\ICocart_{\I{C}}$ along the Yoneda embedding
	\begin{equation*}
	h_{\Delta\times\I{C}^{\op}}\colon\Delta\times\I{C}^{\op}\into\IPSh_{\Univ}(\Delta\times\I{C}^{\op})\simeq\iFun{\I{C}}{\IPSh_{\Univ}(\Delta)}.
	\end{equation*} 
\end{definition}

Our next goal is to show that the unstraightening functor $\Un_{\I{C}}$ is left adjoint to the straightening functor $\St_{\I{C}}$. To that end, recall that a large $\BB$-category $\I{C}$ is locally small if the left fibration $\Tw(\I{C})\to\I{C}^{\op}\times\I{C}$ is small, cf.~\cite[\S~4.7]{Martini2021}.
\begin{lemma}
	\label{lem:BaseChangeLocallySmall}
	Let $f_\ast\colon \BB\to\AA$ be a geometric morphism of $\infty$-topoi, and let $\I{C}$ be a locally small $\BB$-category. Then $f_\ast\I{C}$ is locally small.
\end{lemma}
\begin{proof}
	In light of Lemma~\ref{lem:geometricMorphismPowering}, it suffices to show that $f_\ast$ preserves small left fibrations. In other words, we need to show that if $p\colon \I{P}\to \I{C}$ is a small left fibration in $\Cat(\BBB)$, the left fibration $f_\ast p$ in $\Cat(\AAA)$ is small too. Consider a cartesian square
	\begin{equation*}
	\begin{tikzcd}
	(f_\ast\I{P})\vert_c\arrow[r]\arrow[d] & f_\ast \I{P}\arrow[d, "f_\ast p"]\\
	A\arrow[r, "c"] & f_\ast \I{C}
	\end{tikzcd}
	\end{equation*}
	where $A\in\AA$ is an arbitrary object. By~\cite[Proposition~4.7.2]{Martini2021}, it suffices to show that $(f_\ast\I{P})\vert_c$ is a small $\AA$-groupoid. The transpose of $c\colon A\to f_\ast\I{C}$ defines an object $c^\prime\colon f^\ast A\to\I{C}$ of $\I{C}$ in context $f^\ast A\in\BB$. By assumption, the $\BB$-groupoid $\I{P}\vert_{c^\prime}$ is small. Hence $f_\ast(\I{P}\vert_{c^\prime})$ is a small $\AA$-groupoid. The claim now follows from the observation that $(f_\ast\I{P})\vert_c$ arises as the pullback of $f_\ast(\I{P}\vert_{c^\prime})$ along the adjunction unit $A\to f_\ast f^\ast A$.
\end{proof}

\begin{proposition}
	\label{prop:CocartCLocallySmall}
	For every $\I{C}\in\Cat(\BB)$, the large $\BB$-category $\ICocart_{\I{C}}$ is locally small.
\end{proposition}
\begin{proof}
	Being a full subcategory of $\Univ[\mSimp\BB]$, the large $\mSimp\BB$-category $\ICocart^+$ is locally small. Using Lemma~\ref{lem:BaseChangeLocallySmall}, the large $\Simp\BB$-category $\ICocart$ is locally small as well, hence so is $\iFun{\I{C}}{\ICocart}$, see~\cite[Proposition~4.7.6]{Martini2021}. Applying Lemma~\ref{lem:BaseChangeLocallySmall} once more, we conclude that $\ICocart_{\I{C}}=\iFun{\I{C}}{\ICocart}\vert_{\BB}$ is locally small, as desired.
\end{proof}

As a result of Proposition~\ref{prop:CocartCLocallySmall}, we deduce from~\cite[Remark~7.1.4]{Martini2021a} that the unstraightening functor $\Un_{\I{C}}$ admits a right adjoint $r$. The computation
\begin{align*}
r &\simeq\map{\IPSh_{\Univ}(\Delta\times\I{C}^{\op})}(h_{\Delta\times\I{C}^{\op}}(-,-), r(-))\\
&\simeq \map{\ICocart_{\I{C}}}(\Delta^{\bullet}\otimes h_{\I{C}^{\op}}(-), -)\\
&\simeq\map{\ICocart_{\I{C}}}(h_{\I{C}^{\op}}(-), (-)^{\Delta^\bullet})\\
&\simeq\map{\ILFib_{\I{C}}}(h_{\I{C}^{\op}}(-), (-)^{\Delta^{\bullet}}_\sharp)\\
&\simeq (-)^{\Delta^\bullet}_\sharp
\end{align*}
and Remark~\ref{rem:straighteningTranposePowering} now show:
\begin{proposition}
	\label{prop:adjunctionStUn}
	The unstraightening functor $\Un_{\I{C}}$ is left adjoint to the straightening functor $\St_{\I{C}}$.\qed
\end{proposition}
As a direct consequence of Proposition~\ref{prop:StraighteningIsCategory} and Proposition~\ref{prop:adjunctionStUn}, one obtains:
\begin{corollary}
	\label{cor:straighteningUnstraigheningAdjunction}
	The straightening and unstraightening functors restrict to an adjunction
	\begin{equation*}
	(\Un_{\I{C}}\dashv \St_{\I{C}})\colon \ICocart_{\I{C}}\leftrightarrows\iFun{\I{C}}{\ICat_{\BB}}
	\end{equation*}
	for every $\BB$-category $\I{C}$.\qed
\end{corollary}

In general, we have no explicit way to compute the unstraightening $\Un_{\I{C}}(f)$ of a functor $f\colon\I{C}\to\IPSh_{\Univ}(\Delta)$ unless $f$ is contained in the image of the Yoneda embedding $h_{\Delta\times\I{C}^{\op}}$, in which case the unstraightening is simply given by the tensoring in $\ICocart_{\I{C}}$. We conclude this section by explaining how this description extends to a slightly larger class of functors.

\begin{lemma}
	\label{lem:YonedaEmbeddingProduct}
	Let $\I{C}$ and $\I{D}$ be $\BB$-categories. Then there is a commutative square
	\begin{equation*}
	\begin{tikzcd}
	\I{C}\times\I{D}\arrow[r, "h_{\I{C}\times\I{D}}"]\arrow[d, "h_{\I{C}}\times h_{\I{D}}"] & \IPSh_{\Univ}(\I{C}\times\I{D})\\
	\IPSh_{\Univ}(\I{C})\times\IPSh_{\Univ}(\I{D})\arrow[r, "\pr_0^\ast\times\pr_1^\ast"] & \IPSh_{\Univ}(\I{C}\times\I{D})\times\IPSh_{\Univ}(\I{C}\times\I{D})\arrow[u, "-\times-"']
	\end{tikzcd}
	\end{equation*}
	in which $\pr_0\colon \I{C}\times\I{D}\to\I{C}$ and $\pr_1\colon\I{C}\times\I{D}\to\I{D}$ are the two projections.
\end{lemma}
\begin{proof}
	 On account of the evident equivalence $\Tw(\I{C}\times\I{D})\simeq\Tw(\I{C})\times\Tw(\I{D})$ over $(\I{C}\times\I{D})^\op\times(\I{C}\times\I{D})$, there is a cartesian square
	 \begin{equation*}
	 \begin{tikzcd}[column sep=large]
	 \Tw(\I{C}\times\I{D})\arrow[d]\arrow[r] & \Under{\Univ}{1}\times\Under{\Univ}{1}\arrow[d]\\
	 (\I{C}^{\op}\times\I{C})\times(\I{D}^\op\times\I{D})\arrow[r, "\map{\I{C}}\times\map{\I{D}}"] & \Univ\times\Univ.
	 \end{tikzcd}
	 \end{equation*}
	 In light of the equivalence $\Under{\Univ}{1_{\Univ}}\times\Under{\Univ}{1_{\Univ}}\simeq\Under{(\Univ\times\Univ)}{(1_{\Univ},1_{\Univ})}$, the left fibration $\Under{\Univ}{1_{\Univ}}\times\Under{\Univ}{1_{\Univ}}\to\Univ\times\Univ$ is classified by the copresheaf $\Univ\times\Univ\to\Univ$ that is represented by $(1_{\Univ},1_{\Univ})\colon 1\to\Univ\times\Univ$. By~\cite[Corollary~4.4.7]{Martini2021a}, this copresheaf is the product functor $-\times -\colon \Univ\times\Univ\to\Univ$. We therefore obtain a commutative square
	 \begin{equation*}
	 \begin{tikzcd}[column sep=huge]
	 (\I{C}\times\I{D})^{\op}\times(\I{C}\times\I{D})\arrow[d, "\simeq"] \arrow[r, "\map{\I{C}\times\I{D}}"] & \Univ\\
	 (\I{C}^{\op}\times\I{C})\times(\I{D}^\op\times\I{D})\arrow[r, "\map{\I{C}}\times\map{\I{D}}"] & \Univ\times\Univ\arrow[u, "-\times-"'].
	 \end{tikzcd}
	 \end{equation*}
	 By transposing along the adjunction $(\I{C}\times\I{D})^{\op}\times -\dashv \iFun{(\I{C}\times\I{D})^{\op}}{-}$, this square translates into a commutative square
	\begin{equation*}
	\begin{tikzcd}
	\I{C}\times\I{D}\arrow[r, "h_{\I{C}\times\I{D}}"]\arrow[d, "h_{\I{C}}\times h_{\I{D}}"] & \IPSh_{\Univ}(\I{C}\times\I{D})\\
	\IPSh_{\Univ}(\I{C})\times\IPSh_{\Univ}(\I{D})\arrow[r] & \iFun{(\I{C}\times\I{D})^{\op}}{\Univ\times\Univ}\arrow[u, "(-\times-)_\ast"']
	\end{tikzcd}
	\end{equation*}
	in which the lower horizontal map is the transpose of
	\begin{equation*}
	\ev_{\I{C}}\times\ev_{\I{D}}\colon \IPSh_{\Univ}(\I{C})\times\I{C}^{\op}\times\IPSh_{\Univ}(\I{D})\times\I{D}^{\op}\to\Univ\times\Univ.
	\end{equation*}
	It now suffices to observe that with respect to the equivalence $\iFun{(\I{C}\times\I{D})^{\op}}{\Univ\times\Univ}\simeq\IPSh_{\Univ}(\I{C}\times\I{D})\times\IPSh_{\Univ}(\I{C}\times\I{D})$, the lower horizontal map corresponds to $\pr_0^\ast\times\pr_1^\ast$ and the right vertical map corresponds to the product functor on $\IPSh_{\Univ}(\I{C}\times\I{D})$.
\end{proof}
\begin{lemma}
	\label{lem:productFunctorCocontinuous}
	Let $\I{C}$ be a $\BB$-category and let $f\colon\I{C}^{\op}\to\Univ$ be a presheaf on $\I{C}$. Then the product functor $f\times-\colon\IPSh_{\Univ}(\I{C})\to\IPSh_{\Univ}(\I{C})$ has a right adjoint.
\end{lemma}
\begin{proof}
	Let $p\colon \I{P}\to\I{C}$ be the right fibration that is classified by $f$. Then $f\times -$ corresponds to the product functor $\I{P}\times -\colon \IRFib_{\I{C}}\to\IRFib_{\I{C}}$. On local sections over $A\in\BB$, this functor is given by the composition
	\begin{equation*}
	\RFib(\pi_A^\ast\I{C})\xrightarrow{p^\ast} \RFib(\pi_A^\ast\I{P})\xrightarrow{p_!} \RFib(\pi_A^\ast\I{C}).
	\end{equation*}
	By the theory of Kan extensions~\cite[\S~6]{Martini2021a} and the fact that $\Univ$ is complete, the functor $p^\ast\colon \IRFib_{\I{C}}\to\IRFib_{\I{P}}$ has a right adjoint $p_\ast$, which implies that $\I{P}\times -$ section-wise admits a right adjoint that is given by the composition $p_\ast p^\ast$. Now if $s\colon B\to A$ is a map in $\BB$, the mate transformation $s^\ast p_\ast p^\ast\to p_\ast p^\ast s^\ast$ is given by the composition
	\begin{equation*}
	s^\ast p_\ast p^\ast \to p_\ast s^\ast p^\ast\simeq  p_\ast p^\ast s^\ast
	\end{equation*}
	in which the first map is induced by the mate transformation $s^\ast p_\ast \to p_\ast s^\ast$. Since $p_\ast$ is an \emph{internal} right adjoint of $p^\ast$, this map must be an equivalence. Using~\cite[Proposition~3.2.8]{Martini2021a}, the claim follows.
\end{proof}

Lemma~\ref{lem:YonedaEmbeddingProduct} implies that there is a commutative triangle
\begin{equation*}
\begin{tikzcd}[column sep=large]
	\Delta\times\I{C}^{\op}\arrow[r, "{h_{\Delta\times\I{C}^{\op}}}"]\arrow[d, "\id_{\Delta}\times h_{\I{C}^{\op}}"] & \IPSh_{\Univ}(\Delta\times\I{C}^{\op})\\
	\Delta\times\IPSh_{\Univ}(\I{C}^{\op}).\arrow[ur, "\pr_0^\ast\Delta^\bullet\times\pr_1^\ast(-)"'] &
\end{tikzcd}
\end{equation*}
We are now able to compute the unstraightening of all functors in $\iFun{\I{C}}{\IPSh_{\Univ}(\Delta)}$ that lie in the image of $\pr_0^\ast \Delta^{\bullet}\times\pr_1^\ast(-)$:
\begin{proposition}
	\label{prop:UnstraighteningExplicitly}
	There is a commutative square
	\begin{equation*}
	\begin{tikzcd}[column sep={between origins,12em}]
	\Delta\times\IPSh_{\Univ}(\I{C}^{\op})\arrow[d, "\simeq"]\arrow[r, "\pr_0^\ast \Delta^\bullet\times\pr_1^\ast(-)"] &\iFun{\I{C}}{\IPSh_{\Univ}(\Delta)}\arrow[d, "\Un_{\I{C}}"]\\
	\Delta\times\ILFib_{\I{C}}\arrow[r, "\Delta^\bullet\otimes -"] & \ICocart_{\I{C}}.
	\end{tikzcd}
	\end{equation*}
\end{proposition}
\begin{proof}
	By construction of the unstraightening functor, the square commutes when restricted along the inclusion $\id_{\Delta}\times h_{\I{C}^{\op}}\colon \Delta\times\I{C}^{\op}\into\Delta\times\IPSh_{\Univ}(\I{C}^{\op})$. By making use of the universal property of presheaf $\BB$-categories~\cite[Theorem~7.1.1]{Martini2021a} and the fact that $\Un_{\I{C}}$ is a left adjoint functor and therefore cocontinuous~\cite[Proposition~5.1.5]{Martini2021a}, it is enough to show that for every integer $n\geq 0$ both $\Delta^n\otimes -$ and $\pr_0^\ast(\Delta^n)\times\pr_1^\ast(-)$ are cocontinuous as well. For the first functor, this follows from the observation that it has a right adjoint given by $(-)^{\Delta^n}_\sharp$. Regarding the second functor, since $\pr_1^\ast$ is cocontinuous, it suffices to show that $\pr_0^\ast(\Delta^n)\times-$ is cocontinuous as well, which follows from Lemma~\ref{lem:productFunctorCocontinuous}.
\end{proof}
\subsection{The straightening equivalence}
\label{sec:straighteningEquivalence}
We are finally ready to state and prove the main theorem of this paper:
\begin{theorem}
	\label{thm:StraighteningEquivalence}
	For every $\BB$-category $\I{C}$, the straightening functor $\St_{\I{C}}\colon \ICocart_{\I{C}}\to\iFun{\I{C}}{\ICat_{\BB}}$ is an equivalence of large $\BB$-categories that is natural in $\I{C}\in\Cat(\BB)$.
\end{theorem}
To prove Theorem~\ref{thm:StraighteningEquivalence}, we will show that both the unit $\eta_{\I{C}}\colon \id_{\iFun{\I{C}}{\ICat_{\BB}}}\to \St_{\I{C}}\Un_{\I{C}}$ and the counit $\epsilon_{\I{C}}\colon\Un_{\I{C}}\St_{\I{C}}\to\id_{\ICocart_{\I{C}}}$ of the adjunction $\Un_{\I{C}}\dashv\St_{\I{C}}$ from Corollary~\ref{cor:straighteningUnstraigheningAdjunction} is an equivalence. By Corollary~\ref{cor:fibrewiseCriterionEquivalencesCocartesian} and the fact that equivalences in functor $\BB$-categories are detected objectwise~\cite[Corollary~4.7.17]{Martini2021}, it will be enough to show that for every object $c\colon A\to \I{C}$ the maps $c^\ast\eta_{\I{C}}$ and $c^\ast\epsilon_{\I{C}}$ are equivalences. Using Remark~\ref{rem:BCCocart}, we may assume $A\simeq 1$. Since furthermore Remark~\ref{rem:StraigeningPoint} implies that $\St_{1}$ is an equivalence, it will suffice to construct equivalences $c^\ast\eta_{\I{C}}\simeq \eta_1c^\ast$ and $c^\ast\epsilon_{\I{C}}\simeq \epsilon_{1}c^\ast$. In other words, we need to show that the  map $\Un_{1}c^\ast\to c^\ast\Un_{\I{C}}$ that arises as the mate of the commutative square
\begin{equation*}
\begin{tikzcd}
\ICocart_{\I{C}}\arrow[d, "c^\ast"]\arrow[r, "\St_{\I{C}}"] & \iFun{\I{C}}{\ICat_{\BB}}\arrow[d, "c^\ast"]\\
\ICocart_{1}\arrow[r, "\St_{1}"] & \ICat_{\BB}
\end{tikzcd}
\end{equation*}
is an equivalence. This will require a few preparatory steps.

\begin{lemma}
	\label{lem:straighteningTensoringCommute}
	There is a commutative square 
	\begin{equation*}
	\begin{tikzcd}
	\Delta\times\ILFib_{\I{C}}\arrow[d, "\Delta^\bullet\otimes -"]\arrow[r, "\Delta^\bullet\times\int^{-1}"] & \IPSh_{\Univ}(\Delta)\times\IPSh_{\Univ}(\I{C}^{\op})\arrow[d, "\pr_0^\ast(-)\times\pr_1^\ast(-)"]\\
	\ICocart_{\I{C}}\arrow[r, "\St_{\I{C}}"] & \iFun{\I{C}}{\IPSh_{\Univ}(\Delta)}
	\end{tikzcd}
	\end{equation*}
	where $\smallint\colon \IPSh_{\Univ}(\I{C}^{\op})\simeq\ILFib_{\I{C}}$ denotes the Grothendieck construction.
\end{lemma}
\begin{proof}
	Being a right adjoint, the straightening functor commutes with products. Thus the claim follows from Proposition~\ref{prop:StraighteningGrothendieckConstruction} as well as from combining~\ref{rem:StraigeningPoint} with the naturality of straightening.
\end{proof}

The argument in the proof of the lemma below was communicated to the author by Maxime Ramzi:
\begin{lemma}
	\label{lem:AdjunctionUnitEquivalenceCriterion}
	Let $(l\dashv r)\colon \I{C}\leftrightarrows\I{D}$ be an adjunction between (not necessarily small) $\BB$-categories in which $\I{D}$ is cocomplete and locally small. Let $\I{E}^0\into\I{E}$ be a full inclusion of $\BB$-categories where $\I{E}^0$ is small and $\I{E}$ is locally small, and let $f\colon \I{E}\to\I{D}$ be a functor such that $f$ is the left Kan extension of $fi$ along $i$ and the identity on $\I{D}$ is the left Kan extension of $fi$ along itself. Suppose furthermore that there is an arbitrary equivalence $\phi\colon f\simeq rlf$. Then the adjunction unit $\eta$ induces an equivalence $\eta f\colon f\simeq rl f$.
\end{lemma}
\begin{proof}
	Note that by~\cite[Corollary~6.3.7]{Martini2021a} the assumptions on $\I{E}^0$, $\I{E}$ and $\I{D}$ make sure that the functors of left Kan extension exist.
	Using the triangle identities, one can construct a commutative diagram
	\begin{equation*}
	\begin{tikzcd}
	f\arrow[r, "\eta f"]\arrow[d, "\phi"] \arrow[rr, bend left,  "\id"]& rlf\arrow[r, dashed] \arrow[d, "rl\phi"]& f\arrow[d, "\phi"]\\
	rlf\arrow[r, "rl\eta f"] \arrow[rr, bend right, "\id"]& rlrlf\arrow[r, "r\epsilon lf"] & rlf
	\end{tikzcd}
	\end{equation*}
	in which $\epsilon$ denotes the adjunction counit. Therefore, the map $s=\phi^{-1}\eta f\colon f\to f$ admits a retraction $r$ that is obtained by composing $\phi$ with the dashed arrow in the above diagram. We complete the proof by showing that $s$ is an equivalence. Since by assumption $f$ is the left Kan extension of $fi$ along $i$ and since $i$ is fully faithful, the functor of left Kan extension $i_!$ being fully faithful~\cite[Theorem~6.3.6]{Martini2021a} implies that it suffices to show that $i^\ast(s)$ is an equivalence. Since furthermore the identity on $\I{D}$ is the left Kan extension of $fi$ along itself, the two maps $i^\ast(s)$ and $i^\ast(r)$ induce maps $s^\prime\colon \id_{\I{D}}\to\id_{\I{D}}$ and $r^\prime\colon \id_{\I{D}}\to\id_{\I{D}}$ such that $r^\prime s^\prime\simeq\id$ and such that $i^\ast f^\ast(s^\prime)\simeq i^\ast(s)$. It therefore suffices to show that $s^\prime$ is an equivalence. Given $d\colon A\to \I{D}$, naturality of $s^\prime$ implies that there is a commutative square
	\begin{equation*}
	\begin{tikzcd}
	d\arrow[r, "s^\prime(d)"]\arrow[d, "r^\prime(d)"] & d\arrow[d, "r^\prime(d)"]\\
	d\arrow[r, "s^\prime(d)"] & d,
	\end{tikzcd}
	\end{equation*}
	hence $r^\prime(d)$ is both a left and right inverse of $s^\prime(d)$, which implies that $s^\prime(d)$ is an equivalence. As $d$ was chosen arbitrarily, the result follows.
\end{proof}

\begin{proposition}
	\label{prop:CocartBCCocontinuousMate}
	Let $f\colon \I{D}\to\I{C}$ be a functor in $\Cat(\BB)$ such that $f^\ast\colon \ICocart_{\I{C}}\to\ICocart_{\I{D}}$ is cocontinuous. Then the mate of the commutative square
	\begin{equation*}
	\begin{tikzcd}
	\ICocart_{\I{C}}\arrow[r, "\St_{\I{C}}"]\arrow[d, "f^\ast"] & \iFun{\I{C}}{\IPSh_{\Univ}(\Delta)}\arrow[d, "f^\ast"]\\
	\ICocart_{\I{D}}\arrow[r, "\St_{\I{D}}"] & \iFun{\I{D}}{\IPSh_{\Univ}(\Delta)}
	\end{tikzcd}
	\end{equation*}
	is an equivalence.
\end{proposition}
\begin{proof}
	By Proposition~\ref{prop:CocartCocomplete} and the assumption on the functor $f$, the mate transformation $\phi\colon \Un_{\I{D}}f^\ast\to f^\ast\Un_{\I{C}}$ is a map of cocontinuous functors between cocomplete large $\BB$-categories. Using the universal property of presheaf $\BB$-categories~\cite[Theorem~7.1.1]{Martini2021a}, the map $\phi$ is therefore an equivalence whenever its restriction along the Yoneda embedding $h_{\Delta\times\I{C}^{\op}}\colon \Delta\times\I{C}^{\op}\into\iFun{\I{C}}{\IPSh_{\Univ}(\Delta)}$ is one.  By making use of the commutative triangle
	\begin{equation*}
	\begin{tikzcd}[column sep=large]
		\Delta\times\I{C}^{\op}\arrow[r, "{h_{\Delta\times\I{C}^{\op}}}"]\arrow[d, "\id_{\Delta}\times h_{\I{C}^{\op}}"] & \IPSh_{\Univ}(\Delta\times\I{C}^{\op})\\
		\Delta\times\IPSh_{\Univ}(\I{C}^{\op}),\arrow[ur, "\pr_0^\ast\Delta^\bullet\times\pr_1^\ast(-)"'] &
	\end{tikzcd}
	\end{equation*}
	from Lemma~\ref{lem:YonedaEmbeddingProduct}, we might as well show that $\phi (\pr_0^\ast\Delta^{\bullet}\times\pr_1^\ast(-))$ is an equivalence. 
	
	Next, let us show that the functor $\pr_0^\ast\Delta^{\bullet}\times\pr_1^\ast(-)$ is the left Kan extension of $h_{\Delta\times\I{C}^{\op}}$ along $\id_{\Delta}\times h_{\I{C}^{\op}}$. Note that Lemma~\ref{lem:productFunctorCocontinuous} implies that the the associated functor $\Delta\to \iFun{\IPSh_{\Univ}(\I{C}^{\op})}{\iFun{\I{C}}{\IPSh_{\Univ}(\Delta)}}$ takes values in the full subcategory $\iFun{\IPSh_{\Univ}(\I{C}^{\op})}{\iFun{\I{C}}{\IPSh_{\Univ}(\Delta)}}^{\ICat_{\BB}}$ of cocontinuous functors and therefore factors through the inclusion 
	\begin{equation*}
	(h_{\I{C}^{\op}})_!\colon \iFun{\I{C}}{\iFun{\I{C}}{\IPSh_{\Univ}(\Delta)}}\into\iFun{\IPSh_{\Univ}(\I{C}^{\op})}{\iFun{\I{C}}{\IPSh_{\Univ}(\Delta)}},
	\end{equation*}
	where we make use of the universal property of presheaf $\BB$-categories.
	Consequently, $\pr_0^\ast\Delta^{\bullet}\times\pr_1^\ast(-)$ is in the image of the inclusion
	\begin{equation*}
	(\id\times h_{\I{C}^{\op}})_!\colon \iFun{\Delta\times\I{C}^{\op}}{\iFun{\I{C}}{\IPSh_{\Univ}(\Delta)}}\into \iFun{\Delta\times\IPSh_{\Univ}(\I{C}^{\op})}{\iFun{\I{C}}{\IPSh_{\Univ}(\Delta)}},
	\end{equation*}
	as claimed. 
	
	By combining the previous observation with Proposition~\ref{prop:UnstraighteningExplicitly} and Lemma~\ref{lem:straighteningTensoringCommute}, we are in the situation of Lemma~\ref{lem:AdjunctionUnitEquivalenceCriterion} and may therefore conclude that the map $\eta_{\I{C}}(\pr_0^\ast\Delta^{\bullet}\times\pr_1^\ast(-))$ is an equivalence, where $\eta_{\I{C}}$ denotes the unit of the adjunction $\Un_{\I{C}}\dashv\St_{\I{C}}$. As a consequence, since $\phi$ is explicitly given by the composition
	\begin{equation*}
	\Un_{\I{D}}f^\ast \xrightarrow{\Un_{\I{C}}f^\ast\eta_{\I{C}}} \Un_{\I{D}}f^\ast\St_{\I{C}}\Un_{\I{C}}\xrightarrow{\simeq} \Un_{\I{D}}\St_{\I{D}}f^\ast\Un_{\I{C}}\xrightarrow{\epsilon_{\I{D}}f^\ast\Un_{\I{C}}} f^\ast\Un_{\I{C}},
	\end{equation*}
	the proof is finished once we show that also the counit $\epsilon_{\I{D}}f^\ast\Un_{\I{C}}(\pr_0^\ast\Delta^{\bullet}\times\pr_1^\ast(-))$ is an equivalence. But as in light of Proposition~\ref{prop:UnstraighteningExplicitly} and the naturality of tensoring there is an equivalence
	\begin{equation*}
	f^\ast\Un_{\I{C}}(\pr_0^\ast\Delta^{\bullet}\times\pr_1^\ast(-))\simeq \Un_{\I{D}}(\pr_0^\ast\Delta^{\bullet}\times\pr_1^\ast f^\ast(-)),
	\end{equation*}
	the triangle identities for the adjunction $\Un_{\I{D}}\dashv \St_{\I{C}}$ imply that this follows once we prove that the map $\eta_{\I{D}}(\pr_0^\ast\Delta^{\bullet}\times\pr_1^\ast f^\ast(-))$ is an equivalence, which has already been shown above.
\end{proof}

\begin{proof}[{Proof of Theorem~\ref{thm:StraighteningEquivalence}}]
	Let $c\colon 1\to \I{C}$ be an arbitrary global object. As discussed above, we need only show that the natural map $\phi\colon\Un_{1}c^\ast\to c^\ast\Un_{\I{C}}$ is an equivalence. In light of the factorisation of $c$ into the composition $(\pi_c)_!\id_c\colon 1\to\Over{\I{C}}{c}\to\I{C}$ of a final map and a right fibration, the map $\phi$ arises as the mate of the composite square in the commutative diagram
	\begin{equation*}
	\begin{tikzcd}
	\ICocart_{\I{C}}\arrow[r, "\St_{\I{C}}"]\arrow[d, "(\pi_c)_!^\ast"] & \iFun{\I{C}}{\ICat_{\BB}}\arrow[d, "(\pi_c)_!^\ast"]\\
	\ICocart_{\Over{\I{C}}{c}}\arrow[r, "\St_{\Over{\I{C}}{c}}"]\arrow[d, "\id_c^\ast"] & \iFun{\Over{\I{C}}{c}}{\ICat_{\BB}}\arrow[d, "\id_c^\ast"]\\
	\ICocart_{1}\arrow[r, "\St_{1}"] & \ICat_{\BB}.
	\end{tikzcd}
	\end{equation*}
	Therefore, it suffices to show that the mate of each individual square in the diagram commutes. Using Proposition~\ref{prop:CocartBCCocontinuousMate}, this follows once we show that the two vertical maps on the left-hand side of the above diagram are cocontinuous. As for $(\pi_c)_!^\ast$, this is a consequence of Proposition~\ref{prop:CocartBaseChangeRightAdjoint}, so it suffices to consider the map $\id_c^\ast$. Let $\pi_{\Over{\I{C}}{c}}\colon \Over{\I{C}}{c}\to 1$ be the projection. In light of Proposition~\ref{prop:CocartBaseChangeLeftAdjoint}, we obtain a map
	\begin{equation*}
	\phi\colon (\pi_{\Over{\I{C}}{c}})_!\to \id_c^\ast(\id_c)_!(\pi_{\Over{\I{C}}{c}})_!\simeq \id_c^\ast
	\end{equation*}
	in which the first map is given by the unit of the adjunction $(\id_c)_!\dashv\id_c^\ast$ and the equivalence on the right-hand side follows from the fact that  $\id_c$ is a section of $(\pi_{\Over{\I{C}}{c}})$. Since $(\pi_{\Over{\I{C}}{c}})_!$ is a left adjoint and therefore cocontinuous~\cite[Proposition~5.1.5]{Martini2021a}, it thus suffices to verify that $\phi$ is an equivalence. Explicitly, if $p\colon\I{P}\to A\times\I{C}$ is a cocartesian fibration, the map $\phi(p)$ is constructed via the commutative diagram
	\begin{equation*}
	\begin{tikzcd}
	(\I{P}\vert_c)^\flat\arrow[d]\arrow[r, "i"]\arrow[rr, "\phi(p)^\flat", bend left] & \I{P}^\natural\arrow[d, "p^\natural"]\arrow[r, "j"] & (\pi_{\Over{\I{C}}{c}})_!(\I{C})^\flat\arrow[d]\\
	A\arrow[r, "c^\sharp"] & A\times\I{P}^\sharp\arrow[r, "\pr_0"] & A 
	\end{tikzcd}
	\end{equation*}
	in $\mSimp\BB$ in which the left square is cartesian and the right square is defined by the condition that $j$ is marked left anodyne. Since $c$ is final, the map $c^\sharp$ is contained in the internal saturation of $d^0\colon \Delta^0\into(\Delta^1)^\sharp$. Using the dual of Remark~\ref{rem:cartesianPartiallyProper}, we thus conclude that $i$ is marked right anodyne. Note that since $\Cocart(A)\simeq\Cart(A)$ as full subcategories of $\Over{(\mSimp\BB)}{A}$, the map $j$ is simultaneously the reflection map into $\Cocart(A)$ and $\Cart(A)$ and must therefore be marked right anodyne as well. We therefore conclude that also $\phi(p)^\flat$ is a marked right anodyne map. Being a morphism between marked cartesian fibrations over $A$, this is necessarily also a marked cartesian fibration. Hence $\phi(p)$ is an equivalence.
\end{proof}

\begin{remark}
	\label{rem:straighteningEquivalenceCart}
	By combining Theorem~\ref{thm:StraighteningEquivalence} with Remark~\ref{rem:straighteningCart}, one also obtains that the straightening functor $\St_{\I{C}}\colon\ICart_{\I{C}}\to\iFun{\I{C}^{\op}}{\ICat_{\BB}}$ is an equivalence.
\end{remark}

\subsection{The universal cocartesian fibration}
\label{sec:universalCocartesianFibration}
Let $\I{C}$ be a large $\BB$-category. Recall from~\cite[\S~4.5]{Martini2021} that a functor $p\colon \I{P}\to\I{C}$ in $\Cat(\BBB)$ is said to be \emph{small} if for every (small) $\BB$-category $\I{D}$ and every functor $\I{D}\to\I{C}$ the pullback $\I{P}\times_{\I{C}}\I{D}$ is small as well. The collection of {small} cocartesian fibrations defines a subpresheaf $\Cocart^{\bU}\into\Cocart$ of the sheaf $\Cocart$ on $\Cat(\BBB)$. 

\begin{lemma}
	\label{lem:characterisationSmallCocartesianFibration}
	Let $p\colon\I{P}\to\I{C}$ be a cocartesian fibration between large $\BB$-categories. Then $p$ is small if and only if for all objects $c\colon A\to \I{C}$ in context $A\in \BB$ the fibre $\I{P}\vert_c$ is a small $\BB$-category.
\end{lemma}
\begin{proof}
	The condition is clearly necessary. Conversely, it suffices to show that if $\I{C}$ is small and if $\I{P}\vert_c$ is small for all objects $c\colon A\to \I{C}$ in context $A\in\BB$, then $\BB$-category $\I{P}$ is small as well. By letting $c$ be the tautological object $\I{C}_0\to \I{C}$, one finds that $\I{P}_0$ is small. It therefore suffices to show that $\I{P}$ is locally small, see~\cite[Proposition~4.7.4]{Martini2021}. Using~\cite[Proposition~4.7.2]{Martini2021}, we need to show that for any two objects $x,y\colon A\rightrightarrows \I{P}$ in context $A\in\BB$ the (large) mapping $\BB$-groupoid $\map{\I{P}}(x,y)$ is contained in $\BB$. Let $c=p(x)$ and $d=p(y)$. Note that $\map{\I{P}}(x,y)\to \map{\I{C}}(c,d)$ is the fibre of $\map{\I{P}}(x,y)\times_A\map{\I{C}}(c,d)\to\map{\I{C}}(c,d)\times_A\map{\I{C}}(c,d)$ over the diagonal $\map{\I{C}}(c,d)\to\map{\I{C}}(c,d)\times_A\map{\I{C}}(c,d)$. Therefore, by replacing $A$ with $\map{\I{C}}(c,d)$, we may assume that there exists a map $\alpha\colon c\to d$ in context $A$ and that we only have to show that the fibre of $\map{\I{P}}(x,y)\to\map{\I{C}}(c,d)$ over $\alpha$ is contained in $\BB$. Let $f\colon x\to z$ be a cocartesian lift of $\alpha$. We then obtain a cartesian square
	\begin{equation*}
	\begin{tikzcd}
		\map{\I{P}}(z,y)\arrow[r, "f^\ast"]\arrow[d] & \map{\I{P}}(x,y)\arrow[d]\\
		\map{\I{C}}(d,d)\arrow[r, "\alpha^\ast"] & \map{\I{C}}(c,d)
	\end{tikzcd}
	\end{equation*}
	such that the fibre of the left vertical map over the identity $\id_d\colon A\to \map{\I{C}}(d,d)$ recovers the fibre of the right vertical morphism over $\alpha$. Hence this fibre recovers the mapping $\BB$-groupoid $\map{\I{P}\vert_d}(z,y)$ and is therefore small.
\end{proof}

Recall that there is an inclusion $\ICat_{\BB}\into\ICat_{\BBB}$ of very large $\BB$-categories that identifies $\ICat_{\BB}$ with the full subcategory of $\ICat_{\BBB}$ that is spanned by the small functors $\I{D}\to A$ over all $A\in\BB$~\cite[Example~2.7.6]{Martini2021a}. As a consequence of Lemma~\ref{lem:characterisationSmallCocartesianFibration}, we now obtain:
\begin{proposition}
	\label{prop:characterisationSmallCocartesianFibrations}
	For every large $\BB$-category $\I{C}$, the subpresheaf $\Cocart^{\bU}(-\times \I{C})\into\Cocart(-\times\I{C})$ that is spanned by the small cocartesian fibrations over $\I{C}$ is a sheaf on $\BB$ and hence defines an (a priori very large) $\BB$-category $\ICocart_{\I{C}}^{\bU}$. Moreover, restricting the straightening functor $\St_{\I{C}}$ to $\ICocart_{\I{C}}^{\bU}$ determines an equivalence $\ICocart_{\I{C}}^{\bU}\simeq \iFun{\I{C}}{\ICat_{\BB}}$.\qed
\end{proposition}

As a consequence of Proposition~\ref{prop:characterisationSmallCocartesianFibrations}, there is a small cocartesian fibration $\phi_{\BB}\colon\LaxUnder{(\ICat_{\BB})}{1}\to\ICat_{\BB}$ that arises as the unstraightening of the identity $\id\colon \ICat_{\BB}\simeq\ICat_{\BB}$ and that is referred to as the \emph{universal} cocartesian fibration.

\begin{remark}
	\label{rem:BCUniversalCocartesianFibration}
	Given $A\in\BB$, Remark~\ref{rem:BCCocart} implies that the base change $\pi_A^\ast(\phi_{\BB})$ of the universal cocartesian fibration $\phi_{\BB}\colon\LaxUnder{(\ICat_{\BB})}{1}\to\ICat_{\BB}$ in $\BB$ is equivalent to the universal cocartesian fibration $\phi_{\Over{\BB}{A}}\colon\LaxUnder{(\ICat_{\Over{\BB}{A}})}{1}\to\ICat_{\Over{\BB}{A}}$.
\end{remark}
\begin{remark}
	\label{rem:universalCocartesianFibration}
	On account of the naturality of straightening, if $\I{C}\to\ICat_{\BB}$ is a functor in $\Cat(\BBB)$, the associated small cocartesian fibration $\Un_{\I{C}}(f)\to\I{C}$ fits into a unique pullback square
	\begin{equation*}
	\begin{tikzcd}
	\Un_{\I{C}}(f)\arrow[d]\arrow[r] & \LaxUnder{(\ICat_{\BB})}{1}\arrow[d, "\phi_{\BB}"]\\
	\I{C}\arrow[r, "f"] & \ICat_{\BB}.
	\end{tikzcd}
	\end{equation*}
\end{remark}
The goal for the remainder of this section is to relate the cocartesian fibration $\Gamma(\phi_{\BB})$ that is obtained by taking global sections of the universal cocartesian fibration in $\BB$ with the universal cocartesian fibration in $\SS$.

\begin{lemma}
	\label{lem:YonedaEmbeddingGlobalSections}
	For any $\infty$-category $\CC$, transposing the Yoneda embedding $\CC\into \IPSh_{\Univ}(\CC)$ in $\Cat(\BB)$ across the adjunction $\const_\BB\dashv \Gamma$ yields the composition
	\begin{equation*}
	\CC\xhookrightarrow{h_{\CC}} \PSh_{\SS}(\CC)\xrightarrow{(\const_\BB)_\ast} \PSh_{\BB}(\CC).
	\end{equation*}
\end{lemma}
\begin{proof}
	Transposing the Yoneda embedding $\CC\into\IPSh_{\Univ}(\CC)$ across the adjunction $\const_{\BB}\dashv \Gamma$ yields the composition
	\begin{equation*}
	\CC\xrightarrow{\eta} \Gamma(\CC)\into \PSh_{\BB}(\CC)
	\end{equation*}
	in which $\eta$ is the adjunction unit of $\const_{\BB}\dashv\Gamma$ and the right map is given by taking global sections of the Yoneda embedding in $\Cat(\BB)$. By in turn transposing the above map across the adjunction $\CC^{\op}\times -\dashv \Fun(\CC^{\op},-)$ in $\CatSS$, one ends up with the functor
	\begin{equation*}
	\CC^{\op}\times\CC\xrightarrow{\eta} \Gamma(\CC^{\op}\times\CC)\xrightarrow{\Gamma (\map{\CC})}\BB.
	\end{equation*}
	On the other hand, the transpose of the composition $\CC\into\PSh_{\SS}(\CC)\to\PSh_{\BB}(\CC)$ yields
	\begin{equation*}
	\CC^{\op}\times\CC\xrightarrow{\map{\CC}} \SS\xrightarrow{\const_\BB}\BB,
	\end{equation*}
	so it suffices to show that these two functors are equivalent. By~\cite[Lemma~4.4.10]{Martini2021a} the functor $\map{\Gamma\CC}$ is equivalent to the composition $\Gamma\circ\Gamma(\map{\CC})\colon \Gamma\CC^{\op}\times\Gamma\CC\to \BB\to \SS$, hence the morphism $\map{\CC}\to \map{\Gamma\CC}\circ\eta$ that is induced by the action of $\eta$ on mapping $\infty$-groupoids determines a morphism
	$\map{\CC}\to \Gamma\circ\Gamma(\map{\CC})\circ\eta$
	which in turn transposes to a map
	\begin{equation*}
	\const_\BB\circ\map{\CC}\to\Gamma(\map{\CC})\circ\eta.
	\end{equation*}
	By the triangle identities, this is an equivalence.
\end{proof}

\begin{lemma}
	\label{lem:InternalExternalPresheavesFibrations}
	For any $\infty$-category $\CC$, there is a commutative square
	\begin{equation*}
	\begin{tikzcd}
	\Fun(\CC,\SS)\arrow[r, "(\const_{\BB})_\ast"]\arrow[d, "\simeq"] & \Fun(\CC,\BB)\arrow[d, "\simeq"]\\
	\LFib_{\SS}(\CC)\arrow[r, "\const_{\BB}"] & \LFib_{\BB}(\CC)
	\end{tikzcd}
	\end{equation*}
	in which the two vertical equivalences are given by the Grothendieck construction in $\SS$ and in $\BB$, respectively.
\end{lemma}
\begin{proof}
	By using that both $\Fun(-,\SS)$ and $\Fun(-,\BB)$ are sheaves of $\infty$-categories on $\CatS$, it suffices to show that we have a commutative square
	\begin{equation*}
		\begin{tikzcd}
		\Fun(\Delta^\bullet,\SS)\arrow[r, "(\const_{\BB})_\ast"]\arrow[d, "\simeq"] & \Fun(\Delta^\bullet,\BB)\arrow[d, "\simeq"]\\
		\LFib_{\SS}(\Delta^\bullet)\arrow[r, "\const_{\BB}"] & \LFib_{\BB}(\Delta^\bullet)
		\end{tikzcd}
	\end{equation*}
	of functors $\Delta^{\op}\to\CatS$. Recall from~\cite[\S~4.5]{Martini2021} that the Grothendieck construction fits into a commutative diagram
	\begin{equation*}
	\begin{tikzcd}
	\LFib_{\BB}(\Delta^\bullet)\arrow[r, hookrightarrow]\arrow[d, "\simeq"] & \Over{(\Simp\BB)}{\Delta^\bullet}\arrow[d,"\simeq"]\\
	\Fun(\Delta^\bullet,\BB)\arrow[r, hookrightarrow, "\epsilon^\ast"] & \PSh_{\BB}(\Over{\Delta}{\Delta^\bullet})
	\end{tikzcd}
	\end{equation*}
	in which $\epsilon\colon (\Over{\Delta}{\Delta^\bullet})^\op\to\Delta^\bullet$ carries a map $\tau\colon \ord{k}\to\ord{n}$ to $\tau(0)\in\ord{n}$. It is now straightforward to verify that the two horizontal maps and the equivalence on the right are natural in $\BB\in\LTop$, hence the claim follows.
\end{proof}

\begin{lemma}
	\label{lem:comparisonStraighteningInternalExternal}
	For any $\BB$-category $\I{C}$ there exists a commutative square
	\begin{equation*}
	\begin{tikzcd}
		\Cocart(\I{C})\arrow[r, hookrightarrow, "\Gamma(\St_{\I{C}})"] \arrow[d, "\Gamma"]& \Simp{\LFib(\I{C})}\arrow[d, "\Gamma"]\\
		\Cocart(\Gamma(\I{C}))\arrow[r, hookrightarrow, "\St_{\Gamma\I{C}}"] & \Simp{\LFib(\Gamma\I{C})}.
	\end{tikzcd}
	\end{equation*}
\end{lemma}
\begin{proof}
	By construction of the straightening functor and by making use of Remark~\ref{rem:CocartCTensoringPoweringSectionwise} and Lemma~\ref{lem:YonedaEmbeddingGlobalSections}, the functor $\Gamma(\St_{\I{C}})$ fits into the diagram
	\begin{equation*}
	\begin{tikzcd}
	\Gamma(\St_{\I{C}})(-)_\bullet\arrow[r, hookrightarrow]\arrow[dr] & (-)^{\Delta^{\bullet}}\arrow[d] \arrow[r] & \iFun{\Delta^{\bullet}}{-}\arrow[d]\\
	&\I{C}\arrow[r, "\diag"] & \iFun{\Delta^{\bullet}}{\I{C}}.
	\end{tikzcd}
	\end{equation*}
	in which the square is a pullback and the upper left horizontal map is given by the inclusion of the underlying left fibration. As $\Gamma$ commutes with limits, cotensoring by $\infty$-categories and taking the underlying left fibration of a cocartesian fibration, the claim follows.
\end{proof}

\begin{lemma}
	\label{lem:globalSectionsUniversalCosimplicialObject}
	The functor $\Fun_{\BB}(\ICat_{\BB},\Univ)\to\Fun(\Cat(\BB),\SS)$ that corresponds to the global sections functor $\Gamma\colon \LFib_{\BB}(\ICat_{\BB})\to\LFib_{\SS}(\Cat(\BB))$ via the Grothendieck construction carries the simplicial object in $\Fun_{\BB}(\ICat_{\BB},\Univ)$ that is given by $\map{\ICat_{\BB}}(\Delta^\bullet,-)$ to the simplicial object $\map{\Cat(\BB)}(\Delta^\bullet,-)$ in $\Fun(\Cat(\BB),\SS)$.
\end{lemma}
\begin{proof}
	Let $i\colon \ICat_{\BB}\into\IPSh_{\Univ}(\Delta)$ be the inclusion. The simplicial object $\map{\ICat_{\BB}}(\Delta^\bullet,-)$ is obtained as the image of the simplicial object $\map{\IPSh_{\Univ}(\Delta)}(\Delta^\bullet,-)$ in $\Fun_{\BB}(\IPSh_{\Univ}(\Delta),\Univ)$ along the functor
	\begin{equation*}
	i^\ast\colon \Fun_{\BB}(\IPSh_{\Univ}(\Delta),\Univ)\to\Fun_{\BB}(\ICat_{\BB},\Univ).
	\end{equation*}
	Analogously, the simplicial object $\map{\Cat(\BB)}(\Delta^\bullet,-)$ is the image of $\map{\Simp\BB}(\Delta^\bullet,-)$ along $\Gamma(i)^\ast$.
	As the global sections functor $\Gamma\colon \LFib_{\BB}(\I{C})\to\LFib_{\SS}(\Gamma\I{C})$ is natural in $\I{C}\in\Cat(\BBB)$, we may thus replace $\ICat_{\BB}$ by $\IPSh_{\Univ}(\Delta)$ and $\Cat(\BB)$ by $\Simp\BB$. Now the functor of left Kan extension $(h_{\Delta})_!\colon \iFun{\Delta}{\Univ}\into\iFun{\IPSh_{\Univ}(\Delta)}{\Univ}$ induces an inclusion $\LFib_{\BB}(\Delta)\into\LFib_{\BB}(\IPSh_{\Univ}(\Delta))$ that acts by sending a left fibration $\I{P}\to\Delta$ to the left fibration $\I{Q}\to\IPSh(\Delta)$ that arises from the factorisation of $\I{P}\to\Delta\into\IPSh_{\Univ}(\Delta)$ into an initial map and a left fibration (see~\cite[Corollary~3.3.3]{Martini2021a}). In particular, one obtains an initial map $\I{P}\to \I{Q}$. In the case that $\I{P}$ is corepresented by a global object in $\Delta$, the left fibration $\I{Q}\to\IPSh_{\Univ}(\Delta)$ is corepresented by its image along $h_{\Delta}$. Under these conditions, the global sections functor $\Gamma$ carries the initial map $\I{P}\to\I{Q}$ to an initial map in $\CatSS$. As a consequence, the lax square
	\begin{equation*}
	\begin{tikzcd}
	\LFib_{\BB}(\IPSh_{\Univ}(\Delta))\arrow[r, "\Gamma"] & \LFib_{\SS}(\Simp\BB)\\
	\LFib_{\BB}(\Delta)\arrow[r, "\Gamma"]\arrow[u, hookrightarrow, "(h_{\Delta})_!"] & \LFib_{\SS}(\Gamma\Delta)\arrow[u, hookrightarrow, "(\Gamma h_{\Delta})_!"]
	\end{tikzcd}
	\end{equation*}
	commutes after restricting along the inclusion $\Gamma(h_{\Delta^{\op}})\colon\Gamma\Delta^{\op}\into\LFib_{\BB}(\Delta)$. Now by virtue of Lemma~\ref{lem:YonedaEmbeddingGlobalSections}, the restriction of $\Gamma h_{\Delta}$ along the adjunction unit $\eta_{\Delta}\colon \Delta\to\Gamma\Delta$ is equivalent to the composition 
	\begin{equation*}
	 \Delta\xhookrightarrow{h_{\Delta}}\Simp\SS\xrightarrow{\const_\BB}\Simp\BB. 
	\end{equation*}
	Hence, the equivalence $\map{\Simp\BB}(\Delta^\bullet,-)\simeq\map{\Simp\SS}(\Delta^\bullet, \Gamma(-))$ implies that the simplicial object $\map{\Simp\BB}(\Delta^\bullet,-)$ arises as the image of $\eta_{\Delta}^{\op}\in\Simp{(\Gamma\Delta^{\op})}$ along the inclusion $(\Gamma h_{\Delta})_!\circ h_{\Gamma\Delta^{\op}}\colon \Simp{(\Gamma\Delta^{\op})}\into \Simp{\Fun(\Simp\BB,\SS)}$. To complete the proof, it therefore suffices to construct a commutative diagram
	\begin{equation*}
	\begin{tikzcd}
	\LFib_{\BB}(\Delta)\arrow[r, "\Gamma"] & \LFib_{\SS}(\Gamma\Delta)\\
	\Delta^{\op}\arrow[ur, "h_{\Gamma\Delta^{\op}}\circ\eta_{\Delta}^\op"']\arrow[u, "\Gamma (h_{\Delta^{\op}})\circ\eta_{\Delta}^\op"].
	\end{tikzcd}
	\end{equation*}
	Again by Lemma~\ref{lem:YonedaEmbeddingGlobalSections} and by moreover making use of Lemma~\ref{lem:InternalExternalPresheavesFibrations}, the map $\Gamma (h_{\Delta^{\op}})\circ\eta_{\Delta}^\op$ is equivalent to the composition 
	\begin{equation*}
	\Delta^{\op}\xhookrightarrow{h_{\Delta^{\op}}}\LFib_{\SS}(\Delta)\xrightarrow{	\const_\BB}\LFib_{\BB}(\Delta).
	\end{equation*}
	Note that the adjunction unit $\eta$ induces a map $\id_{\LFib_{\SS}(\Delta)}\to \eta_{\Delta}^\ast\circ\Gamma\circ\const_\BB$ in which the codomain denotes the composition
	\begin{equation*}
	\LFib_{\SS}(\Delta)\xrightarrow{\const_\BB} \LFib_{\BB}(\Delta)\xrightarrow{\Gamma} \LFib_{\SS}(\Gamma\Delta)\xrightarrow{\eta_{\Delta}^\ast} \LFib_{\SS}(\Delta).
	\end{equation*}
	By transposition, we therefore end up with a map $\phi\colon(\eta_{\Delta})_!\to \Gamma\circ\const_\BB$. On account of the equivalence $(\eta_{\Delta})_!\circ h_{\Delta^{\op}}\simeq h_{\Gamma\Delta^{\op}}\circ\eta$, it now suffices to show that $\phi h_{\Delta^{\op}}$ is an equivalence. But if $n\geq 0$ is an arbitrary integer, the map $\eta\colon\Under{\Delta}{\ord{n}}\to\Gamma\Under{\Delta}{\ord{n}}$ is already initial, which implies the claim.
\end{proof}

\begin{proposition}
	\label{prop:globalSectionsUniversalCocartesianFibration}
	There is a cartesian square
	\begin{equation*}
	\begin{tikzcd}
	\Gamma\LaxUnder{(\ICat_{\BB})}{1}\arrow[d, "\Gamma(\phi_{\BB})"]\arrow[r] & \LaxUnder{(\CatS)}{1}\arrow[d, "\phi_{\SS}"]\\
	\Gamma (\ICat_{\BB})\arrow[r, "\Gamma"] & \CatS
	\end{tikzcd}
	\end{equation*}
	of $\infty$-categories.
\end{proposition}
\begin{proof}
	Upon identifying $\Fun_{\BB}(\ICat_{\BB},\Univ)\simeq \LFib(
	\ICat_{\BB})$ and $\Fun(\Cat(\BB),\SS)\simeq\LFib(\Cat(\BB))$ via the Grothendieck construction, Lemma~\ref{lem:comparisonStraighteningInternalExternal} gives rise to a commutative square
	\begin{equation*}
	\begin{tikzcd}
			\Cocart(\ICat_{\BB})\arrow[r, hookrightarrow, "\Gamma(\St_{\ICat_{\BB}})"] \arrow[d, "\Gamma"]& \Simp{\Fun_{\BB}(\ICat_{\BB},\Univ)}\arrow[d]\\
			\Cocart(\Cat(\BB))\arrow[r, hookrightarrow, "\St_{\Cat(\BB)}"] & \Simp{\Fun(\Cat(\BB),\SS)}.
	\end{tikzcd}
	\end{equation*}
	The functor $\Gamma(\St_{\ICat_{\BB}})$ carries the universal cocartesian fibration to $\map{\ICat_{\BB}}(\Delta^{\bullet},-)$. By Lemma~\ref{lem:globalSectionsUniversalCosimplicialObject}, the right vertical map in the above diagram sends $\map{\ICat_{\BB}}(\Delta^{\bullet},-)$ to the simplicial object $\map{\Cat(\BB)}(\Delta^{\bullet},-)$, which is equivalent to $\map{\CatS}(\Delta^\bullet,\Gamma(-))$ by virtue of the adjunction $\const_\BB\dashv \Gamma$. Using the naturality of straightening (in $\CatS$), we conclude that $\Gamma(\phi_{\BB})$ is the pullback of $\phi_{\SS}$ along the global sections functor $\Gamma$, as claimed.
\end{proof}
\begin{corollary}
	\label{cor:StraighteningGlobalSections}
	Let $f\colon \I{C}\to\ICat_{\BB}$ be a functor and let $\I{P}\to\I{C}$ be the cocartesian fibration of $\BB$-categories that is classified by $f$. Then the cocartesian fibration $\Gamma\I{P}\to\Gamma\I{C}$ is classified by $\Gamma\circ\Gamma(f)\colon \Gamma\I{C}\to\CatS$.\qed
\end{corollary}

\subsection{Straightening over the interval}
\label{sec:straighteningInterval}
Let $p\colon \I{M}\to\Delta^1$ be a cocartesian fibration in $\Cat(\BB)$, and let $\I{M}\vert_0$ and $\I{M}\vert_1$ be its fibres over $d^1\colon\Delta^0\into\Delta^1$ and $d^0\colon\Delta^0\into\Delta^1$, respectively. Our goal in this section is to understand the functor $f\colon\I{M}\vert_0\to\I{M}\vert_1$ that arises from applying the straightening functor $\St_{\Delta^1}$ to $p$. Note that the inclusion $d^1\colon \I{M}\vert_0^\flat\into(\Delta^1)^\sharp\otimes\I{M}\vert_1^\flat$ being marked left anodyne implies that there exists a unique map $h\colon(\Delta^1)^\sharp\otimes\I{M}\vert_0^\flat\to\I{M}^\natural$ that makes the diagram
\begin{equation*}
\begin{tikzcd}
\I{M}\vert_0^\flat\arrow[r, hookrightarrow]  \arrow[d, "d_1", hookrightarrow]& \I{M}^\natural\arrow[d, "p^\natural"]\\
(\Delta^1)^\sharp\otimes\I{M}\vert_0^\flat\arrow[r, "\pr_0"]\arrow[ur, dashed, "h"]& (\Delta^1)^\sharp 
\end{tikzcd}
\end{equation*}
commute. Upon applying the restriction functor $(-)\vert_{\Delta}$ to this diagram, we therefore end up with a morphism $h\colon \Delta^1\otimes\I{M}\vert_0\to\I{M}$ in $\Cocart(\Delta^1)$ whose fibre over $d^1\colon\Delta^0\into\Delta^1$ recovers the identity on $\I{M}\vert_0$. Note that the cocartesian fibration $\pr_0\colon\Delta^1\otimes\I{M}\vert_0\to\Delta^1$ is the pullback of $\I{M}\vert_0\to 1$ along $s^0\colon\Delta^1\to 1$ and therefore corresponds via straightening to the identity on $\I{M}\vert_0$. Consequently, applying the functor $\St_{\Delta^1}$ to $h$ results in a commutative square
\begin{equation*}
\begin{tikzcd}
	\I{M}\vert_0\arrow[d, "\id"]\arrow[r, "\id"] & \I{M}\vert_0\arrow[d, "f"]\\
	\I{M}\vert_0\arrow[r, "g"] & \I{M}\vert_1
\end{tikzcd}
\end{equation*}
in $\Cat(\BB)$, which of course implies $f\simeq g$. In other words, we may recover $f$ as the fibre of $h$ over the final object $d^0\colon\Delta^0\into\Delta^1$. To proceed, we first need the following characterisation of cocartesian fibrations over $\Delta^1$:

\begin{proposition}
	\label{prop:cocartesianFibrationIntervalReflective}
	A functor $p\colon\I{M}\to\Delta^1$ in $\Cat(\BB)$ is a cocartesian fibration if and only if the inclusion $i_1\colon\I{M}\vert_{1}\into\I{M}$ of the fibre of $p$ over $d^0\colon\Delta^0\into\Delta^1$ admits a left adjoint $L_1$, in which case the adjunction unit $\eta\colon m\to i_1L_1 (m)$ is a cocartesian map for every object $m\colon A\to \I{M}$ in context $A\in\BB$.
\end{proposition}
The proof of Proposition~\ref{prop:cocartesianFibrationIntervalReflective} will make repeated use of the following observation:
\begin{lemma}
	\label{lem:cocartesianFibrationPoset}
	Let $\CC$ be a poset. Then a functor $p\colon\I{P}\to\CC$ in $\Cat(\BB)$ is a cocartesian fibration if and only if for every $c<d$ in $\CC$ and every object $x\colon A\to \I{P}\vert_c$ there exists a morphism $f\colon x\to y $ in $\I{P}$ such that $p(y)\simeq \pi_A^\ast(d)$ and such that the map
	\begin{equation*}
	f^\ast\colon \map{\I{P}}(y,i_{\geq d}(-))\to\map{\I{P}}(x,i_{\geq d}(-))
	\end{equation*}
	is an equivalence, where $i_{\geq d}\colon \I{P}_{\geq d}=\I{P}\times_{\CC}\Under{\CC}{d}\into\I{P}$ denotes the pullback of the inclusion $(\pi_d)_!\colon\Under{\CC}{d}\into\CC$ along $p$. If this is the case, the map $f$ is a cocartesian morphism.
\end{lemma}
\begin{proof}
	For every relation $c\leq d$ in the poset $\CC$, we shall denote by $(c\leq d)\colon 1\to\CC_1$ the associated morphism in the constant $\BB$-category. Note that $\CC$ being constant implies that $\CC_1$ admits a cover $\bigsqcup_{c\leq d} 1\onto \CC_1$, which implies that for every map $f\colon A\to \CC$ there is a cover $(s_i)\colon \bigsqcup_i A_i\onto A$ such that $s_i^\ast(f)\simeq\pi_{A_i}^\ast(c\leq d)$ for some relation $(c \leq d)$ in the poset $\CC$.
	 By combining this observation with Proposition~\ref{prop:characterisationCocartesianFibrationCocartesianMaps} and Remark~\ref{rem:CocartLiftIsLocal}, we thus conclude that $p$ is cocartesian if and only if for every $c\leq d$ and every object $x\colon A\to \I{P}\vert_c$ there exists a cocartesian lift $f\colon x\to y$ of $\pi_A^\ast(c\leq d)$. Since this is always possible when $c\simeq d$, we may assume $c<d$. Note that since $\CC$ is a poset, the map $(d_1,d_0)\colon\CC_1\to\CC_0\times\CC_0$ is a monomorphism in $\BB$. Therefore, a map $f\colon x\to y$ is a lift of $\pi_A^\ast(c< d)$ if and only if $p(y)\simeq \pi_A^\ast(d)$. Now in order to finish the proof, we only need to show that the map $f$ is cocartesian if and only if the morphism
	\begin{equation*}
	f^\ast\colon \map{\I{P}}(y,i_{\geq d}(-))\to\map{\I{P}}(x,i_{\geq d}(-))
	\end{equation*}
	is an equivalence. By replacing $\BB$ with $\Over{\BB}{A}$, we can assume that $A\simeq 1$. By definition, $f$ being cocartesian means that the commutative square
	\begin{equation*}
		\begin{tikzcd}
		\map{\I{P}}(y,-)\arrow[r, "f^\ast"]\arrow[d] & \map{\I{P}}(x,-)\arrow[d]\\
		\map{\CC}(d,p(-))\arrow[r, "(c< d)^\ast"] & \map{\CC}(c,p(-))
		\end{tikzcd}
	\end{equation*}
	is a pullback. Let $\CC_{\not\geq d}$ be the full subposet of $\CC$ that is spanned by the objects in $\CC$ that do not admit a map from $d$, and let us set $\I{P}_{\not\geq d}=\I{P}\times_{\CC}\CC_{\not\geq d}$. Then $\CC_0$ decomposes into a coproduct $(\Under{\CC}{d})_0\sqcup(\CC_{\not\geq d})_0$, which in turn induces a decomposition $\I{P}_0\simeq(\I{P}_{\geq d})_0\sqcup(\I{P}_{\not\geq d})_0$. As a consequence, the above square is cartesian if and only if its restriction along both $i_{\geq d}\colon \I{P}_{\geq d}\into\I{P}$ and $i_{\not\geq d}\colon\I{P}_{\not\geq d}\into\I{P}$ is cartesian. By construction of $\CC_{\not\geq d}$, the restriction of $\map{\CC}(d,-)$ along the inclusion $\CC_{\not\geq d}\into\CC$ yields the initial object. Consequently, restricting the above square along $i_{\not\geq d}$ trivially gives rise to a pullback diagram. On the other hand, the restriction of $(c<d)^\ast$ along the inclusion $\Under{\CC}{d}\into\CC$ produces an equivalence, which shows that the above square being cartesian is equivalent to the condition that $f^\ast\colon \map{\I{P}}(y,i_{\geq d}(-))\to\map{\I{P}}(x,i_{\geq d}(-))$ is an equivalence.
\end{proof}

\begin{proof}[{Proof of Proposition~\ref{prop:cocartesianFibrationIntervalReflective}}]
	Let us first assume that the inclusion $i_1\colon \I{M}\vert_1\into\I{M}$ admits a left adjoint $L_1$. Let $m\colon A\to\I{M}$ be an arbitrary object and let $\eta\colon m\to iL(m)$ be the adjunction unit. Then the map
	\begin{equation*}
		\begin{tikzcd}
		\map{\I{M}}(i_1L_1(m),i_1(-))\arrow[r, "\eta^\ast"] & \map{\I{M}}(m,i_1(-))
		\end{tikzcd}
	\end{equation*}
	is an equivalence. Hence Lemma~\ref{lem:cocartesianFibrationPoset} implies $p$ is a cocartesian fibration and that $\eta$ is a cocartesian morphism.
	
	Conversely, suppose that $p$ is a cocartesian fibration. Given $m\colon A\to \I{M}$ and $c=p(m)$, the fact that $1$ is a final object in $\Delta^1$ gives rise to a unique map $\alpha \colon c\to 1$ in context $A$. Let $f\colon m\to m^\prime$ be a cocartesian lift of $\alpha$. By construction, $m^\prime$ is contained in the essential image of $i_1$. We would like to show that the map
	\begin{equation*}
	\eta^\ast\colon\map{\I{M}}(m^\prime,i_1(-))\to\map{\I{M}}(m,i_1(-))
	\end{equation*}
	is an equivalence. But the map
	\begin{equation*}
		 \alpha^\ast\colon \map{\Delta^1}(p( m^\prime),pi_1(-))\to\map{\Delta^1}(p(m),pi_1(-))
	\end{equation*}
	 is an equivalence, hence $\eta^\ast$ is one as well on account of $\eta$ being cocartesian. This shows that $i_1$ admits a left adjoint that is given by sending $m$ to $m^\prime$.
\end{proof}
Let $\chi\colon i_0\to i_1 f$ be the morphism of functors $\I{M}\vert_0\to\I{M}$ that is encoded by the map $h\colon\Delta^1\otimes\I{M}\vert_0\to\I{M}$. Note that as for every object $m\colon A\to\I{M}\vert_0$ the associated map $(\pi_A^\ast(0<1),\id_m)$ in $\Delta^1\otimes\I{M}\vert_0$ is cocartesian, the map $\chi(m)\colon i_0(m)\to i_1 f(m)$ in $\I{M}$ is cocartesian too. 
By virtue of Proposition~\ref{prop:cocartesianFibrationIntervalReflective}, the left adjoint $L_1\colon\I{M}\to\I{M}\vert_1$ to $i_1$ thus carries $\chi(m)$ to an equivalence in $\I{M}\vert_1$. In other words, the map $L_1\chi$ is an equivalence of functors. But since $L_1 i_1\simeq \id_{\I{M}\vert_1}$ via the counit of the adjunction $L_1\dashv i_1$, we conclude:
\begin{proposition}
	\label{prop:StraighteningIntervalExplicitly}
	The functor $f\colon\I{M}\vert_0\to\I{M}\vert_1$ that classifies the cocartesian fibration $p\colon\I{M}\to \Delta^1$ is equivalent to the composition $L_1 i_0\colon\I{M}\vert_0\into\I{M}\to\I{M}\vert_1$.\qed
\end{proposition}
\begin{remark}
	\label{rem:StraighteningIntervalExplicitly}
	Proposition~\ref{prop:StraighteningIntervalExplicitly} in particular shows that for any object $m\colon A\to\I{M}\vert_0$ in context $A\in\BB$, the object $f(m)\colon A\to\I{M}\vert_1$ is given by the codomain of the unique cocartesian lift $f\colon m\to m^\prime$ of the map $\pi_A^\ast(0<1)$ in $\Delta^1$. More generally, if $F\colon\I{C}\to\ICat_{\BB}$ is an arbitrary functor and if $f\colon c\to d$ is a map in $\I{C}$ in context $A\in\BB$, the straightforward observation that a lift $h\colon x\to y$ in $\Un_{\I{C}}(F)$ of $f$ is cocartesian if and only if it defines a cocartesian lift of $0<1$ in the pullback $\Un_{\pi_A^\ast\I{C}}(\pi_A^\ast F)\times_{\pi_A^\ast\I{C}}\Delta^1\simeq\Un_{\Delta^1}(F(f))$ implies that the object $y\colon A\to \Un_{\I{C}}(F)\vert_d\simeq F(d)$ recovers the image of $x$ along $F(f)$.
\end{remark}

\begin{corollary}
	\label{cor:cartesianCocartesianFibrationInterval}
	A cocartesian fibration $p\colon \I{M}\to\Delta^1$ is cartesian if and only if the functor $f\colon\I{M}\vert_0\to\I{M}\vert_1$ admits a right adjoint $g\colon \I{M}\vert_1\to\I{M}\vert_0$. If this is the case, then $g^{\op}$ is classified by the cocartesian fibration $p^\op\colon\I{M}^\op\to (\Delta^1)^\op\simeq\Delta^1$.
\end{corollary}
\begin{proof}
	The dual of Proposition~\ref{prop:cocartesianFibrationIntervalReflective} implies that $p$ is a cartesian fibration if and only if the inclusion $i_0\colon\I{M}\vert_0\into\I{M}$ admits a right adjoint $R_0\colon\I{M}\to\I{M}\vert_0$. Hence Proposition~\ref{prop:StraighteningIntervalExplicitly} both shows that the functor $g=R_0i_1$ is right adjoint to $f\simeq L_1 i_0$ and that $g^\op$ arises as the straightening of $p^\op$. Conversely, suppose that $f$ has a right adjoint $g$. For any object $m\colon A\to\I{M}\vert_1$, we obtain a map $h\colon i_0 g(m)\to i_1(m)$ that is defined via the composition
	\begin{equation*}
	i_0g(m)\xrightarrow{\eta} i_1L_1 i_0g(m)\xrightarrow{\simeq} i_1 f g(m)\xrightarrow{\epsilon}i_1(m)
	\end{equation*}
	where $\eta$ is the unit of the adjunction $L_1\dashv i_1$ and $\epsilon$ is the counit of the adjunction $f\dashv g$. We claim the map
	\begin{equation*}
	h_\ast\colon \map{\I{M}}(i_0(-),i_0g(m))\to\map{\I{M}}(i_0(-),i_1(m))
	\end{equation*}
	is an equivalence. Unwinding the definitions, the composition of $h_\ast$ with the equivalence
	\begin{equation*}
	\map{\I{M}\vert_0}(-,g(m))\xrightarrow{\simeq} \map{\I{M}}(i_0(-),i_0g(m))
	\end{equation*}
	turns out to be equivalent to the composition
	\begin{equation*}
	\map{\I{M}\vert_0}(-,g(m))\xrightarrow{\simeq}\map{\I{M}\vert_1}(f(-), m)\xrightarrow{\simeq} \map{\I{M}\vert_1}(L_1 i_0(-), m)\xrightarrow{\simeq} \map{\I{M}}(i_0(-), i_1(m))
	\end{equation*}
	in which the two outer equivalences are determined by the two adjunctions $f\dashv g$ and $L_1\dashv i_1$.  Consequently, $h_\ast$ is an equivalence, hence the dual version of Lemma~\ref{lem:cocartesianFibrationPoset}  implies that $p$ is a cartesian fibration.
\end{proof}
\begin{remark}
	In large parts, our treatment of cocartesian fibrations over the interval is an adaptation of the discussion in~\cite[\href{https://kerodon.net/tag/02FJ}{\S~02FJ}]{kerodon} to $\BB$-categories.
\end{remark}

\section{Applications}
\label{sec:app}
Having established the straightening equivalence in \S~\ref{sec:straighteningEquivalence}, we will use this final chapter to briefly mention two applications: in \S~\ref{sec:applicationLimitsColimits} we give formulas for the limit and colimit of a diagram in $\ICat_{\BB}$, and in \S~\ref{sec:applicationFunctorialityAdjoints} we discuss how passing from a left adjoint functor between $\BB$-categories to its right adjoint can be turned into a functor.
\subsection{Internal limits and colimits of $\BB$-categories}
\label{sec:applicationLimitsColimits}
The straightening equivalence allows us to derive formulas for the limit and the colimit of a diagram of the form $d\colon\I{J}\to\ICat_{\BB}$. As the colimit functor $\colim\colon \iFun{\I{J}}{\ICat_{\BB}}\to\ICat_{\BB}$ is left adjoint to the diagonal functor, one can compute $\colim d$ as the value of the left adjoint of the pullback map $\pi_{\I{J}}^\ast\colon\ICat_{\BB}\to\ICocart_{\I{J}}$ at $\Un_{\I{J}}(d)$. Together with Remark~\ref{rem:LeftAdjointPullbackCocartPoint}, this shows:
\begin{proposition}
	\label{prop:formulaColimitsCatB}
	Let $d\colon \I{J}\to\ICat_{\BB}$ be a small diagram, and let $p\colon\I{P}\to\I{J}$ be the unstraightening of $d$. Then the $\BB$-category $\colim d$ is equivalent to the localisation $\I{P}_\sharp^{-1}\I{P}$ of $\I{P}$ at the subcategory $\I{P}_\sharp\into\I{P}$ that is spanned by the cocartesian maps.\qed
\end{proposition}
Dually, the limit functor $\lim\colon\iFun{\I{J}}{\ICat_{\BB}}\to\ICat_{\BB}$ is right adjoint to the diagonal functor and therefore corresponds via unstraightening to the right adjoint of $\pi_{\I{J}}^\ast\colon \ICat_{\BB}\to\ICocart_{\I{J}}$. Using Remark~\ref{rem:pushforwardCocartPoint}, this shows:
\begin{proposition}
	\label{prop:formulaLimitsCatB}
	Let $d\colon \I{J}\to\ICat_{\BB}$ be a small diagram, and let $p\colon\I{P}\to\I{J}$ be the unstraightening of $d$. Then the $\BB$-category $\lim d$ is equivalent to the $\BB$-category $(\Over{\iFun{\I{J}^\sharp}{\I{P}^\natural}}{\I{J}^\sharp})\vert_{\Delta}$ of cocartesian sections of $p$.\qed
\end{proposition}

\subsection{Functoriality of passing between right and left adjoints}
\label{sec:applicationFunctorialityAdjoints}
Let $R\into (\ICat_{\BB})_1$ be the subobject that is spanned by the right adjoint functors, and let $\ICat_{\BB}^R$ be the subcategory of $\ICat_{\BB}$ that is determined by $R$ (in the sense of the discussion in~\cite[\S~2.9]{Martini2021a}). Note that since the condition for a functor between $\BB$-categories to be a right adjoint is local~\cite[Remark~3.3.6]{Martini2021a}, a functor $f\colon\I{C}\to\I{D}$ between $\Over{\BB}{A}$-categories defines an object in $R$ if and only if $f$ is a right adjoint. Moreover, $R$ is closed under equivalences and composition in the sense of~\cite[Proposition~2.9.8]{Martini2021a}, hence the inclusion $\ICat_{\BB}^R\into\ICat_{\BB}$ induces an equivalence $(\ICat_{\BB}^R)_1\simeq R$. In particular, a functor between $\Over{\BB}{A}$-categories defines a map in $\ICat_{\BB}^R$ if and only if it is a right adjoint. We define the subcategory $\ICat_{\BB}^L\into\ICat_{\BB}$ that is spanned by the \emph{left} adjoints in an analogous fashion. Note that the equivalence $(-)^\op\colon \ICat_{\BB}\simeq\ICat_{\BB}$ restricts to an equivalence $\ICat_{\BB}^R\simeq\ICat_{\BB}^L$. Our goal in this section is to prove:
\begin{proposition}
	\label{prop:dualityRightLeftAdjoint}
	There is an equivalence $(\ICat_{\BB}^R)^\op\simeq\ICat_{\BB}^L$ that carries a right adjoint functor to its left adjoint.
\end{proposition}

The proof of Proposition~\ref{prop:dualityRightLeftAdjoint} requires the following lemma, whose analogue for cocartesian fibrations of $\infty$-categories appears as~\cite[\href{https://kerodon.net/tag/02FP}{Proposition~02FP}]{kerodon}.
\begin{lemma}
	\label{lem:cartesianCocartesianFibrationCriterion}
	A cocartesian fibration $p\colon\I{P}\to\I{C}$ in $\Cat(\BB)$ is a cartesian fibration if and only if for every morphism $f\colon\Delta^1\otimes A\to \I{C}$ the functor $\I{P}\vert_f\to\Delta^1\otimes A$ is a cartesian fibration.
\end{lemma}
\begin{proof}
	The condition is clearly necessary, so it suffices to prove the converse. Since both $\Cocart^+$ and $\Cart^+$ are sheaves on $\mSimp\BB$ and on account of the equivalence $\I{C}^\sharp\simeq\colim (\Delta^n\otimes A)^\sharp$ in $\mSimp\BB$, we may assume $\I{C}\simeq\Delta^n\otimes A$. Using Remark~\ref{rem:BCCocart} (and its dual), we can furthermore reduce to the case $A\simeq 1$. By the dual of Lemma~\ref{lem:cocartesianFibrationPoset}, we need to show that for any $k<l$ in $\Delta^n$ and any object $x\colon A\to\I{P}\vert_l$, there exists a map $f\colon y\to x$ in $\I{P}$ such that $p(y)\simeq k$ and such that the map
	\begin{equation*}
	f_\ast\colon \map{\I{P}}(i_{\leq k}(-),y)\to\map{\I{P}}(i_{\leq k}(-),x)
	\end{equation*}
	is an equivalence. By assumption, the pullback $\I{P}\vert_{k< l}\to\Delta^1$ of $p$ along $(k< l)\colon\Delta^1\into\Delta^n$ is a cartesian fibration. We can therefore choose a map $f\colon y\to x$ that defines a cartesian morphism in $\I{P}\vert_{k<l}$. It will be sufficient to show that for every object $z\colon A\to \I{P}_{\leq k}$, the morphism
	\begin{equation*}
	f_\ast\colon \map{\I{P}}(i_{\leq k}(z),y)\to\map{\I{P}}(i_{\leq k}(z),x)
	\end{equation*}
	is an equivalence in $\Over{\BB}{A}$. As $z$ is locally contained in $\I{P}\vert_j$ for some $j\leq k$, we can furthermore assume that $z$ is already contained in $\I{P}\vert_j$, i.e.\ that $p(z)\simeq j$ holds. Let $g\colon z\to w$ be a cocartesian morphism in $\I{P}$ such that $p(w)\simeq k$. We then obtain a commutative square
	\begin{equation*}
	\begin{tikzcd}
		\map{\I{P}}(i_{\leq k}(w),y)\arrow[r, "f_\ast"]\arrow[d, "g^\ast"] & \map{\I{P}}(i_{\leq k}(w),x)\arrow[d, "g^\ast"]\\
		\map{\I{P}}(i_{\leq k}(z),y)\arrow[r, "f_\ast"] & \map{\I{P}}(i_{\leq k}(z),x).
	\end{tikzcd}
	\end{equation*}
	By Lemma~\ref{lem:cocartesianFibrationPoset}, the two vertical maps are equivalences. On the other hand, since $w$ defines an object in $\I{P}\vert_{k<l}$, the dual of Lemma~\ref{lem:cocartesianFibrationPoset} shows that the upper horizontal map is an equivalence on account of $f$ being a cartesian morphism in $\I{P}\vert_{k<l}$. Hence the claim follows. 
\end{proof}
\begin{proof}[{Proof of Proposition~\ref{prop:dualityRightLeftAdjoint}}]
	By making use of the straightening equivalence, there is a chain of equivalences
	\begin{equation*}
		\ICat_{\BB}\simeq \iFun{\Delta^\bullet}{\ICat_{\BB}}^\core\simeq (\ICocart_{\Delta^\bullet})^\core
	\end{equation*}
	of simplicial objects in $\BBB$. Moreover, since $\ICat_{\BB}^L\into \ICat_{\BB}$ is a subcategory, a functor $\Delta^n\to \ICat_{\BB}$ factors through $\ICat_{\BB}^L$ if and only if its restriction along $\Delta^1\otimes A\to \Delta^n$ factors through $\ICat_{\BB}^L$ for every morphism $f\colon \Delta^1\otimes A\to\Delta^n$ (see~\cite[Proposition~2.9.3]{Martini2021a}). By combining this observation with Corollary~\ref{cor:cartesianCocartesianFibrationInterval}, one concludes that a cocartesian fibration $p\colon \I{P}\to\Delta^n$ arises as the unstraightening of a functor $\Delta^n\to\ICat_{\BB}^R$ if and only if for every map $f\colon\Delta^1\otimes A\to\Delta^n$ the functor $\I{P}\vert_f\to\Delta^1\otimes A$ is also a cartesian fibration. By Lemma~\ref{lem:cartesianCocartesianFibrationCriterion}, this is in turn equivalent to $p$ being a cartesian fibration itself. As the same is true locally, the equivalence $\iFun{\Delta^n}{\ICat_{\BB}}^\core\simeq (\ICocart_{\Delta^n})^\core$ identifies $\iFun{\Delta^n}{\ICat_{\BB}^L}^\core\into\iFun{\Delta^n}{\ICat_{\BB}}^\core$ with the subobject $(\ICocart_{\Delta^n}^{\Cart})^\core\into(\ICocart_{\Delta^n})^\core$ that is spanned by the cartesian and cocartesian fibrations. Since taking opposite $\BB$-categories determines an equivalence $(\ICocart_{\Delta^n}^{\Cart})^\core\simeq (\ICocart_{(\Delta^n)^\op}^{\Cart})^\core$ that is natural in $n$, we obtain equivalences
	\begin{equation*}
	\iFun{\Delta^\bullet}{\ICat_{\BB}^L}^\core\simeq \iFun{(\Delta^\bullet)^\op}{\ICat_{\BB}^L}^\core\simeq\iFun{\Delta^\bullet}{(\ICat_{\BB}^L)^\op}^\core
	\end{equation*}
	of simplicial objects in $\BBB$ and thus an equivalence $\ICat_{\BB}^L\simeq(\ICat_{\BB}^L)^{\op}$. The desired result now follows by composing this map with $(-)^\op\colon(\ICat_{\BB}^L)^\op\simeq(\ICat_{\BB}^R)^\op$.
\end{proof}

\appendix
\section{The proof of Lemma~\ref{lem:markedSimplexOrdinaryCategory}}
\label{sec:appA}
The goal of this section is to show that the  $\infty$-category $\Delta_+$ as defined in \S~\ref{sec:markedSimplicialObjects} is a $1$-category. To that end, recall that since $+\in\Delta_+$ is the only object that is not contained in the essential image of the inclusion $\iota\colon\Delta\into\Delta_+$, it suffices to show that the two functors $\map{\Delta_+}(+,-)$ and $\map{\Delta_+}(-,+)$ take values in sets. Furthermore, by making use of the adjunctions $\flat\dashv \iota$ and $\iota \dashv\sharp$, there are equivalences $\map{\Delta_+}(\iota \ord{n},+)\simeq \map{\Delta}(\ord{n},\ord{1})$ and $\map{\Delta_+}(+,\iota \ord{n})\simeq\map{\Delta}(\ord{0},\ord{n})$ for all $n\geq 0$. Consequently, we only need to show that $\map{\Delta_+}(+,+)$ is a set. We will do so by explicitly constructing a simplicial model of this $\infty$-groupoid using the approach via necklaces due to Dugger and Spivak~\cite{Dugger2011}.

\subsection{The Dugger-Spivak model for mapping $\infty$-groupoids}
A \emph{necklace} is defined to be a simplicial set $T$ of the form
\begin{equation*}
T=\Delta^{n_0}\vee\cdots\vee \Delta^{n_k}
\end{equation*}
with $n_i\geq 0$ and where in each wedge the final vertex of $\Delta^{n_i}$ has been glued to the initial vertex of $\Delta^{n_{i+1}}$. Note that in the case $n_i=1$ for all $i=0,\dots,k$, the above necklace is precisely the $k$-spine $I^k$. Every necklace is naturally bi-pointed by its initial and final vertex and will therefore be regarded as an object in the $1$-category $\Under{(\Simp\Set)}{\partial\Delta^1}$. We let $\Nec$ be the full subcategory of $\Under{(\Simp\Set)}{\partial\Delta^1}$ that is spanned by the necklaces. Now if $S$ is a simplicial set and if $s,t\in S$ are vertices, we denote by $S_{(s,t)}$ the associated bi-pointed simplicial set. The main input to our proof of Lemma~\ref{lem:markedSimplexOrdinaryCategory} is the following theorem:
\begin{theorem}[{\cite[Theorem~1.2]{Dugger2011}}]
	\label{thm:DuggerSpivak}
	Let $S$ be a simplicial set and let $S\to\CC$ be a fibrant replacement in the Joyal model structure on $\Simp\Set$. Given two vertices $s,t\in S$, the mapping $\infty$-groupoid $\map{\CC}(s,t)$ is equivalent to $(\Over{\Nec}{S_{(s,t)}})^\gp$.
\end{theorem}
Now let $K$ be the simplicial set that is defined by the pushout
\begin{equation*}
\begin{tikzcd}
\Delta^1\arrow[d, "\sigma_0", hookrightarrow]\arrow[r, hookrightarrow, "d_{\{0,2\}}"] & \Delta^2\arrow[d, hookrightarrow]\\
\Delta\arrow[r, hookrightarrow] & K
\end{tikzcd}
\end{equation*}
in $\Simp\Set$, where we implicitly identify the three $1$-categories with their associated nerves. We will again denote by $+$ the image of $\{1\}\in\Delta^2$ in K, and we will implicitly identify $\Delta$ with its image in $K$. Any fibrant replacement of $K$ in the Joyal model structure on $\Simp\Set$ will be a model for $\Delta_+$. Therefore, by making use of Theorem~\ref{thm:DuggerSpivak}, the $\infty$-groupoid $\map{\Delta_+}(+,+)$ is presented by the nerve of the $1$-category $\Over{\Nec}{K_{(++)}}$. Lemma~\ref{lem:markedSimplexOrdinaryCategory} will thus be an immediate consequence of the following proposition:
\begin{proposition}
	\label{prop:HomotopyTypeNecklaces}
	The $\infty$-groupoid $(\Over{\Nec}{K_{(++)}})^\gp$ is a set.
\end{proposition}
In order to prove Proposition~\ref{prop:HomotopyTypeNecklaces}, we need to understand the $1$-category $\Over{\Nec}{K_{(++)}}$ in more detail. This is the content of the next section.

\subsection{Necklaces in $K$}
 By construction, there is a unique non-degenerate edge $\alpha\colon\ord{1}\to +$ in $K$ with codomain $+$. Similarly, there is a unique non-degenerate edge $\beta\colon+\to \ord{0}$ in $K$ with domain $+$. Therefore, an arbitrary object $f\colon\Delta^{n_0}\vee\cdots\vee\Delta^{n_k}\to K$ in $\Over{\Nec}{K_{(+,+)}}$ satisfies exactly one of the two disjoint conditions:
\begin{enumerate}
	\item for all $0\leq i\leq k$, the $n_i$-simplex $\sigma_k\colon\Delta^{n_k}\to K$ factors through $+\colon\Delta^0\to K$;
	\item there are indices $0\leq l<r\leq k$ such that
	\begin{enumerate}
	\item for all $i<l$ and all $i>r$ the simplex $\sigma_i\colon \Delta^{n_i}\to K$ factors through $+\colon\Delta^0\to K$,
	\item the simplex $\sigma_{l}\colon \Delta^{n_l}\to K$ factors uniquely into a surjection $\Delta^{n_l}\onto \Delta^1$ followed by $\beta\colon\Delta^1\to K$,
	\item  the simplex $\sigma_{r}\colon \Delta^{n_r}\to K$ factors uniquely into a surjection $\Delta^{n_r}\onto\Delta^1$ followed by $\alpha\colon\Delta^1\to K$.
	\end{enumerate}
\end{enumerate}
We say that an object in $\Over{\Nec}{K_{(++)}}$ is \emph{degenerate} if it satisfies condition~(1), and \emph{non-degenerate} otherwise.
\begin{lemma}
	\label{lem:mapsDegenerateNondegenerateNecklaces}
	There are no maps between a degenerate and a non-degenerate object in $\Over{\Nec}{K_{(++)}}$.
\end{lemma}
\begin{proof}
	Let us fix a degenerate object $f\colon\Delta^{n_0}\vee\cdots\vee\Delta^{n_k}\to K$ and a non-degenerate object $g\colon\Delta^{n_0}\vee\cdots\vee\Delta^{n_l}\to K$. Note that there is always a unique map from $f$ to the degenerate object $+\colon\Delta^0\to K$. Therefore, if there were a map $g\to f$ in $\Over{\Nec}{K_{(+,+)}}$, we would in particular obtain a map $g\to +$, which would however imply that $g$ is degenerate. Conversely, note that if we set $n=\sum_{i=0}^k n_i$, the inclusion of the spine $I^{n_i}\into \Delta^{n_i}$ for all $i=0,\dots k$ induces an inclusion $I^n\into \Delta^{n_0}\vee\cdots\vee\Delta^{n_k}$ of bi-pointed simplicial sets that in turn gives rise to a degenerate object $f^\prime\colon I^n\to K$ in $\Over{\Nec}{K_{(+,+)}}$. Therefore, any map $f\to g$ in $\Over{\Nec}{K_{(+,+)}}$ restricts to a map $f^\prime\to g$, which is clearly not possible as this would imply that the image of $f^\prime$ in $K$ contains objects that are different from $+$.
\end{proof}
As a consequence of Lemma~\ref{lem:mapsDegenerateNondegenerateNecklaces}, there is a decomposition
\begin{equation*}
\Over{\Nec}{K_{(+,+)}}\simeq \Over{\Nec}{K_{(+,+)}}^{\text{deg}}\sqcup\Over{\Nec}{K_{(+,+)}}^{\text{nondeg}}
\end{equation*}
of $\Over{\Nec}{K_{(+,+)}}$ into its degenerate and non-degenerate part. Together with the fact that the groupoidification functor $(-)^\gp\colon\CatS\to \SS$ commutes with colimits, this implies that we may treat the degenerate and the non-degenerate part of $\Over{\Nec}{K_{(+,+)}}$ separately.

\subsection{Proof of Proposition~\ref{prop:HomotopyTypeNecklaces}}
The computation of the groupoidification of $\Over{\Nec}{K_{(+,+)}}^{\text{deg}}$ is easy: as observed in the proof of Lemma~\ref{lem:mapsDegenerateNondegenerateNecklaces}, this category has a final object $+\colon \Delta^0\to K$, which immediately implies:
\begin{lemma}
	\label{lem:homotopyTypeDeg}
	There is an equivalence $(\Over{\Nec}{K_{(+,+)}}^{\mathrm{deg}})^\gp\simeq 1$.\qed
\end{lemma}
In order to compute the groupoidification of $\Over{\Nec}{K_{(+,+)}}^{\text{nondeg}}$, on the other hand, we need one additional step. Recall that we denote by $\alpha\colon \Delta^1\to K$ the map that picks out the unique $1$-simplex $+\to\ord{0}$. We obtain an evident functor
\begin{equation*}
\alpha\vee -\colon \Over{\Nec}{K_{(0,+)}}\to \Over{\Nec}{K_{(+,+)}}^{\mathrm{nondeg}},\quad (f\colon T\to K)\mapsto (\alpha\vee f\colon\Delta^1\vee T\to K).
\end{equation*}
\begin{lemma}
	\label{lem:precompositionNecklacesFinal}
	The functor $\alpha\vee -$ is homotopy final.
\end{lemma}
\begin{proof}
	Let us fix an arbitrary object $f\colon \Delta^{n_0}\vee\cdots\vee\Delta^{n_k}\to K$ in $\Over{\Nec}{K_{(+,+)}}^{\mathrm{nondeg}}$. By Quillen's theorem A, it suffices to show that the category $\Under{(\Over{\Nec}{K_{(0,+)}})}{f}$ admits an initial object. Recall that we denote by $0\leq l$ the largest index such that for all $i<l$ the simplex $\sigma_i\colon \Delta^{n_i}\to K$ factors through $+\colon\Delta^0\to K$ and such that $\sigma_l$ factors into a surjection $\tau\colon\Delta^{n_l}\onto\Delta^1$ followed by $\beta\colon\Delta^1\to K$. We may therefore construct a map
	\begin{equation*}
	\Delta^{n_0}\vee\cdots\vee\Delta^{n_k}\to \Delta^{1}\vee\Delta^{n_{k_l+1}}\vee\cdots\vee\Delta^{n_k}
	\end{equation*}
	over $K$ that sends $\Delta^{n_i}$ to the initial object for all $i< k_l$, that sends $\Delta^{k_l}$ to $\Delta^1$ via the degeneracy map $\tau$ and that acts as the identity on the remaining summands. This map defines the desired initial object in $\Under{(\Over{\Nec}{K_{(0,+)}})}{f}$.
\end{proof}

\begin{proof}[{Proof of Proposition~\ref{prop:HomotopyTypeNecklaces}}]
	By Lemma~\ref{lem:homotopyTypeDeg}, the $\infty$-groupoid $(\Over{\Nec}{K_{(+,+)}}^{\mathrm{deg}})^\gp$ is a set. By Lemma~\ref{lem:precompositionNecklacesFinal}, the functor $\alpha\vee -$ induces an equivalence
	\begin{equation*}
	(\Over{\Nec}{K_{(+,+)}}^{\mathrm{nondeg}})^\gp\simeq(\Over{\Nec}{K_{(\ord{0},+)}})^\gp.
	\end{equation*}
	Since the right-hand side is equivalent to $\map{\Delta_+}(\ord{0},+)$ by Theorem~\ref{thm:DuggerSpivak}, this is a set as well.
\end{proof}

\bibliographystyle{halpha}
\bibliography{references.bib}

\end{document}